\documentclass{article}
\usepackage{amsmath, amssymb, amsthm, graphicx, cite}
\usepackage[arrow,curve,matrix,arc,2cell]{xy}
\UseAllTwocells
\DeclareFontFamily{U}{rsfs}{} \DeclareFontShape{U}{rsfs}{n}{it}{<->
rsfs10}{} \DeclareSymbolFont{mscr}{U}{rsfs}{n}{it}
\DeclareSymbolFontAlphabet{\scr}{mscr}
\def\mathscr{\scr}
\begin{document}
\def\e#1\e{\begin{equation}#1\end{equation}}
\def\ea#1\ea{\begin{align}#1\end{align}}
\def\eq#1{{\rm(\ref{#1})}}
\theoremstyle{plain}
\newtheorem{thm}{Theorem}[section]
\newtheorem{prop}[thm]{Proposition}
\newtheorem{lem}[thm]{Lemma}
\newtheorem{cor}[thm]{Corollary}
\newtheorem{quest}[thm]{Question}
\newtheorem{conj}[thm]{Conjecture}
\newtheorem{princ}[thm]{Principle}
\newtheorem{prob}[thm]{Problem}
\theoremstyle{definition}
\newtheorem{dfn}[thm]{Definition}
\newtheorem{ex}[thm]{Example}
\newtheorem{rem}[thm]{Remark}
\numberwithin{equation}{section}
\numberwithin{figure}{section}
\def\dim{\mathop{\rm dim}\nolimits}
\def\det{\mathop{\rm det}\nolimits}
\def\rank{\mathop{\rm rank}}
\def\Trace{\mathop{\rm Trace}}
\def\Crit{\mathop{\rm Crit}}
\def\id{\mathop{\rm id}\nolimits}
\def\GL{\mathop{\rm GL}}
\def\SO{\mathop{\rm SO}\nolimits}
\def\Ker{\mathop{\rm Ker}}
\def\Re{\mathop{\rm Re}}
\def\Im{\mathop{\rm Im}}
\def\Aut{\mathop{\rm Aut}}
\def\End{\mathop{\rm End}}
\def\SU{\mathop{\rm SU}}
\def\U{{\rm U}}
\def\Hess{\mathop{\rm Hess}}
\def\area{\mathop{\rm area}}
\def\Hom{\mathop{\rm Hom}\nolimits}
\def\vol{\mathop{\rm vol}\nolimits}
\def\an{{\rm an}}
\def\bs{\boldsymbol}
\def\ge{\geqslant}
\def\le{\leqslant\nobreak}
\def\O{{\mathbin{\cal O}}}
\def\cA{{\mathbin{\cal A}}}
\def\cB{{\mathbin{\cal B}}}
\def\cC{{\mathbin{\cal C}}}
\def\cD{{\mathbin{\scr D}}}
\def\cE{{\mathbin{\cal E}}}
\def\cF{{\mathbin{\cal F}}}
\def\cG{{\mathbin{\cal G}}}
\def\cH{{\mathbin{\cal H}}}
\def\cL{{\mathbin{\cal L}}}
\def\cM{{\mathbin{\cal M}}}
\def\oM{{\mathbin{\smash{\,\,\overline{\!\!\mathcal M\!}\,}}}}
\def\bM{{\mathbin{\kern .25em\ov{\kern -.25em M\kern -.1em}\kern .1em}}}
\def\cO{{\mathbin{\cal O}}}
\def\cP{{\mathbin{\cal P}}}
\def\cS{{\mathbin{\cal S}}}
\def\cT{{\mathbin{\cal T}}}
\def\cQ{{\mathbin{\cal Q}}}
\def\cW{{\mathbin{\cal W}}}
\def\C{{\mathbin{\mathbb C}}}
\def\CP{{\mathbin{\mathbb{CP}}}}
\def\K{{\mathbin{\mathbb K}}}
\def\Q{{\mathbin{\mathbb Q}}}
\def\R{{\mathbin{\mathbb R}}}
\def\N{{\mathbin{\mathbb N}}}
\def\Z{{\mathbin{\mathbb Z}}}
\def\sF{{\mathbin{\mathscr F}}}
\def\al{\alpha}
\def\be{\beta}
\def\ga{\gamma}
\def\de{\delta}
\def\io{\iota}
\def\ep{\epsilon}
\def\la{\lambda}
\def\ka{\kappa}
\def\th{\theta}
\def\ze{\zeta}
\def\up{\upsilon}
\def\vp{\varphi}
\def\si{\sigma}
\def\om{\omega}
\def\De{\Delta}
\def\La{\Lambda}
\def\Si{\Sigma}
\def\Th{\Theta}
\def\Om{\Omega}
\def\Ga{\Gamma}
\def\Up{\Upsilon}
\def\sSi{{\smash{\sst\Si}}}
\def\pd{\partial}
\def\db{{\bar\partial}}
\def\ts{\textstyle}
\def\st{\scriptstyle}
\def\sst{\scriptscriptstyle}
\def\w{\wedge}
\def\sm{\setminus}
\def\bu{\bullet}
\def\op{\oplus}
\def\ot{\otimes}
\def\ov{\overline}
\def\ul{\underline}
\def\bigop{\bigoplus}
\def\bigot{\bigotimes}
\def\iy{\infty}
\def\es{\emptyset}
\def\ra{\rightarrow}
\def\Ra{\Rightarrow}
\def\ab{\allowbreak}
\def\longra{\longrightarrow}
\def\hookra{\hookrightarrow}
\def\dashra{\dashrightarrow}
\def\t{\times}
\def\ci{\circ}
\def\ti{\tilde}
\def\d{{\rm d}}
\def\ha{{\ts\frac{1}{2}}}
\def\md#1{\vert #1 \vert}
\def\bmd#1{\big\vert #1 \big\vert}
\def\bms#1{\big\vert #1 \big\vert^2}
\def\ms#1{\vert #1 \vert^2}
\def\nm#1{\Vert #1 \Vert}
\def\bnm#1{\big\Vert #1 \big\Vert}
\title{Uniqueness results for special Lagrangians and Lagrangian
mean curvature flow expanders in $\C^m$}
\author{Yohsuke Imagi, Dominic Joyce, \\
and Joana Oliveira dos Santos}
\date{}
\maketitle

\begin{abstract} We prove two main results:
\smallskip

\noindent{\bf(a)} Suppose $L$ is a closed, embedded, exact special
Lagrangian $m$-fold in $\C^m$ asymptotic at infinity to the union
$\Pi_1\cup\Pi_2$ of two transverse special Lagrangian planes
$\Pi_1,\Pi_2$ in $\C^m$ for $m\ge 3$. Then $L$ is one of the explicit
`Lawlor neck' family of examples found by Lawlor~\cite{Lawl}.
\smallskip

\noindent{\bf(b)} Suppose $L$ is a closed, embedded, exact
Lagrangian mean curvature flow expander in $\C^m$ asymptotic at
infinity to the union $\Pi_1\cup\Pi_2$ of two transverse
Lagrangian planes $\Pi_1,\Pi_2$ in $\C^m$ for $m\ge 3$. Then $L$ is one
of the explicit family of examples in Joyce, Lee and
Tsui~\cite[Th.s C \& D]{JLT}.
\smallskip

If instead $L$ is immersed rather than embedded, the only extra possibility in {\bf(a)},{\bf(b)} is $L=\Pi_1\cup\Pi_2$.

Our methods, which are new and can probably be used to prove other
similar uniqueness theorems, involve $J$-holomorphic curves,
Lagrangian Floer cohomology, and Fukaya categories from symplectic
topology.
\end{abstract}

\baselineskip 11.5pt plus .5pt

\setcounter{tocdepth}{2}
\tableofcontents

\section{Introduction}
\label{la1}

The goal of this paper is to prove the following two theorems, which are
Theorems \ref{la4thm1} and \ref{la4thm2} below. All notation is defined in~\S\ref{la2}--\S\ref{la3}.
 
\begin{thm} Suppose $L$ is a closed, embedded, exact,
Asymptotically Conical special Lagrangian in\/ $\C^m$ for $m\ge 3,$
asymptotic at rate $\rho<0$ to a union $\Pi_1\cup\Pi_2$ of two
transversely intersecting special Lagrangian planes $\Pi_1,\Pi_2$ in
$\C^m$. Then $L$ is equivalent under an $\SU(m)$ rotation to one of
the `Lawlor necks' $L_{\bs\phi,A}$ found by Lawlor {\rm\cite{Lawl},}
and described in Example\/ {\rm\ref{la3ex1}} below.

If instead\/ $L$ is asymptotic to $\Pi_1\cup\Pi_2$ at rate $\rho<2,$
then $L$ is equivalent to some $L_{\bs\phi,A}$ under an $\SU(m)$
rotation and a translation in $\C^m$.

If instead\/ $L$ is immersed rather than embedded, the only extra possibilities are $L=\Pi_1\cup\Pi_2$ for $\rho<0,$ and\/ $L=\Pi_1\cup\Pi_2+\bs c$ for~$\rho<2$.
\label{la1thm1}
\end{thm}

\begin{thm} Suppose $L$ is a closed, embedded, exact,
Asymptotically Conical Lagrangian MCF expander in $\C^m$ for $m\ge
3,$ satisfying the expander equation $H=\al F^\perp$ for $\al>0,$
and asymptotic at rate $\rho<2$ to a union $\Pi_1\cup\Pi_2$ of two
transversely intersecting Lagrangian planes $\Pi_1,\Pi_2$ in $\C^m$.
Then $L$ is equivalent under a $\U(m)$ rotation to one of the LMCF
expanders $L_{\bs\phi}^\al$ found by Joyce, Lee and Tsui\/
{\rm\cite[Th.s C \& D]{JLT},} and described in Example\/
{\rm\ref{la3ex2}} below. If instead\/ $L$ is immersed rather than embedded, the only extra possibility is~$L=\Pi_1\cup\Pi_2$.
\label{la1thm2}
\end{thm}

\begin{rem} We restrict to complex dimensions $m\ge 3$, as our
proofs use results on the derived Fukaya category $D^b\sF(T^*\cS^m\#
T^*\cS^m)$ by Abouzaid and Smith \cite{AbSm}, which are only proved
when $m\ge 3$. Note however that the analogues of Theorems
\ref{la1thm1} and \ref{la1thm2} can be proved by other methods
when~$m=2$.

For Theorem \ref{la1thm1}, special Lagrangian 2-folds in $\C^2$ are
complex curves w.r.t.\ an alternative complex structure $K$ on
$\C^2$, so closed Asymptotically Conical special Lagrangian 2-folds
$L$ asymptotic to $\Pi_1\cup\Pi_2$ are simply conics $L$ in $\C^2$
whose closure $\bar L$ in $\CP^2\supset\C^2$ intersects the line at
infinity $\CP^1=\CP^2\sm\C^2$ in two given points~$\iy_1,\iy_2$.

Lotay and Neves \cite{LoNe} prove the analogue of Theorem
\ref{la1thm2} when $m=2$ by analytic means, without using Fukaya
categories, as we do. They also prove a `local' version of Theorem
\ref{la1thm2} for $m\ge 3$, that for fixed $\Pi_1,\Pi_2,\al$ any $L$
as in Theorem \ref{la1thm2} is rigid, that is, it has no small
deformations. Theorem \ref{la1thm2} answers the first part of Neves~\cite[Question~7.2]{Neve}.
\label{la1rem}
\end{rem}

The methods which we use to prove Theorems \ref{la1thm1} and
\ref{la1thm2} are new, are at least as interesting as the theorems
themselves, and can very probably also be used to prove uniqueness
of other classes of special Lagrangians and solitons for Lagrangian
MCF. They use deep results from symplectic geometry, to do with $J$-{\it holomorphic curves}, {\it Lagrangian Floer cohomology\/} and {\it Fukaya categories}. 

The proof of Theorem \ref{la1thm1} was motivated by a result of Thomas and Yau \cite[Th.~4.3]{ThYa}, who use Lagrangian Floer cohomology to prove that (under some assumptions) a compact special Lagrangian $L$ in a compact Calabi--Yau $m$-fold $M$ is unique in its Hamiltonian isotopy class. 

Also, the idea that Lagrangian Floer cohomology and Fukaya categories should be important in the study of special Lagrangians and Lagrangian MCF, and therefore that methods from symplectic topology should be used to prove new results on special Lagrangians and Lagrangian MCF such as Theorems \ref{la1thm1} and \ref{la1thm2}, was motivated by conjectures of the second author, explained in~\cite{Joyc8}.

We outline the method used to prove Theorems \ref{la1thm1} and
\ref{la1thm2}. First, here are three results proved in Theorem
\ref{la2thm1} and Propositions \ref{la4prop1} and \ref{la4prop2}
below:
\begin{itemize}
\setlength{\parsep}{0pt}
\setlength{\itemsep}{0pt}
\item[(i)] Suppose $(M,\om)$ is a symplectic Calabi--Yau Liouville
manifold of real dimension $2m$, and $L,L'$ are transversely intersecting, exact, graded Lagrangian submanifolds in $M$. Then each intersection point $p\in L\cap L'$ has a {\it
degree\/}~$\mu_{L,L'}(p)\in\Z$.

If $L,L'$ are compact they are objects in the derived Fukaya category
$D^b\sF(M)$. Suppose they are isomorphic in $D^b\sF(M)$. Then
for any generic almost complex structure $J$ on $M$ compatible with
$\om$, there exist points $p,q\in L\cap L'$ with
$\mu_{L,L'}(p)=0$ and $\mu_{L,L'}(q)=m$, and a $J$-holomorphic
disc $\Si$ in $M$ with boundary in $L\cup L'$ and corners at
$p,q$, of the form shown in Figure \ref{la1fig1}.
\begin{figure}[htb]
\centerline{$\splinetolerance{.8pt}
\begin{xy}
0;<1mm,0mm>:
,(-20,0);(20,0)**\crv{(0,10)}
?(.95)="a"
?(.85)="b"
?(.75)="c"
?(.65)="d"
?(.55)="e"
?(.45)="f"
?(.35)="g"
?(.25)="h"
?(.15)="i"
?(.05)="j"
?(.5)="y"
,(-20,0);(-30,-6)**\crv{(-30,-5)}
,(20,0);(30,-6)**\crv{(30,-5)}
,(-20,0);(20,0)**\crv{(0,-10)}
?(.95)="k"
?(.85)="l"
?(.75)="m"
?(.65)="n"
?(.55)="o"
?(.45)="p"
?(.35)="q"
?(.25)="r"
?(.15)="s"
?(.05)="t"
?(.5)="z"
,(-20,0);(-30,6)**\crv{(-30,5)}
,(20,0);(30,6)**\crv{(30,5)}
,"a";"k"**@{.}
,"b";"l"**@{.}
,"c";"m"**@{.}
,"d";"n"**@{.}
,"e";"o"**@{.}
,"f";"p"**@{.}
,"g";"q"**@{.}
,"h";"r"**@{.}
,"i";"s"**@{.}
,"j";"t"**@{.}
,"y"*{<}
,"z"*{>}
,(-20,0)*{\bu}
,(-20,-3)*{p}
,(20,0)*{\bu}
,(20,-3)*{q}
,(0,0)*{\Si}
,(-32,4)*{L}
,(-32,-4)*{L'}
,(32,4)*{L}
,(32.5,-4)*{L'}
\end{xy}$}
\caption{Holomorphic disc $\Si$ with boundary in $L\cup L'$}
\label{la1fig1}
\end{figure}
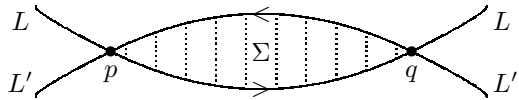
\item[(ii)] Let $M$ be a Calabi--Yau $m$-fold, and $L,L'$ be
transversely intersecting special Lagrangian $m$-folds in $M$.
Then $0<\mu_{L,L'}(p)<m$ for all~$p\in L\cap L'$.
\item[(iii)] Let $L,L'$ be transversely intersecting, graded LMCF
expanders in $\C^m$. Suppose $J$ is any almost complex structure
on $\C^m$ compatible with $\om$, and $\Si$ is a $J$-holomorphic
disc in $\C^m$ with boundary in $L\cup L'$ and corners at
$p,q\in L\cap L',$ as in Figure \ref{la1fig1}.
Then~$\mu_{L,L'}(q)-\mu_{L,L'}(p)<m$.
\end{itemize}

Suppose we have a class $\scr C$ of noncompact special Lagrangian
$m$-folds or LMCF expanders $L$ in $\C^m$ (for Theorems
\ref{la1thm1} and \ref{la1thm2}, $\scr C$ consists of $L$ which are
closed, embedded, exact, and asymptotic at infinity to the union
$\Pi_0\cup\Pi_{\bs\phi}$ of two fixed transverse Lagrangian planes
$\Pi_0,\Pi_{\bs\phi}$ in $\C^m$), and a family ${\scr
D}\subseteq{\scr C}$ of examples that we understand (for Theorems
\ref{la1thm1} and \ref{la1thm2}, $\scr D$ is the family of `Lawlor
necks' $L_{\bs\phi,A}$ \cite{Lawl}, or Joyce--Lee--Tsui expanders
$L_{\bs\phi}^\al$ \cite[Th.s C \& D]{JLT}).

Our goal is to prove that these are the only examples, that is,
${\scr C}={\scr D}$. We perform the following steps:
\begin{itemize}
\setlength{\parsep}{0pt}
\setlength{\itemsep}{0pt}
\item[(A)] Define some cohomological or analytic invariants
$I(L)$ of Lagrangians $L$ in $\scr C$ which distinguish
Lagrangians in $\scr D$. (For Theorems \ref{la1thm1} and
\ref{la1thm2} the graded Lagrangian $L$ has a potential
$f_L:L\ra\R$, and $I(L)\in\R$ is the difference of the limits of
$f_L$ at infinity in the two planes~$\Pi_0,\Pi_{\bs\phi}$).
\item[(B)] Construct a partial compactification $(M,\om)$ of
$\C^m$ which is a symplectic Calabi--Yau Liouville manifold,
such that each $L\in\scr C$ has a compactification $\bar L$ in
$M$ which is a compact, exact, graded Lagrangian. (For Theorems
\ref{la1thm1} and \ref{la1thm2}, $(M,\om)$ is the plumbing
$T^*\cS^m\# T^*\cS^m$.)
\item[(C)] Suppose for a contradiction that $L,L'\in\scr C$
with $I(L)=I(L')$, but $L\ne L'$. We have compactifications
$\bar L,\bar L'$ in $M$. We show that $\bar L\cong\bar L'$ in
$D^b\sF(M)$. (For Theorems \ref{la1thm1} and \ref{la1thm2}, we
use results on $D^b\sF(T^*\cS^m\# T^*\cS^m)$ due to Abouzaid and
Smith~\cite{AbSm}.)
\item[(D)] Let $\bar L''$ be any small Hamiltonian perturbation of $\bar L'$ which is $C^1$-close to $\bar L'$, intersects $\bar L$ transversely, and satisfies some mild asymptotic conditions near $\bar L'\sm L'$. We show that the intersection properties of transverse special Lagrangians in (ii) above, or of transverse LMCF expanders in (iii) above, also hold for the intersection of $\bar L$ and~$\bar L''$.
\item[(E)] We have $\bar L'\cong\bar L''$ in $D^b\sF(M)$ as
$\bar L',\bar L''$ are Hamiltonian isotopic, so $\bar L\cong\bar
L''$ in $D^b\sF(M)$. Therefore (i) gives a curve $\Si$
in $M$ with boundary in $\bar L\cup\bar L''$ and corners at
$p,q$ with $\mu_{\bar L,\bar L''}(p)=0$ and $\mu_{\bar L,\bar
L''}(q)=m$. But $\bar L,\bar L''$ satisfy (ii) or (iii) above by
(D), so this is a contradiction.
\item[(F)] Prove that $I(\scr C)=I(\scr D)$. (For Theorems
\ref{la1thm1} and \ref{la1thm2} we show that $I(L)>0$ for all
$L\in\scr C$.) Then (C)--(E) imply that ${\scr C}={\scr D}$.
\end{itemize}

We begin in \S\ref{la2} with the material on symplectic geometry we
need. Parts of \S\ref{la23}--\S\ref{la26} are new. Section
\ref{la3} discusses special Lagrangians, LMCF expanders, and the Lawlor and Joyce--Lee--Tsui examples in Theorems \ref{la1thm1} and \ref{la1thm2}. Parts of \S\ref{la32} and \S\ref{la35} are new. Section \ref{la4} states and proves our main results. Finally, Appendix \ref{laA} proves results on real analytic functions and Morse theory needed in~\S\ref{la43}--\S\ref{la44}.
\smallskip

\noindent{\it Historical note.} Joyce discussed this project with Imagi during the conference `Special Lagrangians and Related Topics' in Taiwan in July 2011, explaining an outline method to prove Theorem \ref{la1thm1}. In the subsequent two years, Imagi on the one hand, and Joyce and Oliveira dos Santos on the other, worked independently on the project, each unaware that the other was doing so.

In September 2013, Imagi sent Joyce a preprint with a complete proof of a weaker version of Theorem \ref{la1thm1}. At the time, Joyce and Oliveira dos Santos had an incomplete proof, but by a different method which also included Theorem \ref{la1thm2}. We then agreed to combine our projects into one joint paper, and finalized the arguments during a visit of Imagi to Oxford in February 2014.
\smallskip

\noindent{\it Acknowledgements.} The first author was supported by JSPS research fellowship 22-699, and did part of this project during a visit to the SCGP in Stony Brook. The second and third authors were supported by EPSRC grant EP/H035303/1.

The second author would like to thank Mark Haskins for suggesting the problem of proving uniqueness of `Lawlor necks', and Mohammed Abouzaid for explaining how to apply \cite{AbSm} in our situation. The authors would also like to thank Mohammed Abouzaid, Manabu Akaho, Lino Amorim, Kenji Fukaya, Mark Haskins, Kei Irie, Tsuyoshi Kato, Jason Lotay, Andr\'e Neves, Ivan Smith, and Yoshihiro Tonegawa for useful conversations, and to thank the referees.

\section{Symplectic geometry}
\label{la2}

We begin with the material on symplectic geometry we need, which
includes Lagrangian Floer cohomology and Fukaya categories of
compact, exact, graded Lagrangian submanifolds in symplectic
Calabi--Yau domains. Some good references are McDuff and Salamon
\cite{McSa} for elementary symplectic geometry, Seidel \cite{Seid2}
for the more advanced material of \S\ref{la25}, and Abouzaid and
Smith \cite{AbSm} for the example we study in \S\ref{la24} and
\S\ref{la26}. Parts of \S\ref{la23}--\S\ref{la26} are new.

\subsection{Exact symplectic manifolds and Liouville manifolds}
\label{la21}

We begin with some basic definitions in symplectic geometry.

\begin{dfn} A {\it symplectic manifold\/} $(M,\om)$ is a manifold
$M$ of dimension $2m$ with a closed 2-form $\om$ which is
nondegenerate, that is, $\om^m\ne 0$ on $M$.

A {\it Lagrangian\/} in $M$ is an $m$-dimensional submanifold $L$ in
$M$ with~$\om\vert_L=0$.

An {\it almost complex structure\/} $J$ on $M$ is a tensor $J\in
C^\iy(T^*M\ot TM)$ with $J^2=-\id$. We say that $J$ is {\it
compatible with\/} $\om$ if $\om(v,w)=\om(Jv,Jw)$ for all $v,w\in
C^\iy(TM)$, and $g\in C^\iy(S^2(T^*M))$ defined by
$g(v,w)=\om(v,Jw)$ is a positive definite Riemannian metric on $M$.
In index notation $\om=\om_{ab}$ and $J=J_a^b$, with
$J_a^bJ_b^c=-\de_a^c$, $\om_{ab}=J_a^cJ_b^d\om_{cd}$, and
$g_{ab}=J_b^c\om_{ac}$. Every symplectic manifold $(M,\om)$ admits
compatible complex structures $J$.

A $J$-{\it holomorphic curve\/} $\Si$ in $M$ is a Riemann surface
$(\Si,j)$ and a smooth map $u:\Si\ra M$ such that $J\ci\d u=\d u\ci
j:T\Si\ra u^*(TM)$. If $u$ is an embedding then we can think of the
$J$-holomorphic curve $\Si\cong u(\Si)$ as a submanifold of $M$. We
often take $\Si$ to have boundary or corners, and require that $u$
should map $\pd\Si$ to a Lagrangian $L$ or union of Lagrangians
$L_1\cup\cdots\cup L_k$ in~$M$.
\label{la2def1}
\end{dfn}

Next we define exact symplectic manifolds, exact Lagrangians, and
Liouville manifolds. These will be important to us as Lagrangian
Floer cohomology and Fukaya categories are well behaved for exact
Lagrangians in Liouville manifolds.

\begin{dfn} A symplectic manifold $(M,\om)$ is called {\it exact\/}
if $\om$ is exact, that is, $[\om]=0$ in $H^2(M,\R)$. (This is only
possible if $M$ is noncompact, or has boundary.) Then we may choose a 1-form $\la$ on
$M$ with $\d\la=\om$, called a {\it Liouville form}. We often
consider two Liouville forms $\la,\la'$ {\it equivalent\/} if
$\la'=\la+\d h$ for $h:M\ra\R$ a smooth, compactly-supported
function.

Suppose $L$ is a Lagrangian in an exact symplectic manifold
$(M,\om)$ with Liouville form $\la$. Then
$\d(\la\vert_L)=\om\vert_L=0$, so $\la\vert_L$ is a closed 1-form on
$L$. We call $L$ an {\it exact Lagrangian\/} if $\la\vert_L$ is
exact, that is, $[\la\vert_L]=0$ in $H^1(L,\R)$. If $H^1(M,\R)=0$,
this notion of exact Lagrangian is independent of the choice of
Liouville form $\la$ on $M$. If $L$ is an exact
Lagrangian in $(M,\om)$, there exists a smooth {\it potential\/}
$f_L:L\ra\R$ with $\la\vert_L=\d f_L$. If $L$ is connected, $f_L$ is
unique up to $f_L\mapsto f_L+c$ for~$c\in\R$.

If $(M,\om)$ is exact with Liouville form $\la$, there is a unique
vector field $v_\la$ on $M$ such that $v_\la\cdot\om=\la$ called the
{\it Liouville vector field}, and $\cL_{v_\la}\om=\om$, where
$\cL_{v_\la}$ is the Lie derivative. One can attempt to integrate
$v_\la$ to the {\it Liouville flow}, a smooth map $\Psi:M\t\R\ra M$
with $\Psi(p,0)=p$ and $\frac{\d\Psi}{\d t}(p,t)=v_\la
\vert_{(p,t)}$ for all $p\in M$ and $t\in\R$. This $\Psi$ need not
exist, but is unique if it does. If $\Psi$ exists we call
$(M,\om),\la$ a {\it Liouville manifold}.

A {\it Liouville domain\/} $(D,\om),\la$ is a compact symplectic
manifold $(D,\om)$ with boundary $\pd D$ with a Liouville form $\la$
whose Liouville vector field $v_\la$ is transverse to $\pd D$ and
outward-pointing. Then $\pd D$ is a contact manifold, with contact form $\la\vert_{\pd D}$. Every Liouville domain $(D,\om),\la$ has a natural {\it completion\/}
$(\hat D,\hat\om),\hat\la$ which is a Liouville manifold, where
$\hat D\cong D\cup_{\pd D}\bigl(\pd D\t[1,\iy)\bigr)$, and
$\hat\om\vert_D=\om$, $\hat\la\vert_D=\la$ on $D\subset\hat D$, and
$v_{\smash{\hat\la}}\vert_{\pd M\t[1,\iy)}=r\frac{\pd}{\pd r}$,
where $r$ is the coordinate on $[1,\iy)$. A Liouville manifold
$(M,\om),\la$ is called of {\it finite type\/} if it is isomorphic
to the completion of some Liouville domain  $(D,\om),\la$.
\label{la2def2}
\end{dfn}

\begin{ex} Let $\C^m$ have complex coordinates $(z_1,\dots,z_m)$
and its usual complex structure $J$, K\"ahler metric $g$, and
K\"ahler form $\om$ given by
\e
\begin{gathered}
J=\ts i\d z_1\ot\frac{\pd}{\pd z_1}+\cdots+i\d z_m\ot\frac{\pd}{\pd
z_m} -i\d\bar z_1\ot\frac{\pd}{\pd\bar z_1}-\cdots-i\d\bar
z_m\ot\frac{\pd}{\pd\bar z_m},\\
g=\ms{\d z_1}+\cdots+\ms{\d z_m}\;\>\text{and}\;\>
\om=\ts\frac{i}{2}(\d z_1\w\d\bar z_1+\cdots+\d z_m\w\d\bar z_m).
\end{gathered}
\label{la2eq1}
\e
Then $(\C^m,\om)$ is an exact symplectic manifold. Define a real
1-form $\la$ on $\C^m$ by
\e
\la=-\ha\Im(z_1\d\bar z_1+\cdots+z_m\d\bar z_m).
\label{la2eq2}
\e
Then $\d\la=\om$, so $\la$ is a Liouville form for $\C^m$, with
Liouville vector field
\begin{equation*}
v_\la=\ts\ha\bigl(z_1\frac{\pd}{\pd z_1}+\cdots+ z_m\frac{\pd}{\pd
z_m}+\bar z_1\frac{\pd}{\pd\bar z_1}+\cdots+\bar
z_m\frac{\pd}{\pd\bar z_m}\bigr).
\end{equation*}
In spherical polar coordinates with radius
$r=(\ms{z_1}+\cdots+\ms{z_m})^{1/2}$, we have $v_\la=\ha
r\frac{\pd}{\pd r}$. The Liouville flow of $v_\la$ exists, and is
given by
\begin{equation*}
\Psi\bigl((z_1,\ldots,z_m),t\bigr)=\bigl(e^{t/2}z_1,\ldots,e^{t/2}z_m\bigr).
\end{equation*}
Thus $(\C^m,\om),\la$ is a Liouville manifold. For any $R>0$, the
closed ball $\,\ov{\!B}_R$ of radius $R$ about 0 in $\C^m$ gives a Liouville
domain $\bigl(\,\ov{\!B}_R,\om\vert_{\,\ov{\!B}_R}\bigr),\la\vert_{\,\ov{\!B}_R}$, and
$(\C^m,\om),\la$ is its completion.
\label{la2ex1}
\end{ex}

\begin{ex} Suppose $(M,\om)$ is an exact symplectic manifold with
Liouville form $\la$, and $L,L'$ are transversely intersecting exact
Lagrangians in $M$ with potentials $f_L,f_{L'}$. Let $J$ be an
almost complex structure on $M$ compatible with $\om$. The theory of
Lagrangian Floer cohomology discussed in \S\ref{la25} involves
`counting' $J$-holomorphic curves in $M$ with boundary in $L\cup
L'$. We will show by an example how to compute the area of such
$J$-holomorphic curves.

Suppose $\Si$ is a $J$-holomorphic disc $\Si$ in $M$ with boundary
in $L\cup L'$ and corners at $p,q\in L\cap L'$, of the form shown in
Figure \ref{la1fig1}, where the arrows give the orientation of
$\pd\Si$. Then
\e
\begin{split}
\mathop{\rm area}(\Si)&=\int_\Si\om=\int_\Si\d\la=\int_{\pd\Si}\la=
\int_{\!\!\xymatrix@C=21pt{\mathop{\bu}\limits_p \ar[r]_{\text{in
$L$}} & \mathop{\bu}\limits_q }}
\!\!\!\!\!\!\!\!\!\!\!\!\!\!\!\!\!\!\la\,\,\,\,\,\,\,\,\,\,
+\int_{\!\!\xymatrix@C=21pt{\mathop{\bu}\limits_q \ar[r]_{\text{in
$L'$}} & \mathop{\bu}\limits_p }}
\!\!\!\!\!\!\!\!\!\!\!\!\!\!\!\!\!\!\la \\[-5pt]
&= \int_{\!\!\xymatrix@C=21pt{\mathop{\bu}\limits_p \ar[r]_{\text{in
$L$}} & \mathop{\bu}\limits_q }}
\!\!\!\!\!\!\!\!\!\!\!\!\!\!\!\!\!\!\d f_L\,\,\,\,
+\int_{\!\!\xymatrix@C=21pt{\mathop{\bu}\limits_q \ar[r]_{\text{in
$L'$}} & \mathop{\bu}\limits_p }}
\!\!\!\!\!\!\!\!\!\!\!\!\!\!\!\!\!\!\d f_{L'}\,\,\,\,
=f_L(q)\!-\!f_L(p)\!+\!f_{L'}(p)\!-\!f_{L'}(q),
\end{split}
\label{la2eq3}
\e
using $J$ compatible with $\om$ in the first step, $\om=\d\la$ in
the second, Stokes' Theorem in the third, and $\la\vert_L=\d f_L$,
$\la\vert_{L'}=\d f_{L'}$ in the fifth.
\label{la2ex2}
\end{ex}

\subsection[Symplectic Calabi--Yau manifolds and graded
Lagrangians]{Symplectic Calabi--Yaus and graded Lagrangians}
\label{la22}

We now define symplectic Calabi--Yau manifolds and graded
Lagrangians. These will be important to us because Lagrangian Floer
cohomology for graded Lagrangians in symplectic Calabi--Yau
manifolds can be graded over $\Z$ rather than $\Z_2$, and this
$\Z$-grading will be vital in our proofs.

Note that in \S\ref{la31} we will also define {\it Calabi--Yau
manifolds}, which involve the same data $M,\om,J,\Om$ but satisfy
stronger conditions. That is, Calabi--Yau manifolds give examples of
symplectic Calabi--Yau manifolds, but not vice versa.

\begin{dfn} A symplectic manifold $(M,\om)$ is called a {\it symplectic
Calabi--Yau manifold\/} if its first Chern class $c_1(TM)$ is zero
in $H^2(M,\Z)$.

Let $(M,\om)$ be symplectic Calabi--Yau of dimension $2m$, and
choose an almost complex structure $J$ compatible with $\om$, with
associated Hermitian metric $g$. Then we can define the {\it
canonical bundle\/} $K_M$ of complex $(m,0)$-forms on $M$ with respect to $J$. It is a complex line bundle, but not necessarily a holomorphic
line bundle, since $J$ may not be integrable.

Now $K_M$ is trivial as $c_1(K_M)=-c_1(TM)=0$, so we may pick a
nonvanishing section $\Om$ of $K_M$, which we require to satisfy the
normalization condition
\e
\om^m/m!=(-1)^{m(m-1)/2}(i/2)^m\Om\w\bar\Om.
\label{la2eq4}
\e

Let $L$ be a Lagrangian in $M$. Then $\Om\vert_L$ is a complex
$m$-form on $L$. The normalization \eq{la2eq4} of $\Om$ implies that
$\bmd{\Om\vert_L}=1$, where $\md{\,.\,}$ is computed using the
Riemannian metric $g\vert_L$. Suppose $L$ is oriented. Then we have
a volume form $\d V_L$ on $L$ defined using the metric $g\vert_L$
and orientation with $\md{\d V_L}=1$, so $\Om\vert_L=\Th_L\cdot \d
V_L$, where $\Th_L:L\ra\U(1)$ is a unique smooth function, the {\it angle function}, and
$\U(1)=\{z\in\C:\md{z}=1\}$.

This $\Th_L$ induces a morphism of cohomology groups
$\Th_L^*:H^1(\U(1),\Z)\ab\ra H^1(L,\Z)$. The {\it Maslov class\/}
$\mu_L\in H^1(L;\Z)$ of $L$ is the image under $\Th_L^*$ of the
generator of $H^1(\U(1),\Z)\cong\Z$. If $H^1(M,\R)=0$ then $\mu_L$
depends only on $(M,\om),L$ and not on $g,J,\Om$. We call $L$ {\it
Maslov zero\/} if $\mu_L=0$.

A {\it grading\/} or {\it phase function\/} of an oriented
Lagrangian $L$ is a smooth function $\th_L:L\ra\R$ with
$\Th_L=\exp(i\th_L)$, so that $\Om\vert_L=e^{i\th_L}\d V_L$. That
is, $i\th_L$ is a continuous choice of logarithm for $\Th_L$.
Gradings exist if and only if $L$ is Maslov zero. If $L$ is
connected then gradings are unique up to addition of $2\pi n$ for
$n\in\Z$. A {\it graded Lagrangian\/} $(L,\th_L)$ in $M$ is an
oriented Lagrangian $L$ with a grading $\th_L$. Usually we refer to
$L$ as the graded Lagrangian, leaving $\th_L$ implicit.
\label{la2def3}
\end{dfn}

\begin{ex} Let $\C^m,J,\om$ be as in Example \ref{la2ex1}, and
set $\Om=\d z_1\w\cdots\w\d z_m$. Then $(\C^m,\om)$ is symplectic
Calabi--Yau, and $J,\Om$ are as in Definition~\ref{la2def3}.
\label{la2ex3}
\end{ex}

Transverse intersection points $p$ of graded Lagrangians $L,L'$ have
a {\it degree\/} $\mu_{L,L'}(p)\in\Z$ used in the definition of
Lagrangian Floer cohomology.

\begin{dfn} Let $(M,\om)$ be a symplectic Calabi--Yau manifold of
dimension $2m$, and choose $J,\Om$ as in Definition \ref{la2def3}.
Suppose $L,L'$ are graded Lagrangians in $M$, with phase functions
$\th_L,\th_{L'}$, which intersect transversely at~$p\in M$.

By a kind of simultaneous diagonalization, we may choose an
isomorphism $T_pM\cong\C^m$ which identifies
$J\vert_p,g\vert_p,\om\vert_p$ on $T_pM$ with the standard versions
\eq{la2eq1} on $\C^m$, and identifies $T_pL,T_pL'$ with the
Lagrangian planes $\Pi_0,\Pi_{\bs\phi}$ in $\C^m$ respectively,
where
\e
\Pi_0=\bigl\{(x_1,\ldots,x_m):x_j\in\R\bigr\},\;\>
\Pi_{\bs\phi}=\bigl\{({\rm e}^{i\phi_1}x_1,\ldots, {\rm
e}^{i\phi_m}x_m):x_j\in\R\bigr\},
\label{la2eq5}
\e
for $\phi_1,\ldots,\phi_m\in(0,\pi)$. Then $\phi_1,\ldots,\phi_m$
are independent of choices up to order. Define the {\it degree\/}
$\mu_{L,L'}(p)\in\Z$ of $p$ by
\e
\mu_{L,L'}(p)=(\phi_1+\cdots+\phi_m+\th_L(p)-\th_{L'}(p))/\pi.
\label{la2eq6}
\e
This is an integer as $\th_{L'}(p)=\th_L(p)+\phi_1+\cdots+\phi_m\mod\pi\Z$ by \eq{la2eq5}. Exchanging $L,L'$ replaces $\phi_1,\ldots,\phi_m$ by $\pi-\phi_1,\ldots,\pi-\phi_m$, so that
\begin{equation*}
\mu_{L,L'}(p)+\mu_{L',L}(p)=m.
\end{equation*}
Since $\phi_1,\ldots,\phi_m\in(0,\pi)$, we see that
\e
(\th_L(p)-\th_{L'}(p))/\pi<\mu_{L,L'}(p)<(\th_L(p)-\th_{L'}(p))/\pi+m.
\label{la2eq7}
\e

\label{la2def4}
\end{dfn}

\begin{rem} In Definition \ref{la2def3} we first made arbitrary
choices of $J,\Om$ before defining graded Lagrangians. This is not
ideal from the point of view of symplectic geometry, since we prefer to work
only with concepts depending only on $(M,\om)$, and not on a choice
of almost complex structure~$J$.

In fact there is a more abstract, algebraic-topological way to
define graded Lagrangians and the degree $\mu_{L,L'}(p)$ without
choosing $J$, explained by Seidel \cite[\S 2a--\S 2d]{Seid1}. Define
$\mathop{\rm Lag}_+(M)\ra M$ to be the Grassmannian bundle of
oriented Lagrangian planes in $(TM,\om)$. Then choosing $\Om$ above
(up to isotopy) corresponds to choosing a certain $\Z$-cover
$\mathop{\,\,\widetilde{\!\!\rm Lag\!\!}\,\,}_+(M)\ra\mathop{\rm
Lag}_+(M)$. If $L$ is an oriented Lagrangian in $M$, then $TL$ gives
a natural section of the Grassmannian bundle $\mathop{\rm
Lag}_+(M)\vert_L$. A {\it grading\/} on $L$ is a lift of this to a
section of $\mathop{\,\,\widetilde{\!\!\rm
Lag\!\!}\,\,}_+(M)\vert_L$. One can also do a similar thing using
unoriented Lagrangians.

Thus, the notions of grading $\th_L$ and degree $\mu_{L,L'}(p)$ can
be made independent of the choices of $J,\Om$ (though if
$H^1(M,\Z)\ne 0$ the cover $\mathop{\,\,\widetilde{\!\!\rm
Lag\!\!}\,\,}_+(M)\ra\mathop{\rm Lag}_+(M)$ involves a discrete
choice corresponding to the isotopy class of $\Om$), and once we do
choose $J,\Om$ these notions become equivalent to the simpler ones
we gave in Definitions \ref{la2def3} and~\ref{la2def4}.
\label{la2rem1}
\end{rem}

\subsection{Asymptotically Conical Lagrangians in $\C^m$}
\label{la23}

\begin{dfn} A (singular) Lagrangian $m$-fold $C$ in $\C^m$ is called a {\it
cone} if $C=tC$ for all $t>0$, where $tC=\{t\,{\bf x}:{\bf x} \in
C\}$. Let $C$ be a closed Lagrangian cone in $\C^m$ with an isolated
singularity at 0. Then $\Si=C\cap{\cal S}^{2m-1}$ is a compact,
nonsingular Legendrian $(m\!-\!1)$-submanifold of ${\cal S}^{2m-1}$,
not necessarily connected. Let $g_\sSi$ be the metric on $\Si$ induced by
the metric $g$ on $\C^m$ in \eq{la2eq1}, and $r$ the radius function
on $\C^m$. Define $\iota:\Si\t(0,\iy)\ra\C^m$ by
$\iota(\si,r)=r\si$. Then the image of $\iota$ is $C\sm\{0\}$, and
$\iota^*(g)=r^2g_\sSi+\d r^2$ is the cone metric on~$C\sm\{0\}$.

Let $L$ be a closed, nonsingular Lagrangian $m$-fold in $\C^m$, e.g. $L$
could be special Lagrangian, or an LMCF expander. We call $L$
{\it Asymptotically Conical (AC)} with {\it rate} $\rho<2$ and {\it
cone} $C$ if there exist a compact subset $K\subset L$ and a
diffeomorphism $\vp:\Si\t(T,\iy)\ra L\sm K$ for some $T>0$, such
that
\e
\bmd{\nabla^k(\vp-\iota)}=O(r^{\rho-1-k}) \quad\text{as $r\ra\iy$, for all $k=0,1,2,\ldots.$}
\label{la2eq8}
\e
Here $\nabla,\md{\,.\,}$ are computed using the cone
metric~$\iota^*(g)$.
\label{la2def5}
\end{dfn}

The reason for supposing $\rho<2$ in Definition \ref{la2def5} is that then $tL\ra C$ as $t\ra 0$, as submanifolds in $\C^m$ away from 0 in the $C^k$ topology for all $k\ge 0$. That is, $C$ is the tangent cone to $L$ at infinity.

Now let $\phi_1,\ldots,\phi_m\in(0,\pi)$, write
$\bs\phi=(\phi_1,\ldots,\phi_m)$, and define Lagrangian planes
$\Pi_0,\Pi_{\bs\phi}$ in $\C^m$ by
\e
\Pi_0=\bigl\{(x_1,\ldots,x_m):x_j\in\R\bigr\},\;\>
\Pi_{\bs\phi}=\bigl\{({\rm e}^{i\phi_1}x_1,\ldots, {\rm
e}^{i\phi_m}x_m):x_j\in\R\bigr\}.
\label{la2eq9}
\e
Then $\Pi_0,\Pi_{\bs\phi}$ intersect transversely at 0 in $\C^m$, so
$C=\Pi_0\cup\Pi_{\bs\phi}$ is a Lagrangian cone in $\C^m$ with an
isolated singularity at 0. Most of this paper concerns exact,
connected AC Lagrangians $L$ in $\C^m$ with rate $\rho<0$ and cone
$C$. We will define an invariant $A(L)\in\R$ of such Lagrangians:

\begin{dfn} Regard $(\C^m,\om)$ as an exact symplectic manifold with
Liouville form $\la$ as in Example \ref{la2ex1}. Let
$C=\Pi_0\cup\Pi_{\bs\phi}$ be the Lagrangian cone in $\C^m$ above.
Suppose $L$ is an exact, connected, Asymptotically Conical
Lagrangian in $\C^m$ with rate $\rho<0$. Then we may choose smooth
$f_L:L\ra\R$ with $\d f_L=\la\vert_L$, unique up to $f_L\mapsto
f_L+c$. Now $\la\vert_C=0$ (for $\la$ as in \eq{la2eq2}, this is
true for any Lagrangian cone $C$, as the Liouville vector field
$v_\la$ is tangent to $C$), so \eq{la2eq8} implies that
$\bmd{\la\vert_L}=O(r^{\rho-1})$ as~$r\ra\iy$.

As $\rho<0$, we see that there exist unique constants
$c_0,c_{\bs\phi}\in\R$ such that $f_L=c_0+O(r^\rho)$ as $r\ra\iy$ in
the end of $L$ asymptotic to $\Pi_0$, and
$f_L=c_{\bs\phi}+O(r^\rho)$ as $r\ra\iy$ in the end of $L$
asymptotic to $\Pi_{\bs\phi}$. Define $A(L)=c_{\bs\phi}-c_0\in\R$.
This is independent of the choice of potential $f_L$ for
$\la\vert_L$, and is well-defined.
\label{la2def6}
\end{dfn}

\begin{rem} In \cite{Joyc2,Joyc3,Joyc4,Joyc5,Joyc6} we define AC
Lagrangians $L$ to satisfy \eq{la2eq8} for $k=0,1$ only, rather than
all $k\ge 0$. But for AC special Lagrangians \cite[Th.~7.7]{Joyc2}
and AC LMCF expanders \cite[Th.~3.1]{LoNe}, if \eq{la2eq8} holds for
$k=0,1$ then (possibly with different choices of $K,\vp,T$) it also
holds for all $k\ge 0$. In this paper we require \eq{la2eq8} to hold
for all $k\ge 0$, since in Example \ref{la2ex4} below we will see
that any AC Lagrangian $L$ with cone $C=\Pi_0\cup\Pi_{\bs\phi}$ and
rate $\rho<0$ extends to a smooth compact Lagrangian $\bar L$ in a
certain partial compactification $(M,\om)$ of $(\C^m,\om)$. If we
had required \eq{la2eq8} to hold only for $k=0,1$, then $\bar L$
would be only $C^1$ rather than smooth.
\label{la2rem2}
\end{rem}

\subsection{The Liouville manifold $T^*\cS^m\# T^*\cS^m$}
\label{la24}

Suppose $L,L'$ are two exact Asymptotically Conical Lagrangians in
$\C^m$ with cone $C=\Pi_0\cup\Pi_{\bs\phi}$, for
$\Pi_0,\Pi_{\bs\phi}$ as in \eq{la2eq9}. To prove our main results
in \S\ref{la4} we would like to work with the Lagrangian Floer
cohomology $HF^*(L,L')$, but this is not allowed (in the simplest
version of $HF^*$) as $L,L'$ are noncompact.

To deal with this, in the next (rather long) example, we will define
a partial compactification $M$ of $\C^m$ which extends
$(\C^m,\om),\la$ to a symplectic Calabi--Yau Liouville manifold
$(M,\om),\ti\la$, such that $L,L'$ extend to compact, exact
Lagrangians $\bar L,\bar L'$ in $M$. Thus as in \S\ref{la25} we can
form~$HF^*(\bar L,\bar L')$.

\begin{ex} Let $\phi_1,\ldots,\phi_m\in(0,\pi)$, and define
Lagrangian planes $\Pi_0,\Pi_{\bs\phi}$ in $\C^m$ by \eq{la2eq9}.
Define new real coordinates $(x_1,\ldots,x_m,y_1,\ldots,y_m)$ on
$\C^m$ by
\begin{equation*}
x_j=\Re z_j-\cot\phi_j\,\Im z_j, \quad y_j=\Im z_j,\quad
j=1,\ldots,m.
\end{equation*}
In these coordinates $\om=\d x_1\w\d y_1+\cdots+\d x_m\w d y_m$ and
\begin{align*}
\Pi_0&=\bigl\{(x_1,\ldots,x_m ,0,\ldots,0):x_j\in\R\bigr\},\\
\Pi_{\bs\phi}&=\bigl\{(0,\ldots,0,y_1,\ldots,y_m):y_j\in\R\bigr\}.
\end{align*}
Regard $(x_1,\ldots,x_m)$ as coordinates on $\Pi_0$ and
$(y_1,\ldots,y_m)$ as coordinates on $\Pi_{\bs\phi}$. Identify
$(\C^m,\om)\cong T^*\Pi_0\cong T^*\Pi_{\bs\phi}$ as symplectic
manifolds by identifying
\begin{align*}
(x_1,\ldots,x_m,y_1,\ldots,y_m)\in\C^m&\cong
y_1\d x_1+\cdots+y_m\d x_m\in
T^*_{(x_1,\ldots,x_m)}\Pi_0 \\
&\cong -x_1\d
y_1-\cdots-x_m\d y_m\in T^*_{(y_1,\ldots,y_m)}\Pi_{\bs\phi}.
\end{align*}

Write $\cS_0=\Pi_0\amalg\{\iy_0\}$ and
$\cS_{\bs\phi}=\Pi_{\bs\phi}\amalg\{\iy_{\bs\phi}\}$ for the
one-point compactifications of $\Pi_0\cong\R^m$ and
$\Pi_{\bs\phi}\cong\R^m$ by adding points $\iy_0$ at infinity in
$\Pi_0$ and $\iy_{\bs\phi}$ at infinity in $\Pi_{\bs\phi}$. Then
$\cS_0,\cS_{\bs\phi}$ are homeomorphic to the $m$-sphere $\cS^m$. We
make $\cS_0,\cS_{\bs\phi}$ into manifolds diffeomorphic to $\cS^m$ as
follows.

Define $F:(0,\iy)\ra(0,\iy)$ by $F(r)=1/\log(1+r^2)$. Define coordinates $(\ti x_1,\ldots,\ti x_m)$ on $\cS_0\sm\{0\}$ and $(\ti y_1,\ldots,\ti y_m)$ on $\cS_{\bs\phi}\sm\{0\}$ by
\ea
\ti x_j(p)&=\begin{cases} 0, & p=\iy_0, \\
F(r)x_j/r, & p=(x_1,\ldots,x_m)\in\Pi_0\sm\{0\}, \end{cases}
\label{la2eq10}\\
\ti y_j(q)&=\begin{cases} 0, & q=\iy_{\bs\phi}, \\
F(r)y_j/r, & q=(y_1,\ldots,y_m)\in\Pi_{\bs\phi}\sm\{0\}, \end{cases}
\label{la2eq11}
\ea
where $r=(x_1^2+\cdots+x_m^2)^{1/2}$ on $\Pi_0$ and
$r=(y_1^2+\cdots+y_m^2)^{1/2}$ on $\Pi_{\bs\phi}$. Then
$\cS_0,\cS_{\bs\phi}$ have unique smooth structures diffeomorphic to
$\cS^m$ such that $(x_1,\ldots,x_m),(\ti x_1,\ldots,\ti
x_m),(y_1,\ldots,y_m),(\ti y_1,\ldots,\ti y_m)$ are smooth
coordinates on the open sets $\Pi_0\subset\cS_0$,
$\cS_0\sm\{0\}\subset\cS_0$, $\Pi_{\bs\phi}\subset\cS_{\bs\phi}$,
$\cS_{\bs\phi}\sm\{0\}\subset\cS_{\bs\phi}$. We explain the reason for choosing $F(r)=1/\log(1+r^2)$ in Remark~\ref{la2rem3}.

The cotangent bundles $T^*\cS_0$, $T^*\cS_\phi$ are now symplectic
manifolds in the usual way, and as $\Pi_0\subset\cS_0$,
$\Pi_{\bs\phi}\subset\cS_{\bs\phi}$ are open we have open inclusions
$T^*\Pi_0\subset T^*\cS_0$, $T^*\Pi_{\bs\phi}\subset
T^*\cS_{\bs\phi}$, with $T^*\cS_0=T^*\Pi_0\amalg T^*_{\iy_0}\cS_0$,
$T^*\cS_{\bs\phi}=T^*\Pi_{\bs\phi}\amalg
T^*_{\iy_{\bs\phi}}\cS_{\bs\phi}$ on the level of sets. Define a
symplectic manifold $(M,\om)$ by identifying the symplectic
manifolds $T^*\cS_0,T^*\cS_{\bs\phi}$ on their open sets
$T^*\Pi_0\subset T^*\cS_0$, $T^*\Pi_{\bs\phi}\subset
T^*\cS_{\bs\phi}$ by the symplectic identification $T^*\Pi_0
\cong(\C^m,\om)\cong T^*\Pi_{\bs\phi}$ above. Then on the level of
sets we have
\begin{equation*}
M=\C^m\amalg T^*_{\iy_0}\cS_0\amalg T^*_{\iy_{\bs\phi}}\cS_{\bs\phi}
=T^*\cS_0\amalg T^*_{\iy_{\bs\phi}}\cS_{\bs\phi}=
T^*\cS_{\bs\phi}\amalg T^*_{\iy_0}\cS_0.
\end{equation*}

The Liouville form $\la$ on $\C^m$ in \eq{la2eq2} does not extend
continuously to $M$. Here is a way to modify it so that it does. Fix
$T\gg 0$, and let $\eta:\R\ra[-1,1]$ be smooth with $\eta(t)=-1$ for
$t\le -2T$, $\eta(t)=0$ for $-T\le t\le T$, $\eta(t)=1$ for $t\ge
2T$, and $\frac{\d\eta}{\d t}=O(T^{-1})$. Define a smooth 1-form
$\ti\la$ on $\C^m$ by
\ea
\ti\la&=\la+\d h,\;\>\text{where}\;\>
\la=\ha\bigl(x_1\d y_1+\cdots+x_m\d y_m-y_1\d x_1-\cdots-y_m\d x_m\bigr),
\nonumber\\
h&=-\ha\eta(x_1^2+\cdots+x_m^2-y_1^2-\cdots-y_m^2)(x_1y_1+\cdots+x_my_m).
\label{la2eq12}
\ea
Then $\d\ti\la=\om$, and
\begin{equation*}
\ti\la=\begin{cases}\la,
& -T\le \sum_{j=1}^mx_j^2-y_j^2\le T, \\
\sum_{j=1}^m-y_j\d x_j,
& \sum_{j=1}^mx_j^2-y_j^2\ge 2T, \\
\sum_{j=1}^mx_j\d y_j,
& \sum_{j=1}^mx_j^2-y_j^2\le -2T.\end{cases}
\end{equation*}
Here $\sum_{j=1}^m-y_j\d x_j$ is the natural Liouville form on $T^*\Pi_0$ and extends smoothly over $T^*_{\iy_0}\cS_0$ to a Liouville form on $T^*\cS_0$, and $\sum_{j=1}^mx_j\d y_j$ is the natural Liouville form on $T^*\Pi_{\bs\phi}$ and extends smoothly over $T^*_{\iy_{\bs\phi}}\cS_{\bs\phi}$ to $T^*\cS_{\bs\phi}$. Therefore $\ti\la$ extends to a smooth Liouville form on $M$, which we also write as $\ti\la$. One can check that the Liouville flow of $\ti\la$ exists, so $(M,\om),\ti\la$ is a Liouville manifold (of finite type). Note that $\ti\la\vert_{\cS_0}=\ti\la\vert_{\cS_{\bs\phi}}=0$, so $\cS_0,\cS_{\bs\phi}$ are exact Lagrangians, with potentials $f_{\cS_0}=f_{\cS_{\bs\phi}}=0$, and $T^*_{\iy_0}\cS_0,T^*_{\iy_{\bs\phi}}\cS_{\bs\phi}$ are also exact Lagrangians.

Similarly, $J,\Om$ on $\C^m$ do not extend smoothly to $M$, but we can define modified versions $\ti J,\ti \Om$ which do, so that $(M,\om)$ is symplectic Calabi--Yau. As for $\ti\la$, we can arrange that $\ti J=J$ and $\ti\Om=\Om$ when $-T\le \sum_{j=1}^mx_j^2-y_j^2\le T$ for some $T\gg 0$. The Lagrangians $\cS_0,\cS_{\bs\phi},T^*_{\iy_0}\cS_0,T^*_{\iy_{\bs\phi}}\cS_{\bs\phi}$ are all Maslov zero, so can be graded w.r.t.\ $\ti J,\ti\Om$. We choose $\ti J$ so that $\ti J(T_{\iy_0}\cS_0)=T_{\iy_0}(T^*_{\iy_0}\cS_0)$ and $\ti J(T_{\iy_{\bs\phi}}\cS_{\bs\phi})=T_{\iy_{\bs\phi}}(T^*_{\iy_{\bs\phi}}\cS_{\bs\phi})$, and $\ti\Om$ so that $\cS_0,\ab\cS_{\bs\phi},\ab T^*_{\iy_0}\cS_0,\ab T^*_{\iy_{\bs\phi}}\cS_{\bs\phi}$ have constant phase, with gradings
\begin{align*}
\th_{\cS_0}=0,\qquad \th_{\cS_{\bs\phi}}=\phi_1+\cdots+\phi_m,\qquad
\th_{T^*_{\iy_0}\cS_0}=m\pi/2,\\ \text{and}\qquad
\th_{T^*_{\iy_{\bs\phi}}\cS_{\bs\phi}}=\phi_1+\cdots+\phi_m+m\pi/2.
\end{align*}
Equation \eq{la2eq6} now gives indexes of intersection points
\begin{align*}
\mu_{\cS_0,\cS_{\bs\phi}}(0)&=0, & \mu_{\cS_{\bs\phi},\cS_0}(0)&=m, \\
\mu_{\cS_0,T^*_{\iy_0}\cS_0}(\iy_0)&=0, & \text{and}\qquad\qquad
\mu_{\cS_{\bs\phi},T^*_{\iy_{\bs\phi}}\cS_{\bs\phi}}(\iy_{\bs\phi})&=0.
\end{align*}

Now suppose that $L$ is a closed, Asymptotically Conical Lagrangian
in $\C^m$ with cone $C=\Pi_0\cup\Pi_{\bs\phi}$ and rate $\rho<0$, as
in \S\ref{la23}. Define $\bar L=L\amalg\{\iy_0,\iy_{\bs\phi}\}$. We claim that $\bar L$ is a smooth, compact Lagrangian in $(M,\om)$, which intersects $T^*_{\iy_0}\cS_0,T^*_{\iy_{\bs\phi}}\cS_{\bs\phi}$ transversely at $\iy_0,\iy_{\bs\phi}$. To see this, note that the end of $L$ asymptotic to $\Pi_0$ may be written near infinity in $\Pi_0$ as a graph
\begin{align*}
L\supset \Ga_{\d f}=\bigl\{\bigl(x_1+i{\ts\frac{\pd f}{\pd x_1}}(x_1,\ldots,x_m),
\ldots,x_m+i{\ts\frac{\pd f}{\pd x_m}}(x_1,\ldots,x_m)\bigr):&\\
(x_1,\ldots,x_m)\in\R^m,\;\> \md{(x_1,\ldots,x_m)}>R\bigr\}&
\end{align*}
for $R\gg 0$, where $f:\bigl\{(x_1,\ldots,x_m)\in\R^m:$ $\md{(x_1,\ldots,x_m)}>R\bigr\}\ra\R$ is smooth, and \eq{la2eq8} implies that for all $k_1,\ldots,k_m\ge 0$, as $R\ra\iy$ we have
\e
\frac{\pd^{k_1+\cdots+k_m}f(x_1,\ldots,x_m)}{\pd x_1^{k_1}\cdots\pd x_m^{k_m}}=O\bigl(r^{\rho-k_1-\cdots-k_m}\bigr).
\label{la2eq13}
\e

Define $\ti f:\bigl\{(\ti x_1,\ldots,\ti x_m)\in\R^m:$ $0<\md{(\ti x_1,\ldots,\ti x_m)}<F(R)\bigr\}\ra\R$ by $\ti f(\ti x_1,\ldots,\ti x_m)=f(x_1,\ldots,x_m)$, using the alternative coordinates $(\ti x_1,\ldots,\ti x_m)$ from \eq{la2eq10}. Then \eq{la2eq10}, \eq{la2eq13}, $F(r)=1/\log(1+r^2)$ and $\rho<0$ imply that 
\e
\frac{\pd^{k_1+\cdots+k_m}\ti f(\ti x_1,\ldots,\ti x_m)}{\pd\ti x_1^{k_1}\cdots\pd\ti x_m^{k_m}}=O\bigl(\ti r^{-k_1-\cdots-k_m}e^{\rho/2\ti r}\bigr)
\label{la2eq14}
\e
as $\ti r\ra 0$, where $\ti r=(\ti x_1^2+\cdots+\ti x_m^2)^{1/2}=F(r)$. As $\rho<0$, equation \eq{la2eq14} implies that $\ti f$ extends smoothly over 0 in $\bigl\{(\ti x_1,\ldots,\ti x_m)\in\R^m:$ $\md{(\ti x_1,\ldots,\ti x_m)}<F(R)\bigr\}$, with $\ti f(0,\ldots,0)=0$. Then $\bar L$ near $\iy_0$ is the graph of $\d\ti f$ in $T^*\cS_0$, and so is a smooth Lagrangian near $\iy_0$, which intersects $T^*_{\iy_0}\cS_0$ transversely at $\iy_0$. The same argument works near $\iy_{\bs\phi}$. Hence $\bar L$ is a smooth, compact Lagrangian in $(M,\om)$, intersecting $T^*_{\iy_0}\cS_0,T^*_{\iy_{\bs\phi}}\cS_{\bs\phi}$ transversely at~$\iy_0,\iy_{\bs\phi}$.

\begin{rem} The function $F(r)=1/\log(1+r^2)$ used in the coordinate changes \eq{la2eq10}--\eq{la2eq11} was chosen to ensure that, after changing from \eq{la2eq13} to \eq{la2eq14}, the r.h.s.\ of \eq{la2eq14} converges to zero as $\ti r\ra 0$ for all $k_1,\ldots,k_m\ge 0$. This requires that $F(r)\ra 0$ sufficiently slowly as $r\ra\iy$. For example, if we had chosen $F(r)=1/r$ then the r.h.s.\ of \eq{la2eq14} would have been $O\bigl(\ti r^{-\rho-k_1-\cdots-k_m}\bigr)$, so that we only know $\ti f$ is $C^k$ near $0$ for $0\le k<-\rho$, rather than $C^\iy$, and then $\bar L$ might be only $C^{k-1}$ near $\iy_0$, rather than smooth.

Note that if $L$ has rate $0\le\rho<2$ then it may not extend to a compact Lagrangian $\bar L$ in~$(M,\om)$.
\label{la2rem3}
\end{rem}

Suppose also that $L$ is exact and connected. Then by Definition
\ref{la2def6}, $L$ admits a unique potential $f_L$ with $\d
f_L=\la\vert_L$, such that $f_L\ra 0$ as $r\ra\iy$ in the end of $L$
asymptotic to $\Pi_0$, and $f_L\ra A(L)$ as $r\ra\iy$ in the end of
$L$ asymptotic to $\Pi_{\bs\phi}$. Since $\ti\la=\la+\d h$ for $h$
as in \eq{la2eq12} we have $\ti\la\vert_L=f_L+h\vert_L$. Now
\eq{la2eq8} and \eq{la2eq12} imply that $h\vert_L=O(r^\rho)$, so
$h\vert_L\ra 0$ as $r\ra\iy$ in $L$. It follows that $\bar L$ is
exact in $(M,\om),\ti\la$, with potential $f_{\bar L}$ with $\d
f_{\bar L}=\ti\la\vert_{\bar L}$ given by
\e
f_{\bar L}(p)=
\begin{cases} f_L(p)+h(p), & p\in L, \\ 0, & p=\iy_0, \\
A(L), & p=\iy_{\bs\phi}. \end{cases}
\label{la2eq15}
\e
In the same way, if $L$ is graded in $\C^m$ with grading $\th_L$,
then $\bar L$ is graded in $M$ with grading $\th_{\bar L}$, where
\begin{align*}
\th_{\bar L}(\iy_0)&=\lim_{\begin{subarray}{l}\text{$r\ra\iy$ in end
of $L$} \\ \text{asymptotic to $\Pi_0$}\end{subarray}}\th_L\,\,\,\,\,
\text{in $\pi\Z$,} \\
\th_{\bar L}(\iy_{\bs\phi})&=
\lim_{\begin{subarray}{l}\text{$r\ra\iy$ in end of $L$} \\
\text{asymptotic to $\Pi_{\bs\phi}$}\end{subarray}}
\th_L\,\,\,\,\,\text{in $\phi_1+\cdots+\phi_m+\pi\Z$.}
\end{align*}

The symplectic Calabi--Yau Liouville manifold $(M,\om),\ti\la$ defined above is well-known, and can be constructed in several different ways. Our construction is a version of the `plumbing' $T^*\cS^m\# T^*\cS^m$ of $T^*\cS^m$ with $T^*\cS^m$ as in Abouzaid and Smith \cite[Def.~2.1]{AbSm}, written to ensure that $\C^m\subset M$ and AC Lagrangians $L$ in $\C^m$ with cone $C=\Pi_0\cup\Pi_{\bs\phi}$ extend to compact Lagrangians $\bar L$ in $(M,\om)$. As in \cite[\S 1.2]{AbSm}, we may also describe $(M,\om)$ as the `$m$-dimensional $A_2$-Milnor fibre'
\begin{equation*}
A_2^m=\bigl\{(z_1,\ldots,z_m,t)\in\C^{m+1}:z_1^2+\cdots+
z_m^2+t^3=1\bigr\},
\end{equation*}
a symplectic submanifold of $(\C^{m+1},\om)$. The important point for us is that Abouzaid and Smith \cite{AbSm} give a classification of objects in the derived Fukaya category $D^b\sF(M)$, as we will explain in~\S\ref{la26}.

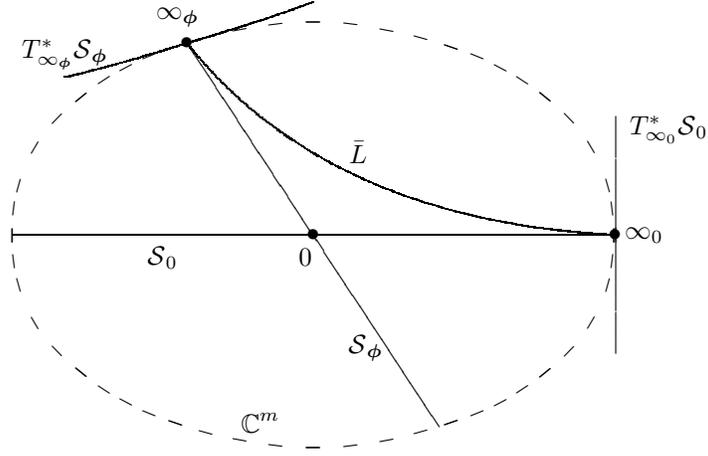
\begin{figure}[htb]
\centerline{$\splinetolerance{.8pt}
\begin{xy}
0;<1mm,0mm>:
,(-40,0);(40,0)**\crv{~**\dir{--}(-40,10)&(-30,20)&(-20,25)&(0,30)&(20,25)&(30,20)&(40,10)}
?(0.35)*{\bu}
?(.35)="a"
,(-18,29)*{\iy_{\bs\phi}}
,(-40,0);(40,0)**\crv{~**\dir{--}(-40,-10)&(-30,-20)&(-20,-25)&(0,-30)&(20,-25)&(30,-20)&(40,-10)}
?(.65)="b"
,(-40,0);(40,0)**\crv{~**\dir{-}}
,(40.15,0)*{\bu}
,(44,0)*{\iy_0}
,(0,0)*{\bu}
,(-1,-3)*{0}
,"a";"b"**\crv{~**\dir{-}}
,(40.3,0);(40.3,15)**\crv{}
,(40.3,0);(40.3,-15)**\crv{}
,"a";(0,31)**\crv{(-5,29)}
,"a";(-33,21)**\crv{(-28,22)}
,"a";(40,0)**\crv{(2,2)}
,(-20,-3)*{\cS_0}
,(47.3,14)*{T^*_{\iy_0}\cS_0}
,(-33,24)*{T^*_{\iy_{\bs\phi}}\cS_{\bs\phi}}
,(7,-15)*{\cS_{\bs\phi}}
,(-7,-25)*{\C^m}
,(6,11)*{\bar L}
\end{xy}$}
\caption{The Liouville manifold $M=T^*\cS^m\# T^*\cS^m$}
\label{la2fig1}
\end{figure}

Figure \ref{la2fig1} is a diagram of our construction. The inside of the dotted ellipse is $\C^m$, and outside are the added Lagrangians at infinity $T^*_{\iy_0}\cS_0,T^*_{\iy_{\bs\phi}}\cS_{\bs\phi}$. We also sketch five Lagrangians in $M$: $\cS_0,\cS_{\bs\phi}$, which are compact spheres $\cS^m$, and $T^*_{\iy_0}\cS_0,T^*_{\iy_{\bs\phi}}\cS_{\bs\phi}$, which are noncompact $\R^m$'s, and the compactification $\bar L$ of an AC Lagrangian $L$ in $\C^m$ with cone $C=\Pi_0\cup\Pi_{\bs\phi}$ and rate $\rho<0$, and three special points $0,\iy_0,\iy_{\bs\phi}$, where various pairs of the five Lagrangians intersect transversely. This concludes Example~\ref{la2ex4}.
\label{la2ex4}
\end{ex}

\subsection{Lagrangian Floer cohomology and Fukaya categories}
\label{la25}

We now explain a little about the version of Lagrangian Floer cohomology and Fukaya categories that we will use, due to Seidel \cite{Seid2}. Our methods should work in other versions as well, since we only use general properties, rather than the details of the machinery. Some references on different versions are Floer \cite{Floe}, Fukaya \cite{Fuka2}, and Fukaya, Oh, Ohta and Ono~\cite{FOOO}.

Let $(M,\om),\la$ be a symplectic Calabi--Yau Liouville manifold of finite type. The choice of Liouville form $\la$ only matters up to $\la\mapsto\la+\d f$ for $f:M\ra\R$ a compactly-supported smooth function. For simplicity, fix an almost complex structure $J$ on $M$ compatible with $\om$, which should satisfy a certain convexity condition at infinity in $M$ (a compatibility condition with $\la$) which prevents families of $J$-holomorphic curves with boundaries in compact Lagrangians in $M$ from escaping to infinity. Almost complex structures satisfying this condition exist on any Liouville manifold. We take $J$ to be generic under this assumption.

As $M$ is symplectic Calabi--Yau, as in \S\ref{la22} we can then choose $\Om\in C^\iy(K_M)$ satisfying \eq{la2eq4}, so that {\it graded\/} Lagrangians make sense in $M$. Note however that as in Remark \ref{la2rem1}, we can define graded Lagrangians using covering spaces without choosing $J,\Om$, so it will make sense later to say that the Lagrangian Floer cohomology and Fukaya categories of graded Lagrangians we discuss are independent of the choices of~$J,\Om$.

For simplicity we will work with Lagrangian Floer cohomology and Fukaya categories over the coefficient ring $\La=\Z_2$. Consider compact, exact, graded Lagrangians $L$ in $M$ as in \S\ref{la21}--\S\ref{la22}, with potential $f_L:L\ra\R$ with $\d f_L=\la\vert_L$ and phase function $\th_L:L\ra\R$ with $\Om\vert_L=e^{i\th_L}\d V_L$. For two such Lagrangians $L,L'$,
as in \cite{Seid2} one defines a complex of free $\Z_2$-modules
$\bigl(CF^*(L,L'),\d\bigr)$ called the {\it Floer complex}, whose
cohomology $HF^*(L,L')$ as a graded $\Z_2$-module is the {\it
Lagrangian Floer cohomology\/} of $L,L'$. Here are some features of
the theory:
\begin{itemize}
\setlength{\parsep}{0pt}
\setlength{\itemsep}{0pt}
\item[(a)] The differential $\d$ in $\bigl(CF^*(L,L'),\d\bigr)$
depends on the choice of almost complex structure $J$. But $HF^*(L,L')$ is
independent of $J$ up to canonical isomorphism, and depends only
on $(M,\om),L$ and $L'$.
\item[(b)] If $L''$ is another Lagrangian Hamiltonian isotopic
to $L'$ then there is a canonical isomorphism $HF^*(L,L')\cong
HF^*(L,L'')$, although the complexes $\bigl(CF^*(L,L'),\d\bigr)$
and $\bigl(CF^*(L,L''),\d\bigr)$ may be very different.

Thus $HF^*(L,L')$ is an invariant of (compact, exact, graded)
Lagrangians up to Hamiltonian isotopy, which is interesting for
symplectic geometers.

\item[(c)] It is easiest to define $\bigl(CF^*(L,L'),\d\bigr)$
if $L,L'$ intersect transversely in $M$. By (b), if $L,L'$ are
not transverse, we can take a Hamiltonian perturbation $L''$ of
$L'$ such that $L,L''$ are transverse, and then compute
$HF^*(L,L')$ using~$\bigl(CF^*(L,L''),\d\bigr)$.
\item[(d)] If $L,L'$ intersect transversely, then $CF^*(L,L')$
is the free $\Z_2$-module with basis the points $p\in L\cap L'$,
graded by degree $\mu_{L,L'}(p)$ from Definition \ref{la2def4}.
Also $\d:CF^k(L,L')\ra CF^{k+1}(L,L')$ is of the form
\e
\d p=\ts\sum_{q\in L\cap L':\mu_{L,L'}(q)=k+1}N_{p,q}\cdot q,
\label{la2eq16}
\e
for $p\in L\cap L'$ with $\mu_{L,L'}(p)=k$, where (if $J$ is
suitably generic) $N_{p,q}\in\Z_2$ is the number modulo 2 of isomorphism classes of $J$-holomorphic discs $\Si$ in $M$ with boundary in $L\cup L'$ and corners at $p,q$, of the form shown in Figure \ref{la2fig2}, where the
arrows give the orientation of $\pd\Si$.
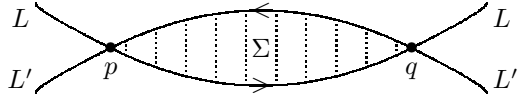
\begin{figure}[htb]
\centerline{$\splinetolerance{.8pt}
\begin{xy}
0;<1mm,0mm>:
,(-20,0);(20,0)**\crv{(0,10)}
?(.95)="a"
?(.85)="b"
?(.75)="c"
?(.65)="d"
?(.55)="e"
?(.45)="f"
?(.35)="g"
?(.25)="h"
?(.15)="i"
?(.05)="j"
?(.5)="y"
,(-20,0);(-30,-6)**\crv{(-30,-5)}
,(20,0);(30,-6)**\crv{(30,-5)}
,(-20,0);(20,0)**\crv{(0,-10)}
?(.95)="k"
?(.85)="l"
?(.75)="m"
?(.65)="n"
?(.55)="o"
?(.45)="p"
?(.35)="q"
?(.25)="r"
?(.15)="s"
?(.05)="t"
?(.5)="z"
,(-20,0);(-30,6)**\crv{(-30,5)}
,(20,0);(30,6)**\crv{(30,5)}
,"a";"k"**@{.}
,"b";"l"**@{.}
,"c";"m"**@{.}
,"d";"n"**@{.}
,"e";"o"**@{.}
,"f";"p"**@{.}
,"g";"q"**@{.}
,"h";"r"**@{.}
,"i";"s"**@{.}
,"j";"t"**@{.}
,"y"*{<}
,"z"*{>}
,(-20,0)*{\bu}
,(-20,-3)*{p}
,(20,0)*{\bu}
,(20,-3)*{q}
,(0,0)*{\Si}
,(-32,4)*{L}
,(-32,-4)*{L'}
,(32,4)*{L}
,(32.5,-4)*{L'}
\end{xy}$}
\caption{Holomorphic disc $\Si$ with boundary in $L\cup L'$}
\label{la2fig2}
\end{figure}

That is, we count $J$-holomorphic maps $u:\R\t[0,1]\ra M$ with $u(s,0)\in L$, $u(s,1)\in L'$ and $\lim_{s\ra -\iy}u(s,t)=p$, $\lim_{s\ra +\iy}u(s,t)=q$ for $s\in\R$ and $t\in[0,1]$, modulo the $\R$-action where $c\in\R$ maps $u(s,t)\mapsto u(s+c,t)$, and after dividing by the $\R$-action, such curves have no automorphisms.

For $p,q\in L\cap L'$, write $\oM_{p,q}$ for the moduli space of
$J$-holomorphic $\Si$ as in Figure \ref{la2fig2}. Modulo some
complicated smoothness issues, for generic $J$ we expect
$\oM_{p,q}$ to be a compact manifold with corners, of dimension
$\mu_{L,L'}(q)-\mu_{L,L'}(p)-1$, with boundary given by
\begin{equation*}
\pd\oM_{p,q}\cong\coprod_{r\in L\cap L':
\mu_{L,L'}(p)<\mu_{L,L'}(r)<\mu_{L,L'}(q)\!\!\!\!\!\!\!\!\!\!\!\!
\!\!\!\!\!\!\!\!\!\!\!\!\!\!\!\!\!\!\!\!\!\!\!\!} \oM_{p,r}\t\oM_{r,q}.
\end{equation*}
So if $\mu_{L,L'}(p)=k$, $\mu_{L,L'}(q)=k+1$ as in \eq{la2eq16}
then $\oM_{p,q}$ is a compact 0-manifold, i.e.\ a finite set,
and $N_{p,q}$ makes sense.

\item[(e)] There is a canonical isomorphism $HF^*(L,L)\cong
H^*(L;\Z_2)$, with $H^*(L;\Z_2)$ the cohomology of $L$ with
coefficients in $\Z_2$.

By Poincar\'e duality we thus have $HF^*(L,L)\cong
H_{m-*}(L;\Z_2)$, where $m=\dim L$. The cochain complex
$\bigl(CF^*(L,L),\d\bigr)$ is isomorphic to the chain complex of
a homology theory computing $H_{m-*}(L;\Z_2)$. In Seidel
\cite{Seid2} this homology theory is Morse homology, for some
choice of Morse function $f:L\ra\R$. In Fukaya, Oh, Ohta and Ono
\cite{FOOO}, it is the singular homology of $L$, defined using
smooth maps $\si:\De^n\ra L$, for $\De^n$ the $n$-simplex.
\end{itemize}

\begin{rem}{\bf(i) Liouville manifolds versus Liouville domains.}
We take $(M,\om),\la$ to be a Liouville manifold of finite type,
which as in \S\ref{la21} means that $(M,\om),\la$ is isomorphic to
the completion of a Liouville domain $(D,\om'),\la'$. Working with
compact Lagrangians in either Liouville manifolds of finite type, or
Liouville domains, is essentially equivalent.

Given compact Lagrangians $L_1,\ldots,L_k$ in $M$ and an almost
complex structure $J$ on $M$ convex at infinity, one can choose a
Liouville domain $D\subset M$ such that $M$ is the completion of
$D$, and $L_1,\ldots,L_k$ and all $J$-holomorphic curves $\Si$ in
$M$ with boundary in $L_1\cup\cdots\cup L_k$ lie in the interior of
$D$. Then working in $D$ or $M$ is equivalent for the purposes of
studying~$L_1,\ldots,L_k$.

Seidel \cite{Seid2} explains Lagrangian Floer homology and Fukaya
categories primarily in Liouville domains, rather than Liouville
manifolds of finite type. But he shows \cite[Cor.~10.6]{Seid2} that
two Liouville domains whose completions are isomorphic as Liouville
manifolds have equivalent derived Fukaya categories.
\smallskip

\noindent{\bf(ii)} As $(M,\om)$ is symplectic Calabi--Yau and $L,L'$
are graded, $\bigl(CF^*(L,L'),\d\bigr)$ and $HF^*(L,L')$ are graded
over $\Z$. If we drop the symplectic Calabi--Yau assumption and take
$L,L'$ to be oriented rather than graded, then
$\bigl(CF^*(L,L'),\d\bigr)$ and $HF^*(L,L')$ would only be graded
over $\Z_2$ rather than~$\Z$.
\smallskip

\noindent{\bf(iii)} Since we work over the coefficient ring $\Z_2$,
we do not need to worry about counting $J$-holomorphic curves with
signs, or orientations of moduli spaces.

To work over a coefficient ring $\La$ with $\mathop{\rm char}\La\ne
2$, we must define orientations on the moduli spaces $\oM_{p,q}$
of $J$-holomorphic discs $\Si$ as in Figure \ref{la2fig2}, in order
to count them with signs to define $N_{p,q}\in\Z$ in \eq{la2eq16}. To
define the orientations, we must suppose $L,L'$ have `relative spin
structures'.
\smallskip

\noindent{\bf(iv)} We consider only {\it exact\/} symplectic
manifolds and {\it exact\/} Lagrangians. If we drop the exactness
assumptions, then the theory becomes a great deal more complicated,
as in Fukaya, Oh, Ohta and Ono \cite{FOOO}:
\begin{itemize}
\setlength{\parsep}{0pt}
\setlength{\itemsep}{0pt}
\item[$\bu$] The moduli spaces $\oM_{p,q}$ are not
manifolds with corners even for generic $J$, but singular spaces
({\it Kuranishi spaces\/} or {\it polyfolds\/}). `Counting'
compact moduli spaces $\oM_{p,q}$ to define $N_{p,q}$ in
\eq{la2eq16}, if this is possible, gives answers in $\Q$ not
$\Z$, as (prestable) curves $\Si$ with finite automorphism
groups $G$ must be counted with weight $1/\md{G}$, so $\La$ must
be a $\Q$-algebra.
\item[$\bu$] The $\oM_{p,q}$ are compact only if we bound
the areas of curves $\Si$. To deal with this we take $\La$
to be a `Novikov ring' of formal power series in a variable
$q$, and `count' curves $\Si$ weighted by $q^{{\rm
area}(\Si)}$.
\item[$\bu$] Because of `bubbling' of holomorphic discs
with boundary in $L$ or $L'$, the na\"\i ve definition of
$\d$ on $CF^*(L,L')$ may not satisfy $\d^2=0$. One
compensates for this by choosing `bounding cochains' for
$L,L'$ (which may not exist, in which case we say that `$HF^*$ is obstructed') and using them to modify the definition of $\d$. The
use of bounding cochains means that the isomorphism
$HF^*(L,L)\ab\cong H^*(L,\La)$ in (e) may no longer hold.
\end{itemize}
\label{la2rem4}
\end{rem}

Lagrangian Floer cohomology is only the beginning of a more general theory of {\it Fukaya categories}, which still have no fully detailed exposition in the literature in the general case. We will use the theory for exact symplectic manifolds and exact Lagrangians developed by Seidel \cite{Seid2}, another reference is Fukaya~\cite{Fuka1}.

As above we take $(M,\om),\la$ to be a symplectic Calabi--Yau
Liouville manifold of finite type, we consider only compact, exact,
graded Lagrangians $L$ in $M$, and we work over the coefficient ring
$\La=\Z_2$. The idea is to define the {\it Fukaya category\/}
$\sF(M)$ of $(M,\om),\la$, an $A_\iy$-{\it category\/} whose objects
are compact, exact, graded Lagrangians $L$ in $M$, such that the
morphisms $\Hom(L_0,L_1)$ in $\sF(M)$ are the graded $\Z_2$-vector
spaces $CF^*(L_0,L_1)$ from above, and the $A_\iy$-operations
\e
\begin{split}
\mu^k:CF^{a_k}(L_{k-1},L_k)\t CF^{a_{k-1}}(L_{k-2},L_{k-1})\t \cdots
\t CF^{a_1}(L_0,L_1)&\\
\longra CF^{a_1+\cdots+a_k+2-k}(L_0,&L_k)
\end{split}
\label{la2eq17}
\e
for all $k\ge 1$, $a_1,\ldots,a_k\in\Z$ and Lagrangians
$L_0,\ldots,L_k$, with $\mu^1:CF^{a_1}(L_0,L_1)\ab\ra
CF^{a_1+1}(L_0,L_1)$ the differential $\d$ in
$\bigl(CF^*(L_0,L_1),\d\bigr)$. The coefficients in the
$\Z_2$-multilinear map $\mu^k$ in \eq{la2eq17} are obtained by
`counting' $J$-holomorphic $(k\!+\!1)$-gons in $M$ with boundary in
$L_0\cup L_1\cup\cdots\cup L_k$.

By a category theory construction similar to taking the derived
category of an abelian category, one then defines the {\it derived
Fukaya category\/} $D^b\sF(M)$, a triangulated category. Objects of
$D^b\sF(M)$ are `twisted complexes' of objects in $\sF(M)$, so
Lagrangians $L$ in $\sF(M)$ are also objects in $D^b\sF(M)$. The
translation functor $[1]$ in the triangulated category $D^b\sF(M)$
acts on Lagrangians $L$ by reversing the orientation of $L$ and
changing the grading $\th_L$ to $\th_L+\pi$, so that
\e
(L,\th_L)[k]\cong (L,\th_L+k\pi).
\label{la2eq18}
\e
The (graded) morphisms of Lagrangians $L,L'$ in $D^b\sF(M)$ are
$\Hom^*(L,L')=HF^*(L,L')$. Thus, if $L,L',L''$ are Lagrangians then
composition of morphisms in $D^b\sF(M)$ gives a $\Z_2$-bilinear map
$HF^j(L',L'')\t HF^i(L,L')\ra HF^{i+j}(L,L'')$ with the obvious
associativity properties.

Kontsevich's {\it Homological Mirror Symmetry Conjecture\/}
\cite{Kont}, motivated by String Theory, says (very roughly) that if
$M,\check M$ are `mirror' Calabi--Yau $m$-folds then there should be
an equivalence of triangulated categories
\begin{equation*}
D^b\sF(M)\simeq D^b\mathop{\rm coh}(\check M,\check J),
\end{equation*}
where $D^b\mathop{\rm coh}(\check M,\check J)$ is the bounded derived category of coherent sheaves on $(\check M,\check J)$. This has driven much research in the area.

The next two theorems will be useful tools for proving our main
results. For the first, of course part (b) implies part (a). We
separate the two cases, as in \S\ref{la4} we will see that it is
natural to use (a) to prove uniqueness of special Lagrangians, and
(b) uniqueness of LMCF expanders. Also (a) is easier to prove.

We suppose $J$ is generic in Theorems \ref{la2thm1} and \ref{la2thm2} as this is used in the definition of $HF^*(L,L')$, but one can remove this assumption by taking limits of almost complex structures and using Gromov compactness.

\begin{thm} Let\/ $(M,\om),\la$ be a symplectic Calabi--Yau
Liouville manifold of dimension $2m,$ so that we can form the
derived Fukaya category $D^b\sF(M)$ of compact, exact, graded
Lagrangians $L$ in $(M,\om)$ with coefficients in $\Z_2,$ as in
Seidel\/ {\rm\cite{Seid2}}. Suppose $L,L'$ are compact, oriented,
graded, transversely intersecting Lagrangians in $M$ which are
isomorphic as objects of\/ $D^b\sF(M)$. Then:
\begin{itemize}
\setlength{\parsep}{0pt}
\setlength{\itemsep}{0pt}
\item[{\bf(a)}] There exists $p\in L\cap L'$ with\/
$\mu_{L,L'}(p)=0$.
\item[{\bf(b)}] Let\/ $J$ be a generic almost complex structure
on $M$ compatible with\/ $\om$ and convex at infinity. Then
there exist\/ $p,q\in L\cap L'$ with\/ $\mu_{L,L'}(p)=0$ and\/
$\mu_{L,L'}(q)=m$ and a $J$-holomorphic disc $\Si$ in $M$ with
boundary in $L\cup L'$ and corners at\/ $p,q,$ of the form shown
in Figure\/ {\rm\ref{la2fig2}}.
\end{itemize}
\label{la2thm1}
\end{thm}

\begin{proof} For (a), as $L\cong L'$ in $D^b\sF(M)$ we have
\begin{equation*}
HF^k(L,L')\cong HF^k(L,L)\cong H^k(L;\Z_2)\quad\text{for all $k\in\Z$.}
\end{equation*}
But $H^0(L;\Z_2)$ contains the identity class $1_L$ which is nonzero
(all manifolds in this paper are assumed nonempty), so
$HF^0(L,L')\ne 0$, and thus $CF^0(L,L')\ne 0$. From above
$CF^0(L,L')$ is the $\Z_2$-vector space with basis those $p\in L\cap
L'$ with $\mu_{L,L'}(p)=0$. Part (a) follows.

For (b) we use a more complicated argument. Composition of morphisms
in the derived Fukaya category $D^b\sF(M)$ induces a map
\e
\ci : HF^0(L',L)\t HF^0(L,L')\longra HF^0(L,L),
\label{la2eq19}
\e
which as $L\cong L'$ in $D^b\sF(M)$ may be identified with the cup
product
\e
\cup:H^0(L;\Z_2)\t H^0(L;\Z_2)\longra H^0(L;\Z_2).
\label{la2eq20}
\e
Since $1_L\cup 1_L=1_L$, we see that \eq{la2eq20} is nonzero, so
that \eq{la2eq19} is nonzero. But \eq{la2eq19} is defined using the
$A_\iy$ operation
\e
\mu^2: CF^0(L',L)\t CF^0(L,L')\longra CF^0(L,L).
\label{la2eq21}
\e

\begin{figure}[htb]
\centerline{$\splinetolerance{.8pt}
\begin{xy}
0;<1mm,0mm>:
,(-20,0);(20,0)**\crv{(0,10)}
?(.95)="a"
?(.85)="b"
?(.75)="c"
?(.65)="d"
?(.55)="e"
?(.45)="f"
?(.35)="g"
?(.25)="h"
?(.15)="i"
?(.05)="j"
?(.5)="y"
,(-20,0);(-30,-6)**\crv{(-30,-5)}
,(20,0);(30,-6)**\crv{(30,-5)}
,(-20,0);(20,0)**\crv{(0,-10)}
?(.95)="k"
?(.85)="l"
?(.75)="m"
?(.65)="n"
?(.55)="o"
?(.45)="p"
?(.35)="q"
?(.25)="r"
?(.15)="s"
?(.05)="t"
?(.5)="z"
,(-20,0);(-30,6)**\crv{(-30,5)}
,(20,0);(30,6)**\crv{(30,5)}
,"a";"k"**@{.}
,"b";"l"**@{.}
,"c";"m"**@{.}
,"d";"n"**@{.}
,"e";"o"**@{.}
,"f";"p"**@{.}
,"g";"q"**@{.}
,"h";"r"**@{.}
,"i";"s"**@{.}
,"j";"t"**@{.}
,"y"*{<}
,"z"*{\bu}
,(-20,0)*{\bu}
,(-20,-3)*{p}
,(-20,-8)*{\mu_{L,L'}(p)=0}
,(20,0)*{\bu}
,(20,-3)*{q}
,(20,-8)*{\mu_{L',L}(q)=0}
,(0,-8)*{r}
,(0,0)*{\Si}
,(-32,4)*{L}
,(-32,-4)*{L'}
,(32,4)*{L}
,(32.5,-4)*{L'}
\end{xy}$}
\caption{Holomorphic triangles $\Si$ used to define $\mu^2$ in \eq{la2eq21}}
\label{la2fig3}
\end{figure}

Now $\mu^2$ in \eq{la2eq21} is defined by `counting' $J$-holomorphic
triangles $\Si$ of the form shown in Figure \ref{la2fig3}, where the
points $p,q,r$ are related to $CF^0(L,L')$, $CF^0(L',L)$ and
$CF^0(L,L)$ respectively. For the purposes of this argument, the
definition of $CF^0(L,L)$, and how this `counting' is done, are
irrelevant: all we need to know is that $\ci$ in \eq{la2eq19} is
nonzero, so $\mu^2$ in \eq{la2eq21} is nonzero, and there exists at
least one $J$-holomorphic triangle $\Si$ as in Figure \ref{la2fig3},
with $\mu_{L,L'}(p)=\mu_{L',L}(q)=0$, which implies that
$\mu_{L,L'}(q)=m$. Forgetting $r$ gives a disc $\Si$ as in Figure \ref{la2fig2}. Part (b) follows.
\end{proof}

\begin{thm} Let\/ $(M,\om),\la$ be a symplectic Calabi--Yau
Liouville manifold of dimension $2m,$ so that we can form the
derived Fukaya category $D^b\sF(M)$ of compact, exact, graded
Lagrangians $L$ in $(M,\om)$ with coefficients in $\Z_2,$ as in
Seidel\/ {\rm\cite{Seid2}}. Suppose $L,L',L''$ are compact,
oriented, graded, pairwise transversely intersecting Lagrangians in
$M$ which fit into a distinguished triangle
\e
\xymatrix@C=37pt{ L \ar[r]^\al & L' \ar[r]^\be & L'' \ar[r]^{\ga} &
L[1] }
\label{la2eq22}
\e
in $D^b\sF(M),$ with\/ $\al,\be,\ga\ne 0$. Let\/ $J$ be a generic
almost complex structure on $M$ compatible with\/ $\om$ and convex
at infinity. Then there exist points $p\in L\cap L',$ $q\in L'\cap
L'',$ $r\in L''\cap L$ with\/ $\mu_{L,L'}(p)=0,$
$\mu_{L',L''}(q)=0,$ and\/ $\mu_{L'',L}(r)=1,$ and a $J$-holomorphic
disc $\Si$ in $M$ with boundary in $L\cup L'\cup L''$ and corners
at\/ $p,q,r,$ of the form shown in Figure\/~{\rm\ref{la2fig4}}.
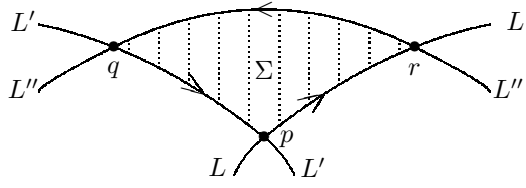
\begin{figure}[htb]
\centerline{$\splinetolerance{.8pt}
\begin{xy}
0;<1mm,0mm>:
,(-20,0);(20,0)**\crv{(0,10)}
?(.95)="a"
?(.85)="b"
?(.75)="c"
?(.65)="d"
?(.55)="e"
?(.45)="f"
?(.35)="g"
?(.25)="h"
?(.15)="i"
?(.05)="j"
?(.5)="y"
,(-20,0);(-30,-6)**\crv{(-30,-5)}
,(20,0);(30,-6)**\crv{(30,-5)}
,(20,0);(0,-12)**\crv{(10,-4)}
?(.1)="k"
?(.3)="l"
?(.5)="m"
?(.7)="n"
?(.9)="o"
?(.6)="x"
,(-20,0);(0,-12)**\crv{(-10,-4)}
?(.9)="p"
?(.7)="q"
?(.5)="r"
?(.3)="s"
?(.1)="t"
?(.6)="z"
,(0,-12);(4,-17)**\crv{(2.5,-14)}
,(0,-12);(-4,-17)**\crv{(-2.5,-14)}
,(20,0);(30,3)**\crv{(25,2)}
,(-20,0);(-30,3)**\crv{(-25,2)}
,"x";(6.5,-8.6)**\crv{(6.5,-8.6)}
,"x";(4.5,-7)**\crv{}
,"z";(-10,-3.5)**\crv{(-10,-3.5)}
,"z";(-11,-5.5)**\crv{}
,"a";"k"**@{.}
,"b";"l"**@{.}
,"c";"m"**@{.}
,"d";"n"**@{.}
,"e";"o"**@{.}
,"f";"p"**@{.}
,"g";"q"**@{.}
,"h";"r"**@{.}
,"i";"s"**@{.}
,"j";"t"**@{.}
,"y"*{<}
,(0,-12)*{\bu}
,(-20,0)*{\bu}
,(-20,-3)*{q}
,(20,0)*{\bu}
,(20,-3)*{r}
,(3,-12.3)*{p}
,(0,-3)*{\Si}
,(-32,3.5)*{L'}
,(-32,-5.5)*{L''}
,(33.4,3.5)*{L}
,(32.5,-5.5)*{L''}
,(6.5,-16)*{L'}
,(-6.2,-16)*{L}
\end{xy}$}
\caption{Holomorphic triangle $\Si$ with boundary in
$L\cup L'\cup L''$}
\label{la2fig4}
\end{figure}

\label{la2thm2}
\end{thm}

\begin{proof} As \eq{la2eq22} is distinguished, $L'$ is
isomorphic to $\mathop{\rm Cone}(L''[-1]\,{\buildrel\ga\over
\longra}\, L)$ in $D^b\sF(M)$, giving morphisms in $D^b\sF(M)$
\begin{equation*}
\phi:\mathop{\rm Cone}(L''[-1]\,{\buildrel\ga\over \longra}\,
L)\longra L',\quad
\psi:L'\longra \mathop{\rm Cone}(L''[-1]\,{\buildrel\ga\over
\longra}\, L),
\end{equation*}
with $\phi\ci\psi=\id_{L'}$. Let $\ti\al\in CF^0(L,L')$, $\ti\be\in
CF^0(L',L'')$ and $\ti\ga\in CF^1(L'',L)$ be cocycles representing
$\al,\be,\ga$ in \eq{la2eq22}.

Using the notation of Seidel \cite[\S I]{Seid2}, $\mathop{\rm
Cone}(L''[-1]\,{\buildrel\ga\over \longra}\, L)$ is represented by
the twisted complex $(L''\op L, \ti\ga)$, and $\phi,\psi$ are
represented by
\begin{equation*}
\ti\phi=(0,\ti\al):(L''\op L,
\ti\ga)\longra L',\quad
\ti\psi=(\ti\be,0):L'\longra (L''\op L,
\ti\ga).
\end{equation*}
Thus using Seidel \cite[eq.~(3.20)]{Seid2}, $\phi\ci\psi=\id_{L'}$
is represented by
\begin{equation*}
\mu^2_{\rm Tw}(\ti\phi,\ti\psi)
=\mu^2(0,\ti\be)+\mu^2(\ti\al,0)+\mu^3(\ti\al,\ti\ga,\ti\be)=
\mu^3(\ti\al,\ti\ga,\ti\be)\in CF^0(L',L').
\end{equation*}
Since $\id_{L'}\ne 0$, it follows that
\e
\mu^3:CF^0(L,L')\t CF^1(L'',L) \t CF^0(L',L'')\longra CF^0(L',L')
\label{la2eq23}
\e
is nonzero. But \eq{la2eq23} is defined by `counting' holomorphic
rectangles of the form shown in Figure \ref{la2fig5}, where the
points $p,q,r,s$ are related to $CF^0(L,L'),\ab CF^0(L',L''),
CF^1(L'',L),CF^0(L',L')$ respectively.
\begin{figure}[htb]
\centerline{$\splinetolerance{.8pt}
\begin{xy}
0;<1mm,0mm>:
,(-20,0);(20,0)**\crv{(0,10)}
?(.95)="a"
?(.85)="b"
?(.75)="c"
?(.65)="d"
?(.55)="e"
?(.45)="f"
?(.35)="g"
?(.25)="h"
?(.15)="i"
?(.05)="j"
?(.5)="y"
,(-20,0);(-30,-6)**\crv{(-30,-5)}
,(20,0);(30,-6)**\crv{(30,-5)}
,(20,0);(0,-12)**\crv{(10,-4)}
?(.1)="k"
?(.3)="l"
?(.5)="m"
?(.7)="n"
?(.9)="o"
?(.6)="x"
,(-20,0);(0,-12)**\crv{(-10,-4)}
?(.9)="p"
?(.7)="q"
?(.5)="r"
?(.3)="s"
?(.1)="t"
?(.55)="z"
,(0,-12);(4,-17)**\crv{(2.5,-14)}
,(0,-12);(-4,-17)**\crv{(-2.5,-14)}
,(20,0);(30,3)**\crv{(25,2)}
,(-20,0);(-30,3)**\crv{(-25,2)}
,"x";(6.5,-8.6)**\crv{(6.5,-8.6)}
,"x";(4.5,-7)**\crv{}
,"a";"k"**@{.}
,"b";"l"**@{.}
,"c";"m"**@{.}
,"d";"n"**@{.}
,"e";"o"**@{.}
,"f";"p"**@{.}
,"g";"q"**@{.}
,"h";"r"**@{.}
,"i";"s"**@{.}
,"j";"t"**@{.}
,"y"*{<}
,(0,-12)*{\bu}
,(-20,0)*{\bu}
,"z"*{\bu}
,(-20,-3)*{q}
,(20,0)*{\bu}
,(20,-3)*{r}
,(-3,-12.3)*{p}
,(-10.4,-7.5)*{s}
,(0,-3)*{\Si}
,(-32,3.5)*{L'}
,(-32,-5.5)*{L''}
,(33.4,3.5)*{L}
,(32.5,-5.5)*{L''}
,(6.5,-16)*{L'}
,(-6.2,-16)*{L}
,(-19,7)*{\mu_{L',L''}(q)=0}
,(19,7)*{\mu_{L'',L}(r)=1}
,(15,-11.5)*{\mu_{L,L'}(p)=0}
\end{xy}$}
\caption{Holomorphic rectangles $\Si$ used to define $\mu^3$
in \eq{la2eq23}}
\label{la2fig5}
\end{figure}

As in the proof of Theorem \ref{la2thm1}(b), we do not care how
$CF^0(L',L')$ is defined, nor how such curves $\Si$ are `counted'.
All that matters is that as \eq{la2eq23} is nonzero, there exists at
least one $J$-holomorphic rectangle $\Si$ as in Figure
\ref{la2fig5}. Forgetting $s$ gives a triangle $\Si$ as in Figure \ref{la2fig4}. This completes the proof.
\end{proof}

To apply Theorems \ref{la2thm1} and \ref{la2thm2}, we need usable
criteria for when two Lagrangians $L,L'$ in $M$ are isomorphic as
objects of $D^b\sF(M)$, or when three Lagrangians form a
distinguished triangle in $D^b\sF(M)$. An elementary one is that
$L,L'$ are isomorphic in $D^b\sF(M)$ if they are Hamiltonian
isotopic in~$(M,\om)$.

Now the state of the art in the field allows one to give a good
description of $D^b\sF(M)$ as a triangulated category for some
simple examples of noncompact, exact symplectic manifolds $(M,\om)$,
and in particular, to give criteria for when two Lagrangians $L,L'$
are isomorphic as objects in $D^b\sF(M)$, or three Lagrangians
$L,L',L''$ from a distinguished triangle. For example, Fukaya,
Seidel and Smith \cite{FSS} show that if $Z$ is a compact,
simply-connected spin manifold and $L\subset T^*Z$ is a compact,
exact, spin, Maslov zero Lagrangian, then $L$ is isomorphic in
$D^b\sF(T^*Z)$ to the zero section $Z$ in $T^*Z$. Also, in
\S\ref{la26} we will discuss Abouzaid and Smith's description
\cite{AbSm} of $D^b\sF(M)$ when $M$ is the plumbing $T^*Y\# T^*Z$ of
two cotangent bundles~$T^*Y,T^*Z$.

One interesting feature of this area, as in \cite{FSS,AbSm}, is that
even if one is interested only in compact Lagrangians $L$ in $M$, it
may be useful to consider noncompact Lagrangians as well. The
derived Fukaya category $D^b\sF(M)$ of compact Lagrangians embeds as
a full subcategory of a larger derived Fukaya category
$D^b\sF(M)_{\rm nc}$ of noncompact Lagrangians defined using
`wrapped Floer cohomology', and one can use $D^b\sF(M)_{\rm nc}$ as
a tool to understand $D^b\sF(M)$.

\subsection{The Fukaya category of $T^*\cS^m\# T^*\cS^m$}
\label{la26}

Now let $(M,\om),\la$ be the symplectic Calabi--Yau Liouville
$2m$-manifold defined in Example \ref{la2ex4} in \S\ref{la24}
starting from transverse Lagrangian planes $\Pi_0,\Pi_{\bs\phi}$ in
$\C^m$, known in the literature as the plumbing $T^*\cS^m\#
T^*\cS^m$, or as the $m$-dimensional Milnor fibre $A_2^m$. Abouzaid
and Smith \cite[Th.~1.3]{AbSm} give a good description of the
derived Fukaya category $D^b\sF(M)$, and classify compact, exact,
graded Lagrangians in $M$ up to isomorphism in~$D^b\sF(M)$.

To state our next result, which we deduce from \cite{AbSm}, recall
that $(M,\om)$ contains two exact, graded Lagrangian spheres
$\cS_0,\cS_{\bs\phi}$, intersecting transversely in one point $0$
with $\mu_{\cS_0,\cS_{\bs\phi}}(0)=0$ and $\mu_{\cS_{\bs\phi},\cS_0}(0)=m$. Therefore by \S\ref{la25}(d), the Floer complex $\bigl(CF^*(\cS_0,\cS_{\bs\phi}),\d\bigr)$ is $\Z_2$ in degree 0 and zero in other degrees, and $\bigl(CF^*(\cS_{\bs\phi},\cS_0),\d\bigr)$ is $\Z_2$ in degree $m$ and zero in other degrees, giving
\e
\begin{split}
HF^k(\cS_0,\cS_{\bs\phi})&=\begin{cases} \{0,\al\}\cong\Z_2, & k=0,
\\ 0, & k\ne 0,\end{cases} \\
HF^k(\cS_{\bs\phi},\cS_0)&=\begin{cases} \{0,\be\}\cong\Z_2, & k=m,
\\ 0, & k\ne m.\end{cases}
\end{split}
\label{la2eq24}
\e
Also $M$ contains closed, noncompact Lagrangians
$T^*_{\iy_0}\cS_0,T^*_{\iy_{\bs\phi}}\cS_{\bs\phi}$ diffeomorphic to
$\R^m$, which intersect $\cS_0,\cS_{\bs\phi}$ transversely
in~$\iy_0,\iy_{\bs\phi}$.

\begin{thm} Let\/ $(M,\om),\la$ be as in Example\/
{\rm\ref{la2ex4},} with compact, exact, graded Lagrangians
$\cS_0,\cS_{\bs\phi}$ and noncompact Lagrangians
$T^*_{\iy_0}\cS_0,T^*_{\iy_{\bs\phi}}\cS_{\bs\phi}$. Suppose $\bar L$ is
a compact, exact, graded Lagrangian in $M$ intersecting each of\/
$T^*_{\iy_0}\cS_0,\ab T^*_{\iy_{\bs\phi}}\cS_{\bs\phi}$ transversely in
one point. Then $H^*(\bar L;\Z_2)\cong H^*(\cS^m;\Z_2),$ and either:
\begin{itemize}
\setlength{\parsep}{0pt}
\setlength{\itemsep}{0pt}
\item[{\bf(a)}] $\bar L$ fits into a distinguished triangle in
$D^b\sF(M)$
\e
\xymatrix@C=31pt{ \cS_{\bs\phi}[n-1] \ar[r] & \bar L \ar[r] &
\cS_0[n] \ar[rr]^{\al[n]} && \cS_{\bs\phi}[n] }
\label{la2eq25}
\e
for some $n\in\Z,$ where $\al$ is as in {\rm\eq{la2eq24},} or
\item[{\bf(b)}] $\bar L$ fits into a distinguished triangle in
$D^b\sF(M)$
\e
\xymatrix@C=27pt{ \cS_0[n] \ar[r] & \bar L \ar[r] &
\cS_{\bs\phi}[n\!-\!m\!+\!1] \ar[rr]^(0.55){\be[n-m+1]} &&
\cS_0[n\!+\!1] }
\label{la2eq26}
\e
for some $n\in\Z,$ where $\be$ is as in \eq{la2eq24}.
\end{itemize}
Thus, the isomorphism class of\/ $\bar L$ in $D^b\sF(M)$ is
determined by the choice of {\bf(a)} or {\bf(b)} and\/ $n\in\Z$.
Changing $n$ corresponds to changing the grading of\/~$\bar L$.
\label{la2thm3}
\end{thm}

\begin{proof} Abouzaid and Smith \cite{AbSm} study the derived
Fukaya category $D^b\sF(M)$ of compact Lagrangians in $M$ by
embedding it as a full subcategory of a larger derived Fukaya
category $D^b\sF(M)_{\rm nc}$ of not-necessarily-compact Lagrangians
in $M$ called the `wrapped Fukaya category', using technology of
Abouzaid and Seidel \cite{AbSe}. In our case the cotangent fibres
$T^*_{\iy_0}\cS_0, \ab T^*_{\iy_{\bs\phi}}\cS_{\bs\phi}$ are objects of
$D^b\sF(M)_{\rm nc}$, and by \cite[Th.~1.2]{AbSe}, $D^b\sF(M)_{\rm
nc}$ is generated by $T^*_{\iy_0}\cS_0, \ab
T^*_{\iy_{\bs\phi}}\cS_{\bs\phi}$. In \cite[Cor.~1.4]{AbSm} they prove
that $H^*(\bar L;\Z_2)\cong H^*(\cS^m;\Z_2)$.

Now $\bar L$ intersects $T^*_{\iy_0}\cS_0, \ab
T^*_{\iy_{\bs\phi}}\cS_{\bs\phi}$ transversely in points $p,q,$ say.
Write $n=\mu_{\bar L,T^*_{\iy_0}\cS_0}(p)\in\Z$ and $n'=\mu_{\bar
L,T^*_{\iy_{\bs\phi}}\cS_{\bs\phi}}(q)\in\Z$. Then in $D^b\sF(M)_{\rm
nc}$ we have
\begin{equation*}
HF^k(\bar L,T^*_{\iy_0}\cS_0)\cong\begin{cases} \Z_2, & k=n, \\ 0, &
k\ne n, \end{cases}\;\> HF^k(\bar
L,T^*_{\iy_{\bs\phi}}\cS_{\bs\phi})\cong\begin{cases} \Z_2, & k=n',
\\ 0, & k\ne n'. \end{cases}
\end{equation*}
As in \cite[\S 2.1]{AbSm}, $\bar L$ is isomorphic in
$D^b\sF(M)\subset D^b\sF(M)_{\rm nc}$ to a `twisted complex' built
upon the object
\begin{equation*}
HF^*(\bar L,T^*_{\iy_0}\cS_0)\ot\cS_0\op
HF^*(\bar L,T^*_{\iy_{\bs\phi}}\cS_{\bs\phi})\ot \cS_{\bs\phi}
\cong \cS_0[n]\op \cS_{\bs\phi}[n'].
\end{equation*}
The allowed morphisms in this twisted complex lie in one or both of
\begin{itemize}
\setlength{\parsep}{0pt}
\setlength{\itemsep}{0pt}
\item[(a$)'$] $HF^1(\cS_0[n],\cS_{\bs\phi}[n'])\cong
HF^{1+n'-n}(\cS_0,\cS_{\bs\phi})$, or
\item[(b$)'$] $HF^1(\cS_{\bs\phi}[n'],\cS_0[n])\cong
HF^{1+n-n'}(\cS_{\bs\phi},\cS_0)$.
\end{itemize}

By \eq{la2eq24}, the group in case (a$)'$ is only nonzero if
$n'=n-1$, and the group in (b$)'$ is only nonzero if $n'=n+1-m$.
Note that both cannot be nonzero, as $m\ge 3$. Thus we have three
possibilities:
\begin{itemize}
\setlength{\parsep}{0pt}
\setlength{\itemsep}{0pt}
\item[(a)] $n'=n-1$, and the only nonzero morphism in the twisted
complex is $\al[n]\in HF^1(\cS_0[n],\cS_{\bs\phi}[n-1])$;
\item[(b)] $n'=n+1-m$, and the only nonzero morphism in the twisted
complex is $\be[n-m+1]\in HF^1(\cS_{\bs\phi}[n+1-m],\cS_0[n])$;
or
\item[(c)] $n,n'$ are arbitrary, and all morphisms in the twisted
complex are zero.
\end{itemize}
Cases (a)--(b) imply that $\bar L$ fits into the distinguished
triangles \eq{la2eq25}--\eq{la2eq26}, respectively. Case (c) gives
$\bar L\cong \cS_0[n]\op\cS_{\bs\phi}[n']$ in $D^b\sF(M)$, but this
does not occur as it contradicts \cite[Th.~1.3 \& Cor.~1.4]{AbSm}, which imply that the Floer cohomology of exact Lagrangians in $M$ is the cohomology of a sphere, so they are indecomposable as objects of $D^b\sF(M)$. The theorem follows.
\end{proof}

From this, using Example \ref{la2ex4}, we can deduce a partial
classification of exact, Maslov zero, AC Lagrangians in $\C^m$ with
cone~$\Pi_0\cup\Pi_{\bs\phi}$:

\begin{cor} Suppose\/ $L$ is a closed, exact, Maslov zero,
Asymptotically Conical Lagrangian in $\C^m$ for $m\ge 3$ with rate
$\rho<0$ and cone $C=\Pi_0\cup\Pi_{\bs\phi}$. Then $H^*(L;\Z_2)\cong
H^*(\cS^{m-1}\t\R;\Z_2),$ so $L$ is connected. Let\/ $\th_L:L\ra\R$
be the unique choice of grading for $L$ such that\/ $\th_L\ra 0$ as
$r\ra\iy$ in the end of\/ $L$ asymptotic to $\Pi_0$. Then either:
\begin{itemize}
\setlength{\parsep}{0pt}
\setlength{\itemsep}{0pt}
\item[{\bf(a)}] $\th_L\ra \phi_1+\cdots+\phi_m-\pi$ as
$r\ra\iy$ in the end of\/ $L$ asymptotic to $\Pi_{\bs\phi},$ and
the compactification $\bar L=L\cup\{\iy_0,\iy_{\bs\phi}\}$ from
Example\/ {\rm\ref{la2ex4}} satisfies Theorem\/
{\rm\ref{la2thm3}(a)} with\/ $n=0,$ or
\item[{\bf(b)}] $\th_L\ra \phi_1+\cdots+\phi_m-(m-1)\pi$ as
$r\ra\iy$ in the end of\/ $L$ asymptotic to $\Pi_{\bs\phi},$ and
the compactification $\bar L=L\cup\{\iy_0,\iy_{\bs\phi}\}$ from
Example\/ {\rm\ref{la2ex4}} satisfies Theorem\/
{\rm\ref{la2thm3}(b)} with\/~$n=0$.
\end{itemize}
\label{la2cor}
\end{cor}

\begin{proof} By Example \ref{la2ex4}, $\bar
L=L\cup\{\iy_0,\iy_{\bs\phi}\}$ is a compact, exact, Maslov zero
Lagrangian in $M$, so we can apply Theorem \ref{la2thm3}. This gives
$H^*(\bar L;\Z_2)\cong H^*(\cS^m;\Z_2)$, so $H^*(L;\Z_2)\cong
H^*(\cS^{m-1}\t\R;\Z_2)$. Thus $L$ is connected, so the grading
$\th_L$ is unique up to addition of $\pi\Z$. If $\th_{\bar L}$ is
the corresponding grading of $\bar L$, then $\th_L\ra\th_{\bar
L}(\iy_0)$ as $r\ra\iy$ in the end of $L$ asymptotic to $\Pi_0$, and
$\th_L\ra\th_{\bar L}(\iy_{\bs\phi})$ as $r\ra\iy$ in the end
asymptotic to~$\Pi_{\bs\phi}$.

Now in the notation of the proof of Theorem \ref{la2thm3}, we have
$\iy_0=p$, $\iy_{\bs\phi}=q$, $\th_{\bar L}(\iy_0)=n\pi$ and
$\th_{\bar L}(\iy_{\bs\phi})=\phi_1+\cdots+\phi_m+n'\pi$. Thus
$\th_L\ra 0$ as $r\ra\iy$ in the end of $L$ asymptotic to $\Pi_0$
forces $n=0$. Then in Theorem \ref{la2thm3}(a) we have $n'=-1$, and
in Theorem \ref{la2thm3}(b) we have $n'=1-m$. The corollary follows.
\end{proof}

\section{Special Lagrangians and LMCF expanders}
\label{la3}

Next we discuss Calabi--Yau manifolds, special Lagrangian
submanifolds, Lagrangian mean curvature flow, and Lagrangian MCF
expanders. Parts of \S\ref{la32} and \S\ref{la35} are new. Some references are Harvey and Lawson \cite{HaLa} and Joyce \cite{Joyc7} on special Lagrangian geometry, Neves \cite{Neve} on Lagrangian mean curvature flow and LMCF expanders, and Lawlor \cite{Lawl} and Joyce, Lee and Tsui \cite{JLT} for the families of examples we study in \S\ref{la33} and~\S\ref{la36}.

\subsection{Calabi--Yau $m$-folds and special Lagrangian $m$-folds}
\label{la31}

\begin{dfn} A {\it Calabi--Yau $m$-fold\/} is a quadruple
$(M,J,g,\Om)$ such that $(M,J)$ is an $m$-dimensional complex
manifold, $g$ is a K\"ahler metric on $(M,J)$ with K\"ahler form
$\om$, and $\Om$ is a holomorphic $(m,0)$-form on $(M,J)$ satisfying
\e
\om^m/m!=(-1)^{m(m-1)/2}(i/2)^m\Om\w\bar\Om.
\label{la3eq1}
\e
Then $g$ is Ricci-flat and its holonomy group is a subgroup of
$\mathop{\rm SU}(m)$. In this paper we do not require $M$
to be compact, nor do we require $g$ to have holonomy $\mathop{\rm
SU}(m)$, although many authors make these restrictions.

Note that $(M,\om)$ is a {\it symplectic Calabi--Yau manifold\/} in
the sense of \S\ref{la22}, and \eq{la2eq4}, \eq{la3eq1} coincide, so
that $J,\Om$ are as in Definition \ref{la2def3}. However, for Calabi--Yau
manifolds $J$ must be integrable and $\Om$ holomorphic, which were
not required for symplectic Calabi--Yau manifolds in~\S\ref{la22}.
\label{la3def1}
\end{dfn}

Let $\C^m$ have complex coordinates $(z_1,\dots,z_m)$ and complex
structure $J$, and define a metric $g$, a real 2-form $\om$ and a
complex $m$-form $\Om$ on $\C^m$ by
\begin{align*}
g=\ms{\d z_1}+\cdots+\ms{\d z_m},\quad
\om&=\ts\frac{i}{2}(\d z_1\w\d\bar z_1+\cdots+\d z_m\w\d\bar z_m),\\
\text{and}\quad\Om&=\d z_1\w\cdots\w\d z_m.
\end{align*}
Then $(\C^m,J,g,\Om)$ is the simplest example of a Calabi--Yau
$m$-fold. We define {\it calibrations} and {\it calibrated
submanifolds}, following Harvey and Lawson~\cite{HaLa}.

\begin{dfn} Let $(M,g)$ be a Riemannian manifold. An {\it oriented
tangent $k$-plane} $V$ on $M$ is a vector subspace $V$ of some
tangent space $T_xM$ to $M$ with $\dim V=k$, equipped with an
orientation. If $V$ is an oriented tangent $k$-plane on $M$ then
$g\vert_V$ is a Euclidean metric on $V$, so combining $g\vert_V$
with the orientation on $V$ gives a natural {\it volume form}
$\vol_V$ on $V$, which is a $k$-form on~$V$.

Now let $\vp$ be a closed $k$-form on $M$. We say that $\vp$ is a
{\it calibration} on $M$ if for every oriented $k$-plane $V$ on $M$
we have $\vp\vert_V\le \vol_V$. Here $\vp\vert_V=\al\cdot\vol_V$ for
some $\al\in\R$, and $\vp\vert_V\le\vol_V$ if $\al\le 1$. Let $N$ be
an oriented submanifold of $M$ with dimension $k$. Then each tangent
space $T_xN$ for $x\in N$ is an oriented tangent $k$-plane. We say
that $N$ is a {\it calibrated submanifold\/} if
$\vp\vert_{T_xN}=\vol_{T_xN}$ for all~$x\in N$.
\label{la3def2}
\end{dfn}

It is easy to show that calibrated submanifolds are automatically
{\it minimal submanifolds} \cite[Th.~II.4.2]{HaLa}. Here is the
definition of special Lagrangian submanifolds, taken from~\cite[\S III]{HaLa}.

\begin{dfn} Let $(M,J,g,\Om)$ be a Calabi--Yau $m$-fold. Then
$\Re\Om$ and $\Im\Om$ are real $m$-forms on $M$. As in \cite[\S
III]{HaLa}, the normalization \eq{la3eq1} implies that $\Re\Om$ is a
calibration on $(M,g)$. We call an oriented real $m$-dimensional
submanifold $L$ in $M$ a {\it special Lagrangian submanifold\/} of
$M$, or {\it SL\/ $m$-fold\/} for short, if $L$ is calibrated with
respect to $\Re\Om$, in the sense of Definition~\ref{la3def2}.

More generally, we call $L$ {\it special Lagrangian with angle\/}
$e^{i\phi}\in\U(1)$ if $L$ is calibrated with respect to
$\Re(e^{-i\phi}\Om)$. If we do not specify an angle, we
mean~$e^{i\phi}=1$.

\label{la3def3}
\end{dfn}

Harvey and Lawson \cite[Cor.~III.1.11]{HaLa} give the following
alternative characterization of special Lagrangian submanifolds:

\begin{prop} Let\/ $L$ be a real\/ $m$-dimensional submanifold of a
Calabi--Yau $m$-fold\/ $(M,J,g,\Om)$. Then $L$ admits an orientation
making it into an SL\/ $m$-fold in $M$ if and only if\/
$\om\vert_L\equiv 0$ and\/~$\Im\Om\vert_L\equiv 0$.
\label{la3prop1}
\end{prop}

Thus SL $m$-folds are Lagrangian submanifolds in the symplectic
manifold $(M,\om)$ satisfying the extra condition that
$\Im\Om\vert_L\equiv 0$, which is how they get their name. Every SL
$m$-fold $L$ is Maslov zero and has a natural grading $\th_L=0$.
Conversely, a graded Lagrangian $L$ with $\th_L=0$ is special
Lagrangian.

The deformation theory of special Lagrangian submanifolds was studied by
McLean \cite[\S 3]{McLe}:

\begin{thm} Let\/ $(M,J,g,\Om)$ be a Calabi--Yau $m$-fold, and\/ $N$
a compact SL\/ $m$-fold in $M$. Then the moduli space
$\cM_{\sst N}$ of special Lagrangian deformations of\/ $N$ is a smooth
manifold of dimension $b^1(N),$ the first Betti number of\/~$N$.
\label{la3thm1}
\end{thm}

\subsection{Special Lagrangian submanifolds in $\C^m$}
\label{la32}

The simplest example of a special Lagrangian submanifold of $\C^m$ is
\begin{equation*}
\R^m=\bigl\{(x_1,\ldots,x_m):x_j\in\R\bigr\}.
\end{equation*}
Any rotation of $\R^m$ by a matrix in $\mathop{\rm SU}(m)$ is also a
special Lagrangian plane in $\C^m$. An explicit class of SL
$m$-folds in $\C^m$ are the {\it special Lagrangian graphs},
\cite[\S III.2]{HaLa}. Let $U\subseteq\R^m$ be open and $f:U\ra\R$ be smooth, and define
\e
\begin{split}
\Ga_{\d f}=\bigl\{\bigl(x_1+i{\ts\frac{\pd f}{\pd x_1}}(x_1,\ldots,x_m),
\ldots,x_m+i{\ts\frac{\pd f}{\pd x_m}}(x_1,\ldots,x_m)\bigr):&\\
(x_1,\ldots,x_m)\in U\bigr\}&.
\end{split}
\label{la3eq2}
\e
Then $\Ga_{\d f}$ is a smooth, Lagrangian submanifold of
$(\C^m,\om)$. Locally, but not globally, every Lagrangian
submanifold in $\C^m$ arises from this construction. By Proposition
\ref{la3prop1}, $\Ga_{\d f}$ is special Lagrangian if
$\Im\Om\vert_{\Ga_{\d f}}=0$. This is equivalent to
\e
\ts\Im\det_{\sst\mathbb C}\bigl(I+i\Hess f\bigr)=0
\quad\text{on $U$,}
\label{la3eq3}
\e
where $\Hess f=\bigl(\frac{\pd^2f}{\pd x_i\pd x_j}\bigr)_{i,j=1}^m$
is the Hessian of $f$. Equation \eq{la3eq3} is a nonlinear
second-order elliptic partial differential equation upon
$f:\R^m\ra\R$, of Monge--Amp\`ere type. Its linearization at $f=0$
is $-\De f=0$, where $\De=-\frac{\pd^2}{\pd x_1^2}-\cdots-\frac{\pd^2}{\pd x_m^2}$ is the Laplacian on~$\R^m$.

Asymptotically Conical special Lagrangians, in the sense of
\S\ref{la23}, are an important class of SL $m$-folds in $\C^m$. In
\cite[\S 7]{Joyc2} the second author studied regularity of AC SL
$m$-folds near infinity, and showed \cite[Th.s 7.7 \& 7.11]{Joyc2}
that we may improve the asymptotic rate $\rho$.

\begin{thm} Let\/ $L$ be an AC SL $m$-fold in $\C^m,$ asymptotic
at rate $\rho<2$ to a special Lagrangian cone $C$ in $\C^m$. Define ${\scr
D}_{\sst\Si}$ to be the set of\/ $\al\in\R$ such that $\al(\al+m-2)$
is an eigenvalue of the Laplacian $\De_{\sst\Si}$ on $\Si$. Then
${\scr D}_{\sst\Si}$ is a discrete subset of\/ $\R$ with\/ ${\scr
D}_{\sst\Si}\cap(2-m,0)=\es$. If\/ $\rho,\rho'$ lie in the same
connected component of\/ $\R\sm{\scr D}_{\sst\Si}$ then $L$ is also
AC with rate~$\rho'$.
\label{la3thm2}
\end{thm}

McLean's Theorem, Theorem \ref{la3thm1}, was generalized to AC SL
$m$-folds by Marshall \cite{Mars} and Pacini \cite{Paci}. Here is a
special case of their results:

\begin{thm} Let\/ $L$ be an Asymptotically Conical SL\/ $m$-fold in
$\C^m$ for $m\ge 3$ with cone $C$ and rate $\rho\in (2-m,0),$ and
write $\cM_{\sst L}^\rho$ for the moduli space of deformations of $L$
as an AC SL\/ $m$-fold in $\C^m$ with cone $C$ and rate $\rho$. Then
$\cM_{\sst L}^\rho$ is a smooth manifold of dimension~$b^1_{\rm
cs}(L)=b^{m-1}(L)$.
\label{la3thm3}
\end{thm}

Note that if $L$ is an AC SL $m$-fold in $\C^m$ with cone $C$ and
rate $\rho\in(2-m,0)$, then so is $tL$ for all $t>0$, so that
$\{tL:t>0\}\subseteq\cM_{\sst L}^\rho$. If $b^{m-1}(L)=1$ then
$\dim\cM_{\sst L}^\rho=1$, so $\cM_{\sst L}^\rho=\{tL:t>0\},$ and the
only continuous deformations of $L$ as an AC SL $m$-fold are the
dilations $tL$ for $t>0$.

In a new result, we now study AC SL $m$-folds in $\C^m$ with cone~$C=\R^m$.

\begin{thm} Suppose $L$ is an SL\/ $m$-fold in $\C^m$ for $m\ge 3,$ which has an Asymptotically Conical end asymptotic to the cone $C=\R^m$ in $\C^m$ with rate $\rho<0$. (We allow $L$ also to have other AC ends asymptotic to other cones $C'$ in $\C^m;$ we are concerned only with\/ $L$ near infinity in\/~$C=\R^m$.) 

Then an open subset of\/ $L$ including the end asymptotic to $C$ may be written uniquely in the form $\Ga_{\d f}$ in {\rm\eq{la3eq2},} where $U=\bigl\{(x_1,\ldots,x_m)\in\R^m:x_1^2+\cdots+x_m^2>R^2\bigr\}$ for $R\gg 0,$ and\/ $f:U\ra\R$ is smooth, satisfies\/ {\rm\eq{la3eq3}} and\/ $\nabla^k f=O(r^{2-m-k})$ for all\/ $k=0,1,\ldots,$ and for all\/ $(x_1,\ldots,x_m)\in U$ we have
\e
f(x_1,\ldots,x_m)=r^{2-m}\cdot F(x_1/r^2,\ldots,x_m/r^2),
\label{la3eq4}
\e
where $F$ is a unique \begin{bfseries}real analytic\end{bfseries} function
\e
F:V=\bigl\{(y_1,\ldots,y_m)\in\R^m:y_1^2+\cdots+y_m^2<R^{-2}\bigr\}\longra\R.
\label{la3eq5}
\e
In particular, $F$ is well-defined and real analytic at\/ $0\in V,$ which corresponds to the limit to infinity in the cone $C=\R^m$ in $\C^m$.
\label{la3thm4}
\end{thm}

\begin{proof} Because $L$ is Lagrangian and converges in $C^1$ to $C=\R^m$ near infinity in $C$, we may write an open set in $L$ as a graph $\Ga_\la$ for $\la$ a closed 1-form on $U$, with $U$ as in the theorem for $R>0$ sufficiently large. As $H^1(U;\R)=0$, noting that $m\ge 3$, we may write $\la=\d f$ for $f:U\ra\R$ smooth. Then $L$ special Lagrangian implies that $f$ satisfies \eq{la3eq3}. Since $L$ has rate $\rho<0$, Theorem \ref{la3thm2} shows that $L$ is asymptotic to $C$ with any rate $\rho'\in(2-m,0)$. Thus $\nabla^kf=O(r^{\rho'-k})$ as $r\ra\iy$ for all $\rho'\in(2-m,0)$ and~$k=0,1,\ldots.$

Define $F:V\sm\{0\}\ra\R$ as in \eq{la3eq5} for $(y_1,\ldots,y_m)\ne 0$ by
\e
F(y_1,\ldots,y_m)=s^{2-m}\cdot f(y_1/s^2,\ldots,y_m/s^2),
\label{la3eq6}
\e
where $s=(y_1^2+\cdots+y_m^2)^{1/2}$. Then \eq{la3eq4} and \eq{la3eq6} are equivalent, and $F:V\sm\{0\}\ra\R$ is smooth. The condition $\nabla^k f=O(r^{\rho'-k})$ for $\rho'\in(2-m,0)$ and $k=0,1,\ldots$ imply that $\nabla^kF=O(s^{-\al})$ for all $\al\in(0,m-2)$. Therefore $F$ lies in $L^p_k$ near 0 in $V$ for all $p>1$ and $k=0,1,\ldots.$ Using the Sobolev Embedding Theorem, we see that $F$ has a unique smooth extension over $0\in V$. 

An elementary calculation shows that
\ea
&\ts(\frac{\pd^2 f}{\pd x_i\pd x_j})(y_1/s^2,\ldots,y_m/s^2)=
s^{m+2}\Bigl[\ts\frac{\pd^2F}{\pd y_i\pd y_j}(\bs y)
\nonumber\\
&-2s^{-2}\ts\sum\limits_{k=1}^m\Bigl(y_iy_k\frac{\pd^2F}{\pd y_j\pd y_k}(\bs y)+y_jy_k\frac{\pd^2F}{\pd y_i\pd y_k}(\bs y)\Bigr)
+4s^{-4}\sum\limits_{k,l=1}^my_iy_jy_ky_l\frac{\pd^2F}{\pd y_k\pd y_l}(\bs y)
\nonumber\\
&\ts-ms^{-2}\Bigl(y_i\frac{\pd F}{\pd y_j}(\bs y)\!+\!y_j\frac{\pd F}{\pd y_i}(\bs y)\Bigr)\!+\!4ms^{-4}\sum\limits_{k=1}^my_iy_jy_k\frac{\pd F}{\pd y_k}(\bs y)\!-\!2s^{-2}\de_{ij}\sum\limits_{k=1}^my_k\frac{\pd F}{\pd y_k}(\bs y)
\nonumber\\
&+m(m-2)s^{-4}y_iy_jF(\bs y)-\de_{ij}(m-2)s^{-2}F(\bs y)\Bigr],
\label{la3eq7}
\ea
where $\bs y=(y_1,\ldots,y_m)$. Consider the sum of \eq{la3eq7} with $i=j$ over $i=1,\ldots,m$. Since $\sum_{i=1}^my_i^2=s^2$ and $\sum_{i=1}^m\de_{ii}=m$, we see that the sum of terms in each of the second, and third, and fourth lines cancel out. Hence
\e
\ts\sum\limits_{i=1}^m\frac{\pd^2 f}{\pd x_i^2}(y_1/s^2,\ldots,y_m/s^2)=
s^{m+2}\sum\limits_{i=1}^m\frac{\pd^2F}{\pd y_i^2}(\bs y).
\label{la3eq8}
\e

Combining \eq{la3eq3}, \eq{la3eq7} and \eq{la3eq8}, we see that
\ea
&\ts s^{-m-2}\cdot\Im\det_{\sst\mathbb C}\bigl(I+i\Hess f\bigr)(y_1/s^2,\ldots,y_m/s^2)
\nonumber\\
&\quad=\ts\sum\limits_{i=1}^m\frac{\pd^2F}{\pd y_i^2}(\bs y)+\sum\limits_{k=1}^{[(m-1)/2]}P_{2k+1}\bigl(\bs y,F(\bs y),\frac{\pd F}{\pd y_i}(\bs y),\frac{\pd^2 F}{\pd y_i\pd y_j}(\bs y)\bigr),
\label{la3eq9}\\
&\text{where $P_{2k+1}=s^{2k(m+2)}\cdot\bigl($homogeneous polynomial of degree $2k+1$ in}
\nonumber\\
&\ts s^{-4}y_iy_jy_ky_l\frac{\pd^2F}{\pd y_k\pd y_l}(\bs y),\>\;
s^{-4}y_iy_jy_k\frac{\pd F}{\pd y_k}(\bs y),\;\>s^{-4}y_iy_jF(\bs y),\;\>i,j,k,l=1,\ldots,m\bigr),
\nonumber
\ea
noting that only odd powers of $\Hess f$ occur in the expansion of \eq{la3eq3}, and all ten terms on the r.h.s.\ of \eq{la3eq7} can be written in terms of the three terms on the bottom line of \eq{la3eq9} using~$s^2=y_1^2+\cdots+y_m^2$.

We now claim that $P_{2k+1}$ in \eq{la3eq9} is a {\it real analytic\/} function of its arguments $\bs y,F(\bs y),\frac{\pd F}{\pd y_i}(\bs y),\frac{\pd^2 F}{\pd y_i\pd y_j}(\bs y)$, {\it including at\/} $\bs y=0$, for each $k=1,\ldots,[(m-1)/2]$. There are two issues to consider. Firstly, $s=(y_1^2+\cdots+y_m^2)^{1/2}$ is not a real analytic function of $\bs y=(y_1,\ldots,y_m)$ at 0, although $s^2$ is. Fortunately, only even powers of $s$ appear in $P_{2k+1}$, so this does not arise.

Secondly, negative powers of $s$ are not real analytic (or well-defined) at $\bs y=0$, so the powers $s^{-4}$ in the bottom line of \eq{la3eq9} look like a problem. The total power of $s$ in $P_{2k+1}$ is $s^{2k(m+2)}\cdot s^{-4(2k+1)}=s^{2k(m-2)-4}$. Since $m\ge 3$ and $k\ge 1$, this total power is nonnegative except in the single case $m=3$, $k=1$, when it gives $s^{-2}$. When $m=3$ we can compute $P_3$ explicitly by hand, and we find that the terms in $y_iy_j,y_iy_jy_k,y_iy_jy_ky_l$ combine to give an overall factor of $y_1^2+\cdots+y_m^2=s^2$ which cancels the $s^{-2}$. Thus, $P_{2k+1}$ is indeed real analytic for all $m\ge 3$ and~$k\ge 1$.

By \eq{la3eq3} and \eq{la3eq9}, $F$ satisfies the real analytic nonlinear elliptic equation
\e
\ts\sum\limits_{i=1}^m\frac{\pd^2F}{\pd y_i^2}(\bs y)+\sum\limits_{k=1}^{[(m-1)/2]}P_{2k+1}\bigl(\bs y,F(\bs y),\frac{\pd F}{\pd y_i}(\bs y),\frac{\pd^2 F}{\pd y_i\pd y_j}(\bs y)\bigr)=0.
\label{la3eq10}
\e
Initially this holds for $\bs y\in V\sm\{0\}\cong U$, but extends to $\bs y\in V$ by continuity. Standard results in the theory of elliptic partial differential equations now show that $F$ is real analytic on all of $V$, as the coefficients of \eq{la3eq10} are real analytic functions on all of~$V$. This completes the proof.
\end{proof}

\subsection{Lawlor necks: a family of explicit AC SL $m$-folds}
\label{la33}

The following family of special Lagrangians in $\C^m$ was first
found by Lawlor \cite{Lawl}, made more explicit by Harvey
\cite[p.~139--140]{Harv}, and discussed from a different point of
view by the second author in \cite[\S 5.4(b)]{Joyc1}. They are often
called {\it Lawlor necks}. Our treatment is based on that of Harvey.

\begin{ex} Let $m>2$ and $a_1,\ldots,a_m>0$, and define
polynomials $p,P$ by
\e
p(x)=(1+a_1x^2)\cdots(1+a_mx^2)-1
\quad\text{and}\quad P(x)=\frac{p(x)}{x^2}.
\label{la3eq11}
\e
Define real numbers $\phi_1,\ldots,\phi_m$ and $A$ by
\e
\phi_k=a_k\int_{-\iy}^\iy\frac{\d x}{(1+a_kx^2)\sqrt{P(x)}}
\quad\text{and}\quad A=\int_{-\iy}^\iy\frac{\d x}{2\sqrt{P(x)}}\,.
\label{la3eq12}
\e
Clearly $\phi_k,A>0$. But writing $\phi_1+\cdots+\phi_m$ as one integral
gives
\begin{equation*}
\phi_1+\cdots+\phi_m=\int_0^\iy\frac{p'(x)\d x}{(p(x)+1)\sqrt{p(x)}}
=2\int_0^\iy\frac{\d w}{w^2+1}=\pi,
\end{equation*}
making the substitution $w=\sqrt{p(x)}$. So $\phi_k\in(0,\pi)$ and
$\phi_1+\cdots+\phi_m=\pi$. This yields a 1-1 correspondence between
$m$-tuples $(a_1,\ldots,a_m)$ with $a_k>0$, and $(m\!+\!1)$-tuples
$(\phi_1,\ldots,\phi_m,A)$ with $\phi_k\in (0,\pi)$,
$\phi_1+\cdots+\phi_m=\pi$ and~$A>0$.

For $k=1,\ldots,m$, define a function $z_k:\R\ra\C$ by
\begin{equation*}
z_k(y)={\rm e}^{i\psi_k(y)}\sqrt{a_k^{-1}+y^2}, \quad\text{where}\quad
\psi_k(y)=a_k\int_{-\iy}^y\frac{\d x}{(1+a_kx^2)\sqrt{P(x)}}\,.
\end{equation*}
Now write ${\bs\phi}=(\phi_1,\ldots,\phi_m)$, and define a
submanifold $L_{{\bs\phi},A}$ in $\C^m$ by
\e
L_{{\bs\phi},A}=\bigl\{(z_1(y)x_1,\ldots,z_m(y)x_m): y\in\R,\;
x_k\in\R,\; x_1^2+\cdots+x_m^2=1\bigr\}.
\label{la3eq13}
\e

Then $L_{{\bs\phi},A}$ is closed, embedded, and diffeomorphic to
${\cal S}^{m-1}\t\R$, and Harvey \cite[Th.~7.78]{Harv} shows that
$L_{{\bs\phi},A}$ is special Lagrangian. Also $L_{{\bs\phi},A}$ is
Asymptotically Conical, with rate $\rho=2-m$ and cone $C$ the union
$\Pi_0\cup\Pi_{\bs\phi}$ of two special Lagrangian $m$-planes
$\Pi_0,\Pi_{\bs\phi}$ in $\C^m$ given by
\e
\Pi_0=\bigl\{(x_1,\ldots,x_m):x_j\in\R\bigr\},\;\>
\Pi_{\bs\phi}=\bigl\{({\rm e}^{i\phi_1}x_1,\ldots, {\rm
e}^{i\phi_m}x_m):x_j\in\R\bigr\}.
\label{la3eq14}
\e

Apply Theorem \ref{la3thm3} with $L=L_{{\bs\phi},A}$ and
$\rho\in(2-m,0)$. As $L\cong{\cal S}^{m-1}\t\R$ we have $b^1_{\rm
cs}(L)=1$, so Theorem \ref{la3thm3} shows that $\dim\cM_{\sst
L}^\rho=1$. This is consistent with the fact that when $\bs\phi$ is
fixed, $L_{{\bs\phi},A}$ depends on one real parameter $A>0$. Here
$\bs\phi$ is fixed in $\cM_{\sst L}^\rho$ as the cone
$C=\Pi_0\cup\Pi_{\bs\phi}$ of $L$ depends on $\bs\phi$, and all
$\hat L\in\cM_{\sst L}^\rho$ have the same cone $C$, by definition.

Now $L_{{\bs\phi},A}$ is an exact, connected Lagrangian asymptotic
with rate $\rho<0$ to $C=\Pi_0\cup\Pi_{\bs\phi}$, so Definition
\ref{la2def6} defines an invariant $A(L_{{\bs\phi},A})\in\R$, which
we will compute. Let $\la$ be the Liouville form on $\C^m$ from
\eq{la2eq2}. Using the coordinates
$\bigl((x_1,\ldots,x_m),y\bigr)\in\cS^{m-1}\t\R$ from \eq{la3eq13},
an easy calculation shows that
\begin{equation*}
\la\vert_{L_{{\bs\phi},A}}=\frac{1}{2\sqrt{P(y)}}\,\d y.
\end{equation*}
Thus in Definition \ref{la2def6}, we may take the potential
$f:L_{{\bs\phi},A}\ra\R$ to be
\begin{equation*}
f\bigl(z_1(y)x_1,\ldots,z_m(y)x_m\bigl)=\int_{-\iy}^y\frac{\d x}{2\sqrt{P(x)}}\,.
\end{equation*}

Now $y\ra -\iy$ is the limit $r\ra\iy$ in the end of
$L_{{\bs\phi},A}$ asymptotic to $\Pi_0$, and $y\ra \iy$ the limit
$r\ra\iy$ in the end of $L_{{\bs\phi},A}$ asymptotic to
$\Pi_{\bs\phi}$. So in Definition \ref{la2def6} we have $c_0=0$ and
$c_{\bs\phi}= \int_{-\iy}^\iy\d x/2\sqrt{P(x)}=A$, and $A(L_{{\bs\phi},A})=A$, which is why we chose this value for $A$ in \eq{la3eq12}. Note that $A(L_{{\bs\phi},A})>0$. We will see later that $A(L_{{\bs\phi},A})$ is the area of a certain $J$-holomorphic curve in~$\C^m$.

Now suppose $\phi_1,\ldots,\phi_m\in(0,\pi)$ with
$\phi_1+\cdots+\phi_m=(m-1)\pi$. Then
$(\pi-\phi_1)+\cdots+(\pi-\phi_m)=\pi$. Write
$\bs\phi=(\phi_1,\ldots,\phi_m)$ and
$\bs\pi-\bs\phi=(\pi-\phi_1,\ldots,\pi-\phi_m)$. For each $A>0$,
define
\e
\ti L_{{\bs\phi},-A}=\begin{pmatrix} e^{i\phi_1} & 0 & \cdots & 0
\\ 0 & e^{i\phi_2} & \cdots & 0 \\ \vdots & \vdots & \ddots & \vdots \\
0 & 0 & \cdots & e^{i\phi_m} \end{pmatrix}L_{{\bs\pi-\bs\phi},A}.
\label{la3eq15}
\e
The matrix $M$ in \eq{la3eq15} has $M\cdot \Pi_0=\Pi_{\bs\phi}$ and
$M\cdot \Pi_{\bs\pi-\bs\phi}=\Pi_0$. Since $L_{{\bs\pi-\bs\phi},A}$
is asymptotic to $\Pi_0\cup\Pi_{\bs\pi-\bs\phi}$, we see that $\ti
L_{{\bs\phi},-A}=M\cdot L_{{\bs\pi-\bs\phi},A}$ is asymptotic to
$\Pi_0\cup\Pi_{\bs\phi}$. Also, as the r\^oles of the ends
$\Pi_0,\Pi_{\bs\phi}$ are exchanged, we have
\begin{equation*}
A(\ti L_{{\bs\phi},-A})=-A(L_{{\bs\pi-\bs\phi},A})=-A<0,
\end{equation*}
which is why we chose the notation $\ti L_{{\bs\phi},-A}$.

Observe that $L_{{\bs\phi},A}$ for $\phi_1,\ldots,\phi_m\in(0,\pi)$
with $\phi_1+\cdots+\phi_m=\pi$ is exact and graded with grading
$\th_{L_{{\bs\phi},A}}=0$, and so satisfies Corollary
\ref{la2cor}(a). Similarly, $\ti L_{{\bs\phi},-A}$ for
$\phi_1,\ldots,\phi_m\in(0,\pi)$ with
$\phi_1+\cdots+\phi_m=(m-1)\pi$ is exact and graded with grading
$\th_{\ti L_{{\bs\phi},-A}}=0$, and so satisfies
Corollary~\ref{la2cor}(b).
\label{la3ex1}
\end{ex}

\subsection{Lagrangian mean curvature flow, LMCF expanders}
\label{la34}

Let $(M,g)$ be a Riemannian manifold, and $N$ a compact manifold
with $\dim N<\dim M$, and consider embeddings $\io:N\hookra M$, so
that $\io(N)$ is a submanifold of $M$. {\it Mean curvature flow\/}
({\it MCF\/}) is the study of smooth 1-parameter families $\io_t$,
$t\in[0,\ep)$ of such $\io_t:N\hookra M$ satisfying $\frac{\d\io_t}{\d t}=H_{\io_t}$, where $H_{\io_t}\in C^\iy(\io_t^*(TM))$ is the mean curvature of $\io_t:N\hookra M$. MCF is the negative gradient flow of the volume functional for compact submanifolds $\io(N)$ of $M$. It has a unique short-time solution starting from any compact submanifold~$\io:N\hookra M$.

Now let $(M,J,g,\Om)$ be a Calabi--Yau $m$-fold, and $\io:L\hookra
M$ a compact Lagrangian submanifold. Then the mean curvature of $L$
is $H=J\nabla\Th_L$, where $\Th_L:L\ra\U(1)$ is the angle function from Definition \ref{la2def3}. Thus $H$ is an infinitesimal deformation of $L$ as a Lagrangian. Smoczyk \cite{Smoc1} shows that MCF starting from $L$
preserves the Lagrangian condition, yielding a 1-parameter family of
Lagrangians $\io_t:L\hookra M$ for $t\in[0,\ep)$, which are all in
the same Hamiltonian isotopy class if $L$ is Maslov zero. This is
{\it Lagrangian mean curvature flow\/} ({\it LMCF\/}). Special
Lagrangians are stationary points of LMCF.

It is an important problem to understand the singularities which
arise in Lagrangian mean curvature flow; for a survey, see Neves
\cite{Neve}. Singularities in LMCF are generally locally modelled on
{\it soliton solutions}, Lagrangians in $\C^m$ which move by
rescaling or translation under LMCF. We will be particularly
interested in {\it LMCF expanders:}

\begin{dfn} A Lagrangian $L$ in $\C^m$ is called an {\it LMCF
expander\/} if $H=\al F^\perp$ in $C^\iy(T\C^m\vert_L)$, where $H$
is the mean curvature of $L$ and $F^\perp$ is the orthogonal
projection of the position vector $F$ (that is, the inclusion
$F:L\hookra\C^m$) to the normal bundle $TL^\perp\subset
T\C^m\vert_L$, and $\al>0$ is constant.

This implies that (after reparametrizing by diffeomorphisms of $L$)
the family of Lagrangians $L_t:=\sqrt{2\al t}\,L$ for $t\in(0,\iy)$
satisfy Lagrangian mean curvature flow. That is, LMCF expands $L$ by
dilations.

Similarly, we call $L$ an {\it LMCF shrinker\/} if $H=\al F^\perp$
for $\al<0$, and then $L_t:=\sqrt{2\al t}\,L$ for $t\in(-\iy,0)$
satisfy LMCF, so LMCF shrinks $L$ by dilations.
\label{la3def4}
\end{dfn}

LMCF expanders and LMCF shrinkers are close relatives of SL
$m$-folds in $\C^m$. Taking the limit $\al\ra 0$ in the equation
$H=\al F^\perp$ gives $H=0$. A Lagrangian $L$ with $H=0$ has angle
function $\Th_L$ constant in $\U(1)$; if $\Th_L=1$ then $L$ is
special Lagrangian, in general $L$ is special Lagrangian with some
angle $e^{i\th}$ in $\U(1)$. Thus, SL $m$-folds in $\C^m$ are
solutions of the limit as $\al\ra 0$ of the LMCF expander
equation~$H=\al F^\perp$.

Let $L$ be an LMCF expander in $\C^m$, satisfying $H=\al F^\perp$
for $\al>0$. As in \S\ref{la22} we could suppose $L$ is graded, with
phase function $\th_L:L\ra\R$ with $\Th_L=e^{i\th_L}$. Or as in
\S\ref{la23} we could suppose $L$ is exact, with potential
$f_L:L\ra\R$ with $\d f_L=\la\vert_L$. For LMCF expanders, these are
connected: we have
\begin{align*}
\d\th_L
&=\bigr(-J\nabla\th_L\cdot\om\vert_L\bigr)\vert_{T^*L}
=-\bigr(H\cdot\om\vert_L\bigr)\vert_{T^*L}
=-\bigr(\al F^\perp\cdot\om\vert_L\bigr)\vert_{T^*L}\\
&=-\al\bigr(F\cdot\om\vert_L\bigr)\vert_{T^*L}
=-\al\bigr(\ha\la\vert_L\bigr)\vert_{T^*L}
=-\ha\al\cdot \d f_L,
\end{align*}
in 1-forms on $L$, using $(\d\th_L)_d=-(J_b^c
g^{ab}(\d\th_L)_a)\om_{cd}$ in index notation for tensors as
$g^{ab}J_b^c\om_{cd}=-\de^a_d$ in the first step, $H=J\nabla\th_L$
in the second step, $H=\al F^\perp$ in the third, $L$ Lagrangian and
$(F-F^\perp)\in TL$ in the fourth, $F\cdot\om=\ha\la$ for $\la$ as
in \eq{la2eq2} in the fifth, and $\d f_L=\la\vert_L$ in the sixth.

This implies that an LMCF expander $L$ is Maslov zero (that is, $L$
admits a grading $\th_L$) if and only if it is exact, and if $L$ is
graded with grading $\th_L$, then we may take the potential $f_L$ to
be
\e
f_L=-2\th_L/\al.
\label{la3eq16}
\e

Theorem \ref{la3thm2} described the regularity of AC SL $m$-folds
near $\iy$. Similarly, Lotay and Neves \cite[Th.~3.1]{LoNe} study
the regularity of AC LMCF expanders $L$ near $\iy$. They restrict to
$L$ with cone $C=\Pi_0\cup\Pi_{\bs\phi}$, but their proof (for the
part of the result we quote) is valid for general cones $C$:

\begin{thm}[Lotay and Neves \cite{LoNe}] Let\/ $L$ be an AC LMCF
expander in $\C^m$ satisfying $H=\al F^\perp$ for $\al>0,$
asymptotic at rate $\rho<2$ to a Lagrangian cone $C$ in $\C^m$. Then in
Definition\/ {\rm\ref{la2def5}} we may choose $K,T,\vp$ with
\e
\bmd{\nabla^k(\vp-\iota)}=O\bigl(e^{-\al r^2/2}\bigr)
\quad\text{as $r\ra\iy,$ for all\/ $k=0,1,2,\ldots.$}
\label{la3eq17}
\e

\label{la3thm5}
\end{thm}

Actually, Lotay and Neves only state the weaker result that $\bmd{\nabla^k(\vp-\iota)}=O\bigl(e^{-br^2}\bigr)$ for some $b>0$ depending on $k$, but their proof can be improved to give $b=\al/2$ as follows. They write $L$ near infinity in $C$ as a graph $\Ga_{\d\psi}$ for $\psi:C\sm\,\ov{\!B}_R\ra\R$ a smooth function. They consider only the case $\al=1$, and their proof by a barrier method shows that $\md{\nabla\psi}\le D e^{-r^2/2}$ on $C\sm\,\ov{\!B}_{R'}$ for some $R'>R$ and $D>0$, which gives \eq{la3eq17} when $k=0$ and $\al=1$. The case $k=0$ and $\al>0$ follows by rescaling $L$ by a homothety in~$\C^n$.

Lotay and Neves then deduce their version of \eq{la3eq17} for $k>0$ using a bound $\nm{\psi}_{C^{k+2}}=O(1)$ and interpolation inequalities in H\"older spaces. Calculation shows that their method yields $\bmd{\nabla^k(\vp-\iota)}=O\bigl(e^{-br^2}\bigr)$ for $b\le\al/4(k+2)$, which is suboptimal, but gives exponential decay, as Lotay and Neves wanted. If we instead bound the higher derivatives of $\psi$ using standard interior estimates for elliptic equations on balls of radius $O(1)$, we easily prove \eq{la3eq17} when~$k>0$.

Thus, for AC LMCF expanders $L$ the convergence to the cone $C$ is
exponential rather than polynomial. Note that \eq{la3eq17} implies
\eq{la2eq8} for all $\rho<2$, so $L$ is Asymptotically Conical with
rate any $\rho<2$ (essentially it has rate $\rho=-\iy$). 

\subsection{LMCF expanders as Lagrangian graphs}
\label{la35}

A straightforward calculation proves:

\begin{prop} Let\/ $U\subseteq\R^m$ be open and connected and\/
$f:U\ra\R$ be smooth, and define a Lagrangian $\Ga_{\d f}$ in $\C^m$
by
\e
\begin{split}
\Ga_{\d f}=\bigl\{\bigl(x_1+i{\ts\frac{\pd f}{\pd x_1}}(x_1,\ldots,x_m),
\ldots,x_m+i{\ts\frac{\pd f}{\pd x_m}}(x_1,\ldots,x_m)\bigr):&\\
(x_1,\ldots,x_m)\in U&\bigr\}.
\end{split}
\label{la3eq18}
\e
Then $\Ga_{\d f}$ satisfies the equation $H=\al F^\perp$ for
$\al\in\R$ if and only if
\e
\arg\det_{\sst\mathbb C}\bigl(I+i\Hess
f\bigr)=\al\bigl(2f-\ts\sum_{j=1}^mx_j \frac{\pd f}{\pd x_j}\bigr)+c
\label{la3eq19}
\e
for some $c\in\R$.
\label{la3prop2}
\end{prop}

If $\al\ne 0$ then replacing $f$ by $f+c/(2\al)$, we may take $c=0$
in \eq{la3eq19}. Then linearizing \eq{la3eq19} at $f=0$ gives
\e
\ts\sum_{j=1}^m\frac{\pd^2f}{\pd x_j^2}+\al\bigl(\sum_{j=1}^mx_j \frac{\pd f}{\pd x_j}-2f\bigr)=0.
\label{la3eq20}
\e
In a new result, we study the asymptotic behaviour of solutions of equations \eq{la3eq19} and \eq{la3eq20}. Essentially the same proof would also give an asymptotic description of LMCF expanders asymptotic to other (special) Lagrangian cones $C$ in $\C^m$ with $\Si=C\cap\cS^{2m-1}$ compact.

\begin{thm}{\bf(a)} Set\/ $U=\R^m\sm\,\ov{\!B}_R$ for some $R>0,$ and suppose $f:U\ra\R$ is smooth and satisfies \eq{la3eq20} and\/ $\nm{\d f}=O(e^{-\al r^2/2})$. Then:
\begin{itemize}
\setlength{\parsep}{0pt}
\setlength{\itemsep}{0pt}
\item[{\bf(i)}] We may write 
\e
f(x_1,\ldots,x_m)=r^{-1-m}e^{-\al r^2/2}h\bigl((x_1/r,\ldots,x_m/r),r^{-2}\bigr),
\label{la3eq21}
\e
where $h\bigl((y_1,\ldots,y_m),t\bigr):\cS^{m-1}\t(0,R^{-2})\ra\R$ is real analytic and 
\e
4t^2\frac{\pd^2h}{\pd t^2}+2\bigl[\al+(m+6)t\bigr]\frac{\pd h}{\pd t}+3(m+1)h-\De_{\cS^{m-1}}h=0.
\label{la3eq22}
\e
Here $\cS^{m-1}=\bigl\{(y_1,\ldots,y_m)\in\R^m:y_1^2+\cdots+y_m^2=1\bigr\}$ is the unit sphere in $\R^m,$ and\/ $\De_{\cS^{m-1}}$ the Laplacian on $\cS^{m-1},$ with the sign convention that\/ $\De_{\R^m}=-\frac{\pd^2}{\pd x_1^2}-\cdots-\frac{\pd^2}{\pd x_m^2},$ so that\/ $\De_{\cS^{m-1}}$ has nonnegative eigenvalues.
\item[{\bf(ii)}] There are unique polynomials $p_k:\R^m\ra\R$ for $k=0,1,2,\ldots$ such that\/ $p_k$ is harmonic on $\R^m$ (i.e. $\De_{\R^m}p_k=0$) and homogeneous of degree $k,$ and 
\e
h\bigl((y_1,\ldots,y_m),t\bigr)=\sum_{k=0}^\iy p_k\vert_{\cS^{m-1}}\bigl(y_1,\ldots,y_m)\cdot A_k(t),
\label{la3eq23}
\e
where the sum converges absolutely in $C^l$ for all\/ $l\ge 0$ on any compact subset of\/ $\cS^{m-1}\t(0,R^{-2}),$ and\/ $A_k:[0,\iy)\ra\R$ is the unique smooth solution of the second order o.d.e.\
\end{itemize}
\e
4t^2\frac{\d^2}{\d t^2}A_k(t)\!+\!2\bigl[\al\!+\!(m\!+\!6)t\bigr]\frac{\d}{\d t}A_k(t)\!+\!\bigl[3(m\!+\!1)\!-\!k(m\!+\!k\!-\!2)\bigr]A_k(t)\!=\!0
\label{la3eq24}
\e
\begin{itemize}
\setlength{\parsep}{0pt}
\setlength{\itemsep}{0pt}
\item[]with\/ $A_k(0)=1$ (note in particular that\/ $A_k$ is smooth at\/ $t=0$). If\/ $k(m+k-2)>3(m+1),$ which holds if\/ $k$ is sufficiently large, then $A_k$ maps $[0,\iy)\ra[1,\iy)$ and is strictly increasing.
\item[{\bf(iii)}] The sum \eq{la3eq23} converges to $\bar h:\cS^{m-1}\t[0,R^{-2})\ra\R$ satisfying {\rm\eq{la3eq22},} where $\bar h$ is smooth on $\cS^{m-1}\t[0,R^{-2}),$ but generally not real analytic when\/ $t=0$. Also $\bar h^\iy:\cS^{m-1}\ra\R$ given by $\bar h^\iy(y_1,\ldots,y_m)=\bar h\bigl((y_1,\ldots,y_m),0\bigr)$ is real analytic, and determines $\bar h,h$ and hence $f$ uniquely.
\end{itemize}

\noindent{\bf(b)} Let\/ $L$ be an AC LMCF expander in $\C^m$ satisfying $H=\al F^\perp$ for $\al>0,$ with one end asymptotic at rate $\rho<2$ to the Lagrangian cone $C=\R^m$ in $\C^m$. Then for $R\gg 0$ we have a unique real analytic $f:U=\R^m\sm\,\ov{\!B}_R\ra\R$ satisfying \eq{la3eq19} with\/ $c=0$ such that\/ $\Ga_{\d f}\subset L,$ for $\Ga_{\d f}$ as in \eq{la3eq18}. Define $h:\cS^{m-1}\t(0,R^{-2})\ra\R$ by \eq{la3eq21}. Then as in\/ {\bf(a)(iii)\rm,} $h$ extends to a smooth\/ $\bar h:\cS^{m-1}\t[0,R^{-2})\ra\R,$ and\/ $\bar h^\iy(y_1,\ldots,y_m)=\bar h\bigl((y_1,\ldots,y_m),0\bigr)$ is real analytic.
\label{la3thm6}
\end{thm}

\begin{proof} Part (a)(i) is a straightforward calculation: for $f,h$ related by \eq{la3eq21}, equation \eq{la3eq20} for $f$ is equivalent to \eq{la3eq22} for $h$. Note that $r^{-1-m}e^{-\al r^2/2}$ in \eq{la3eq21} is chosen to make \eq{la3eq22} simple. In particular, if we instead used a factor $r^ae^{\be r^2}$, then \eq{la3eq22} would have extra terms in $t^{-1}h$ and $t^{-2}h$ which vanish exactly when $a=-m-1$ and $\be=-\ha\al$. That $f$ and $h$ are real analytic follows from the fact that they satisfy elliptic p.d.e.s \eq{la3eq20} and \eq{la3eq22} with real analytic coefficients. Note that the same argument does not prove the extension $\bar h$ of $h$ in (iii) is real analytic, since \eq{la3eq22} is not elliptic when~$t=0$. 

For (ii), by general facts in analysis we may expand $h$ in terms of an orthonormal basis of $L^2(\cS^{m-1})$ consisting of eigenvectors of $\De_{\cS^{m-1}}$. It is known that all eigenvectors of $\De_{\cS^{m-1}}$ are of the form $p\vert_{\cS^{m-1}}$, where $p:\R^m\ra\R$ is a harmonic polynomial homogeneous of degree $k=0,1,2,\ldots,$ and the corresponding eigenvalue of $\De_{\cS^{m-1}}$ is $\la_p=k(m+k-2)$. Choose such $p_k^1,\ldots,p_k^{d_k}$ such that $p_k^1\vert_{\cS^{m-1}},\ldots,p_k^{d_k}\vert_{\cS^{m-1}}$ is an $L^2$-orthonormal basis for the $k(m+k-2)$-eigenspace of $\De_{\cS^{m-1}}$ in $L^2(\cS^{m-1})$. 

Define $A_k^i:(0,R^{-2})\ra\R$ for $k=0,1,\ldots$ and $i=1,\ldots,d_k$ by 
\begin{equation*}
A_k^i(t)=\int_{\bs y\in\cS^{m-1}}h(\bs y,t)\,p_k^i(\bs y)\,\d V_{\cS^{m-1}},
\end{equation*}
where $\bs y=(y_1,\ldots,y_m)$. Since $h:\cS^{m-1}\t(0,R^{-2})\ra\R$ is real analytic we have
\e
h(\bs y,t\bigr)=\sum_{k=0}^\iy\sum_{i=1}^{d_k}p_k^i(\bs y)A_k^i(t),
\label{la3eq25}
\e
where the sum converges absolutely in $C^l$ for $l\ge 0$ on compact subsets of $\cS^{m-1}\t(0,R^{-2})$. As $\De_{\cS^{m-1}}(p_k^i\vert_{\cS^{m-1}})=k(m+k-2)p_k^i\vert_{\cS^{m-1}}$, substituting \eq{la3eq25} into \eq{la3eq22} and noting that the sums in \eq{la3eq25} commute with the differential operators in \eq{la3eq22} shows that each $A_k^i(t)$ satisfies the o.d.e.\ \eq{la3eq24} on~$(0,R^{-2})$.

Consider the second order linear o.d.e.\ \eq{la3eq24}. For $t\in(0,\iy)$, where the coefficient $4t^2$ of $\frac{\d^2}{\d t^2}A_k(t)$ is nonzero, this o.d.e.\ is nonsingular, so the solutions form a 2-dimensional vector space, with basis $A_k(t),B_k(t)$, say. It turns out that up to scale one solution, say $A_k(t)$, extends smoothly to $t=0$, and we determine $A_k$ uniquely by putting $A_k(0)=1$. The second solution $B_k(t)$ on $(0,\iy)$ has $\md{B_k(t)}\ra \iy$ as $t\ra 0$, where to leading order~$B_k(t)\approx Ct^{d}e^{\al t^{-1}/2}$.

We can determine the Taylor series of $A_k$ at $t=0$ from \eq{la3eq24} explicitly: putting $A_k(t)\sim\sum_{l=0}^\iy c_lt^l$, we have $c_0=1$ as $A_k(0)=1$, and taking coefficients in $t^l$ in \eq{la3eq24} yields
\begin{equation*}
2\al(l+1)c_{l+1}=-lc_l\bigl[4l+5m+11-k(m+k-2)\bigr],
\end{equation*}
giving inductive formulae for $c_1,c_2,\ldots.$ If $\frac{1}{4}[k(m+k-2)-5m-11]$ is a nonnegative integer then $c_l=0$ for $l>\frac{1}{4}[k(m+k-2)-5m-11]$, and $A_k(t)$ is a polynomial. Otherwise the $c_l$ grow factorially in $l$, so $\sum_{l=0}^\iy c_lt^l$ does not converge for any $t>0$, and $A_k$ is not real analytic at $t=0$.

Combining $\nm{\d f}=O(e^{-\al r^2/2})$ with \eq{la3eq21} implies that
\begin{equation*}
\bnm{\d_{\cS^{m-1}}h(\bs y,t)\vert_{\cS^{m-1}\t\{t\}}}{}_{g_{\cS^{m-1}}}=O(t^{-m/2}),
\end{equation*}
where $\d_{\cS^{m-1}}$ is the derivative in the $\cS^{m-1}$ directions, and the norm is taken w.r.t.\ the metric $g_{\cS^{m-1}}$ on $\cS^{m-1}$. Hence
\e
\Bigl(\int_{\bs y\in\cS^{m-1}}\bnm{\d_{\cS^{m-1}}h(\bs y,t)}{}_{g_{\cS^{m-1}}}^2\d V_{\cS^{m-1}}\Bigr)^{1/2}=O(t^{-m/2}).
\label{la3eq26}
\e
Since $A_k^i(t)$ satisfies \eq{la3eq24} on $(0,R^{-2})$ we have $A_k^i(t)=\la_k^iA_k(t)+\mu_k^iB_k(t)$ for some $\la_k^i,\mu_k^i\in\R$. If $\mu_k^i\ne 0$ for $k>0$ then $B_k(t)\approx Ct^{d}e^{\al t^{-1}/2}$ implies that the l.h.s.\ of \eq{la3eq26} must be at least $O(t^{d}e^{\al t^{-1}/2})$ as $t\ra 0$, contradicting \eq{la3eq26}. Hence $\mu_k^i=0$ for $k>0$, and a similar argument with $\frac{\d h}{\d t}$ shows that $\mu_0^i=0$, giving $A_k^i(t)=\la_k^iA_k(t)$ for all $k,i$. Define $p_k=\sum_{i=1}^{d_k}\la_k^ip_k^i$. Then $p_k$ is harmonic on $\R^m$ and homogeneous of degree $k$, and \eq{la3eq25} implies \eq{la3eq23}.

Suppose $k(m+k-2)>3(m+1)$, so that the coefficient of $A_k$ in \eq{la3eq24} is negative. Setting $t=0$ and $A_k(0)=1$ in \eq{la3eq24} shows that $\frac{\d}{\d t}A_k(0)>0$. Thus $A_k(t),\frac{\d}{\d t}A_k(t)$ are both positive at $t=0$ and for small $t>0$ in $[0,\iy)$. Suppose for a contradiction that $t=T<\iy$ is the least $t>0$ with $\frac{\d}{\d t}A_k(T)=0$. Then $A_k(T)>0$ as $A_k(0)=1$ and $A_k$ is strictly increasing on $[0,T]$, so \eq{la3eq24} gives $\frac{\d^2}{\d t^2}A_k(T)>0$, and $A_k$ has a local minimum at $t=T$, contradicting $A_k$ strictly increasing on $[0,T]$. Thus no such $T$ exists, and $\frac{\d}{\d t}A_k(t)>0$ for all $t\in[0,\iy)$, so $A_k$ is strictly increasing on $[0,\iy)$. This completes the proof of~(ii). 

Later we will need the fact that if $k(m+k-2)>3(m+1)$ and $t\ge 0$ then
\e
0\le \ts A_k(t)^{-1}\frac{\d}{\d t}A_k(t)\le \bigl[k(m+k-2)-3(m+1)\bigr]/2\al.
\label{la3eq27}
\e
To see this, note that the left hand inequality holds as $A_k$ is increasing on $[0,\iy)$ with $A_k(0)=1$, and if the right hand inequality fails at $t\ge 0$ then \eq{la3eq24} implies that $\frac{\d^2}{\d t^2}A_k(t)<0$, so $A_k(t)^{-1}\frac{\d}{\d t}A_k(t)$ is decreasing at such $t$, but equality holds at $t=0$ in the right hand inequality.

For (iii), note that as for $k$ sufficiently large $A_k$ maps $[0,\iy)\ra[1,\iy)$ and is strictly increasing, if \eq{la3eq23} converges absolutely at $(\bs y,t)$, then it also converges absolutely at $(\bs y,t')$ for any $t'$ with $0\le t'\le t$, including $t=0$. Define $\bar h:\cS^{m-1}\t[0,R^{-2})\ra\R$ to be the limit of \eq{la3eq23}. Extending this argument to derivatives and using (ii) shows that the sum \eq{la3eq23} converges absolutely to $\bar h$ in $C^l$ for all $l\ge 0$ on any compact subset of $\cS^{m-1}\t[0,R^{-2})$. Thus $\bar h$ is smooth on $\cS^{m-1}\t[0,R^{-2})$. On $\cS^{m-1}\t(0,R^{-2})$ we have $\bar h=h$, which satisfies \eq{la3eq22}, so by continuity $\bar h$ satisfies \eq{la3eq22} on $\cS^{m-1}\t[0,R^{-2})$. Since the $A_k(t)$ are generally not real analytic at $t=0$, $\bar h$ is generally not real analytic at~$t=0$.

As $h(-,1):\cS^{m-1}\ra\R$ mapping $(y_1,\ldots,y_m)\mapsto h\bigl((y_1,\ldots,y_m),1\bigr)$ is real analytic, it extends to a holomorphic function $h(-,1)_\C:U\ra\C$ on an open neighbourhood $U$ of $\cS^{m-1}$ in its complexification $\cS^{m-1}_\C=\bigl\{(y_1,\ldots,y_m)\in\C^m:y_1^2+\cdots+y_m^2=1\bigr\}$. Making $U$ smaller if necessary, $h(-,1)_\C$ is given by the sum \eq{la2eq23} for $(y_1,\ldots,y_m)\in U$ and $t=1$, where the sum converges absolutely with all derivatives on compact subsets of $U$. Part (ii) implies that $1\le A_k(0)\le A_k(1)$ for $k\gg 0$, so absolute convergence with all derivatives of the sum \eq{la2eq23} on compact subsets of $U$ when $t=1$, implies 
absolute convergence with all derivatives of \eq{la2eq23} on compact subsets of $U$ when $t=0$. Therefore $\bar h^\iy$ extends to a holomorphic function on the open neighbourhood $U$ of $\cS^{m-1}$ in $\cS^{m-1}_\C$, and $\bar h^\iy:\cS^{m-1}\ra\R$ is real analytic.

Regard $\bar h^\iy$ as an element of $L^2(\cS^{m-1})$, and consider the complete orthogonal decomposition of $L^2(\cS^{m-1})$ into eigenspaces of the self-adjoint operator $\De_{\cS^{m-1}}$. Then $p_k\vert_{\cS^{m-1}}$ lies in the eigenspace of $\De_{\cS^{m-1}}$ with eigenvalue $k(m+k-2)$, as $p_k$ is homogeneous of degree $k$ and harmonic on $\R^m$. So \eq{la3eq23} and $A_k(0)=1$ imply that $p_k\vert_{\cS^{m-1}}$ is the component of $\bar h^\iy\in L^2(\cS^{m-1})$ in the $k(m+k-2)$-eigenspace of $\De_{\cS^{m-1}}$. Hence $\bar h^\iy$ determines the polynomials $p_k$ for $k=0,1,\ldots,$ and thus by \eq{la3eq23} and \eq{la3eq21} determines $\bar h,h$ and $f$. This proves~(iii).

For part (b), let $L$ be as in the theorem. Since $L$ has an end AC at rate $\rho<2$ to $C=\R^m$, as in the proof of Theorem \ref{la3thm4} for $R\gg 0$ there exists $f:U=\R^m\sm\,\ov{\!B}_R\ra\R$ such that $\Ga_{\d f}\subset L$. Theorem \ref{la3thm5} gives $\nm{\nabla^k\d f}=O(e^{-\al r^2/2})$ for $k=0,1,\ldots.$ Adding a constant to $f$ so that $\lim_{r\ra\iy}f=0$, the case $k=0$ implies that $\md{f}=O(e^{-\al r^2/2})$. Since $L\supset\Ga_{\d f}$ satisfies $H=\al F^\perp$, Proposition \ref{la3prop2} shows that $f$ satisfies \eq{la3eq19} for some $c\in\R$, and from $\md{f},\nm{\d f},\nm{\nabla\d f}=O(e^{-\al r^2/2})$ as $r\ra\iy$ we see that $c=0$. Also $f$ is real analytic as it satisfies a nonlinear elliptic p.d.e.\ \eq{la3eq19} with real analytic coefficients.

Now $f$ satisfies the nonlinear p.d.e.\ \eq{la3eq19} rather than the linear p.d.e.\ \eq{la3eq20}. Since $\arg\det_{\sst\mathbb C}\bigl(I+i\Hess f\bigr)$ is an odd nonlinear function of $\nabla^2f$ with linear term $\sum_{j=1}^m\frac{\pd^2f}{\pd x_j^2}$, we may rewrite \eq{la3eq19} in the form
\e
\ts\sum_{j=1}^m\frac{\pd^2f}{\pd x_j^2}+\al\bigl(\sum_{j=1}^mx_j \frac{\pd f}{\pd x_j}-2f\bigr)=Q\bigl[\bigl(\frac{\pd^2f}{\pd x_i\pd x_j}\bigr){}_{i,j=1}^m\bigr],
\label{la3eq28}
\e
where $Q$ is a real analytic nonlinear function of its arguments, and in the Taylor series of $Q$ at 0, all terms are of odd degree at least 3. As $\nm{\nabla\d f}=O(e^{-\al r^2/2})$ we see that $\bmd{Q\bigl[\bigl(\frac{\pd^2f}{\pd x_i\pd x_j}\bigr){}_{i,j=1}^m\bigr]}=O(e^{-3\al r^2/2})$ as $r\ra\iy$. Defining $h$ by \eq{la3eq21}, so that $h=O(t^{-(m+1)/2})$ for small $t$ as $f=O(e^{-\al r^2/2})$, as for \eq{la3eq22}, equation \eq{la3eq28} and this estimate on $Q$ yield, as $t\ra 0$
\e
\begin{split}
4t^2\frac{\pd^2h}{\pd t^2}&+2\bigl[\al+(m+6)t\bigr]\frac{\pd h}{\pd t}+3(m+1)h-\De_{\cS^{m-1}}h\\
&=O\bigl(t^{-(m+1)/2}e^{-\al t^{-1}}\bigr).
\label{la3eq29}
\end{split}
\e

In a similar way to \eq{la3eq25}, with the $p_k^i(\bs y)$ as above, write
\e
h(\bs y,t\bigr)=\sum_{k=0}^\iy\sum_{i=1}^{d_k}p_k^i(\bs y)\la_k^i(t)A_k(t).
\label{la3eq30}
\e
Taking the $L^2$-inner product of \eq{la3eq29} with $p_k^i(\bs y)$ and dividing by $A_k(t)\ge 1$ shows that for $t\in (0,\ha R^{-2})$ and all $i,k$ we have
\e
\begin{split}
&4t^2\frac{\d^2}{\d t^2}\la_k^i(t)+2\Bigl[\al+(m+6)t+4t^2A_k(t)^{-1}\frac{\d}{\d t}A_k(t)\Bigr]\frac{\d}{\d t}\la_k^i(t)=E_k^i(t),\\
&\text{where the error term $E_k^i(t)$ satisfies}\quad \bmd{E_k^i(t)}\le Dt^{-(m+1)/2}e^{-\al t^{-1}},
\end{split}
\label{la3eq31}
\e
for some $D>0$ independent of $i,k$. We can integrate \eq{la3eq31} explicitly: setting
\e
g_k(t)=\exp\Bigl[\,\int^t_{R^{-2}/4}\ha s^{-2}\bigl(\al+(m+6)s+4s^2A_k(s)^{-1}\ts\frac{\d}{\d s}A_k(s)\bigr)\d s\Bigr],
\label{la3eq32}
\e
where the term $A_k(s)^{-1}\ts\frac{\d}{\d s}A_k(s)$ is bounded in \eq{la3eq27}, we can rewrite \eq{la3eq31} as
\begin{equation*}
\ts\frac{\d}{\d t}\bigl[g_k(t)\frac{\d}{\d t}\la_k^i(t)\bigr]=\frac{1}{4}t^{-2}g_k(t)E_k^i(t),
\end{equation*}
so that integrating indefinitely twice gives
\e
\la_k^i(t)=\int^t\Bigl[g_k(u)^{-1}\int^u\ts\frac{1}{4}v^{-2}g_k(v)E_k^i(v)\,\d v\Bigr]\,\d u.
\label{la3eq33}
\e

Now combining equations \eq{la3eq27} and \eq{la3eq32} we find that
\e
g_k(t)=C_k t^{(m+6)/2}e^{-\al t^{-1}/2}\bigl(1+O(k^2t)\bigr),
\label{la3eq34}
\e
for some $C_k>0$. Thus
\begin{equation*}
\ts\frac{1}{4}v^{-2}g_k(v)E_k^i(v)=O(v^{1/2}e^{-3\al v^{-1}/2})\cdot \bigl[1+O(k^2v)\bigr],
\end{equation*}
and integrating gives
\e
\int^u\ts\frac{1}{4}v^{-2}g_k(v)E_k^i(v)\d v=c_k^i+O(u^{5/2}e^{-3\al u^{-1}/2})\cdot \bigl[1+O(k^2u)\bigr],
\label{la3eq35}
\e
for some $c_k^i\in\R$. If $c_k^i\ne 0$ then $\la_k^i(t)\approx c_k^iC_k^{-1} t^{-(m+2)/2}e^{\al t^{-1}/2}$, contradicting $h=O(t^{-(m+1)/2})$, so $c_k^i=0$. Hence equations \eq{la3eq33}--\eq{la3eq35} yield $\la_k^i(t)=\La_k^i+\mu_k^i(t)$ for some $\La_k^i\in\R$, and
\e
\mu_k^i(t)=\int_0^t\Bigl[g_k(u)^{-1}\int_0^u\ts\frac{1}{4}v^{-2}g_k(v)E_k^i(v)\,\d v\Bigr]\,\d u
\label{la3eq36}
\e
is well-defined and depends only on $k,E_k^i(t)$, with 
\e
\mu_k^i(t)=O(t^{(3-m)/2}e^{-\al t^{-1}})\bigl[1+O(k^2t)\bigr].
\label{la3eq37}
\e

By definition, $E_k^i(t)A_k(t)$ is the component of \eq{la3eq29} in the direction of the $L^2$-orthonormal basis element $p^i_k(\bs y)$ in $L^2(\cS^{m-1})$. So squaring and summing over $i,k$, and integrating the squared r.h.s.\ of \eq{la3eq29} over $\cS^{m-1}$, gives 
\e
\ts\sum\limits_{k=0}^\iy\sum\limits_{i=1}^{d_k}\bigl(E_k^i(t)A_k(t)\bigr)^2\le C_0t^{-(m+1)}e^{-2\al t^{-1}}
\label{la3eq38}
\e
for some $C_0>0$ and all $t\in(0,\ha R^{-2})$. More generally, applying $\De_{\cS^{m-1}}^{n/2}$ to \eq{la3eq29} and estimating using $\nm{\nabla^n\d f}=O(e^{-\al r^2/2})$ yields
\e
\ts\sum\limits_{k=0}^\iy\sum\limits_{i=1}^{d_k}k^n\bigl(E_k^i(t)A_k(t)\bigr)^2\le C_nt^{-(m+1)-2n}e^{-2\al t^{-1}}
\label{la3eq39}
\e
for some $C_n>0$ and all $t\in(0,\ha R^{-2})$.

Let us rewrite \eq{la3eq30} as
\e
h(\bs y,t\bigr)=\sum_{k=0}^\iy\sum_{i=1}^{d_k}p_k^i(\bs y)\La_k^iA_k(t)+\sum_{k=0}^\iy\sum_{i=1}^{d_k}p_k^i(\bs y)\mu_k^i(t)A_k(t).
\label{la3eq40}
\e
From equations \eq{la3eq36}--\eq{la3eq39} and estimates on higher derivatives we find that the second sum in \eq{la3eq40} converges absolutely in $C^l$ for all $l\ge 0$ on any compact subset of $\cS^{m-1}\t[0,R^{-2}),$ and the result is $O(t^de^{-\al t^{-1}})$, and so is zero at $t=0$. Since \eq{la3eq30} converges absolutely in $C^l$ for all $l\ge 0$ on any compact subset of $\cS^{m-1}\t(0,R^{-2})$, it follows that the first sum in \eq{la3eq40} converges absolutely in $C^l$ for all $l\ge 0$ on any compact subset of $\cS^{m-1}\t(0,R^{-2})$. As $h=O(t^{-(m+1)/2})$ and the second sum is $O(t^de^{-\al t^{-1}})$, the first sum in \eq{la3eq40} is $O(t^{-(m+1)/2})$.

Now the first sum in \eq{la3eq40} is a linear combination of solutions to \eq{la3eq22}, so as the sum converges locally in $C^2$, it satisfies \eq{la3eq22} on $\cS^{m-1}\t(0,R^{-2})$. Thus we may apply part (a) to the first sum in \eq{la3eq40}. This tells us that the first sum converges to a smooth solution on \eq{la3eq22} on $\cS^{m-1}\t[0,R^{-2})$, which is real analytic on $\cS^{m-1}\t(0,R^{-2})$ and $\cS^{m-1}\t\{0\}$. Since the second sum is zero on $\cS^{m-1}\t\{0\}$, part (b) follows. This completes the proof of Theorem~\ref{la3thm6}.
\end{proof}

\begin{rem} Theorem \ref{la3thm4} implies that if $L$ is an AC SL $m$-fold asymptotic with rate $\rho<0$ to $C=\R^m$ in $\C^m$, then $L$ may be written as a graph $\Ga_{\d f}$ near infinity in $C$, where as in \eq{la3eq25} $f$ has an asymptotic expansion
\e
f(x_1,\ldots,x_m)=\sum_{k=0}^\iy\sum_{i=1}^{d_k}p_k^i\vert_{\cS^{m-1}}(x_1/r,\ldots,x_m/r)A_k^i(r)
\label{la3eq41}
\e
in terms of the eigenfunctions $p_k^i\vert_{\cS^{m-1}}$ of $\De_{\cS^{m-1}}$ on the link $\cS^{m-1}$ of the cone $C=\R^m$, with~$A_k^i(r)=O(r^{2-m-k})$. 

In particular, as the eigenvalue $k(m+k-2)$ of the eigenfunction $p_k^i\vert_{\cS^{m-1}}$ increases, the functions $A_k^i(r)=O(r^{2-m-k})$ get smaller as $r\ra\iy$, so that the asymptotic behaviour of $L$ is dominated by eigenfunctions of $\De_{\cS^{m-1}}$ with small eigenvalues. This is what one would expect from the usual theory of elliptic operators on manifolds with conical ends, as in Lockhart \cite{Lock}, for instance.

Surprisingly, Theorem \ref{la3thm6} shows that {\it AC LMCF expanders do not behave this way}. Instead, in the corresponding expansion \eq{la3eq41} we have $A_k^i(r)\sim C^i_kr^{-1-m}e^{-\al r^2/2}$ to leading order for all $k,i$. That is, {\it the contributions from different eigenfunctions of\/ $\De_{\cS^{m-1}}$ are all roughly of the same size}, and {\it small eigenvalues of\/ $\De_{\cS^{m-1}}$ do not dominate the asymptotic behaviour}.
\label{la3rem}
\end{rem}

\subsection{An explicit family of LMCF expanders}
\label{la36}

Joyce, Lee and Tsui \cite{JLT} construct many solitons for LMCF by
an ansatz using `evolving quadrics', generalizing the method of
Joyce \cite{Joyc1} for special Lagrangians. The next example
describes a family of LMCF expanders constructed in \cite[Th.s C \&
D]{JLT}, which generalize the `Lawlor necks' of
Example~\ref{la3ex1}.

\begin{ex} Let $m>2$, $\al\ge 0$ and $a_1,\ldots,a_m>0$, and define
a smooth function $P:\R\ra\R$ by $P(0)=\al+a_1+\cdots+a_m$ and
\e
P(x)=\ts\frac{1}{x^2}\bigl(e^{\al x^2}\prod_{k=1}^m(1+a_kx^2)-1\bigl),\quad x\ne 0.
\label{la3eq42}
\e
Define real numbers $\phi_1,\ldots,\phi_m$ by
\e
\phi_k=a_k\int_{-\iy}^\iy\frac{\d x}{(1+a_kx^2)\sqrt{P(x)}}\,,
\label{la3eq43}
\e
For $k=1,\ldots,m$ define a function $z_k:\R\ra\C$ by
\e
z_k(y)={\rm e}^{i\psi_k(y)}\sqrt{a_k^{-1}+y^2}, \;\>\text{where}\;\>
\psi_k(y)=a_k\int_{-\iy}^y\frac{\d x}{(1+a_kx^2)\sqrt{P(x)}}\,.
\label{la3eq44}
\e

Now write ${\bs\phi}=(\phi_1,\ldots,\phi_m)$, and define a
submanifold $L_{\bs\phi}^\al$ in $\C^m$ by
\begin{equation*}
L_{\bs\phi}^\al=\bigl\{(z_1(y)x_1,\ldots,z_m(y)x_m): y\in\R,\;
x_k\in\R,\; x_1^2+\cdots+x_m^2=1\bigr\}.
\end{equation*}
Then $L_{\bs\phi}^\al$ is a closed, embedded Lagrangian
diffeomorphic to $\cS^{m-1}\t\R$ and satisfying $H=\al F^\perp$. If
$\al>0$ it is an LMCF expander, and one can show it fits into the framework of \S\ref{la35}, with $\bar h^\iy$ a quadratic form on $\cS^{m-1}$. If $\al=0$ it is a Lawlor neck $L_{{\bs\phi},A}$ from Example \ref{la3ex1}. It is graded, with Lagrangian angle
\e
\th_{L_{\bs\phi}^\al}\bigl((z_1(y)x_1,\ldots,z_m(y)x_m)\bigr)
=\ts\sum_{k=1}^m\psi_k(y)+\arg\bigl(-y-iP(y)^{-1/2}\bigr).
\label{la3eq45}
\e
Note that the only difference between the constructions of
$L_{{\bs\phi},A}$ in Example \ref{la3ex1} and $L_{\bs\phi}^\al$
above is the term $e^{\al x^2}$ in \eq{la3eq42}, which does not
appear in \eq{la3eq11}. If $\al=0$ then $e^{\al x^2}=1$, and the two
constructions agree.

As in \cite[Th.~D]{JLT}, $L_{\bs\phi}^\al$ is Asymptotically
Conical, with cone $C$ the union $\Pi_0\cup\Pi_{\bs\phi}$ of two
Lagrangian $m$-planes $\Pi_0,\Pi_{\bs\phi}$ in $\C^m$ given by
\begin{equation*}
\Pi_0=\bigl\{(x_1,\ldots,x_m):x_j\in\R\bigr\},\;\>
\Pi_{\bs\phi}=\bigl\{({\rm e}^{i\phi_1}x_1,\ldots, {\rm
e}^{i\phi_m}x_m):x_j\in\R\bigr\}.
\end{equation*}
But in contrast to Example \ref{la3ex1}, for $\al>0$ we do not have $\phi_1+\cdots+\phi_m=\pi$, so $\Pi_{\bs\phi}$ and $C$ are not special Lagrangian. Note that we can apply Theorem \ref{la3thm6} in \S\ref{la35} to describe $L_{\bs\phi}^\al$ near infinity in $\Pi_0=\R^m\subset\C^m$, and also (after applying a $\U(m)$ rotation taking $\Pi_{\bs\phi}$ to $\R^m$) near infinity in~$\Pi_{\bs\phi}$.

In \cite[Th.~D]{JLT} Joyce, Lee and Tsui prove that for fixed $\al>0$, the map
$\Phi^\al:(a_1,\ldots,a_m) \mapsto(\phi_1,\ldots,\phi_m)$ gives a
diffeomorphism
\begin{equation*}
\Phi^\al:(0,\iy)^m\longra\bigl\{(\phi_1,\ldots,\phi_m)\in(0,\pi)^m:
0<\phi_1+\cdots+\phi_m<\pi\bigr\}.
\end{equation*}
That is, for all $\al>0$ and ${\bs\phi}=(\phi_1,\ldots,\phi_m)$ with
$0<\phi_1,\ldots,\phi_m<\pi$ and $0<\phi_1+\cdots+\phi_m<\pi$, the
above construction gives a unique LMCF expander
$L_{\bs\phi}^\al$ asymptotic to~$\Pi_0\cup\Pi_{\bs\phi}$.

From \eq{la3eq43}--\eq{la3eq44} we see that $\lim_{y\ra
-\iy}\psi_k(y)=0$ and $\lim_{y\ra\iy}\psi_k(y)=\phi_k$. Also
$\lim_{y\ra-\iy}\arg\bigl(-y-iP(y)^{-1/2}\bigr)=0$ and
$\lim_{y\ra+\iy}\arg\bigl(-y-iP(y)^{-1/2}\bigr)=-\pi$. Hence
\eq{la3eq45} implies that
\e
\begin{split}
\lim_{y\ra -\iy}\th_{L_{\bs\phi}^\al}\bigl((z_1(y)x_1,
\ldots,z_m(y)x_m)\bigr)&=0,\\
\lim_{y\ra \iy}\th_{L_{\bs\phi}^\al}\bigl((z_1(y)x_1,
\ldots,z_m(y)x_m)\bigr)&=\phi_1+\cdots+\phi_m-\pi.
\end{split}
\label{la3eq46}
\e
Thus $\th_{L_{\bs\phi}^\al}\ra 0$ on the end of $L_{\bs\phi}^\al$
asymptotic to $\Pi_0$, and $\th_{L_{\bs\phi}^\al}\ra\phi_1+\cdots+
\phi_m-\pi$ on the end of $L_{\bs\phi}^\al$ asymptotic
to~$\Pi_{\bs\phi}$.

Now $L_{\bs\phi}^\al$ is an AC Lagrangian in $\C^m$ with rate
$\rho<0$ and cone $C=\Pi_0\cup\Pi_{\bs\phi}$, and it is connected
and exact, with potential $f_{L_{\bs\phi}^\al}=
-2\th_{L_{\bs\phi}^\al}/\al$ by \eq{la3eq16}. Also $y\ra -\iy$ is
the limit $r\ra\iy$ in the end of $L_{\bs\phi}^\al$ asymptotic to
$\Pi_0$, and $y\ra \iy$ the limit $r\ra\iy$ in the end of
$L_{\bs\phi}^\al$ asymptotic to $\Pi_{\bs\phi}$. So by \eq{la3eq46}
in Definition \ref{la2def6} we have $c_0=0$ and $c_{\bs\phi}=
2(\pi-\phi_1-\cdots-\phi_m)/\al$, and
\e
A(L_{\bs\phi}^\al)=2(\pi-\phi_1-\cdots-\phi_m)/\al>0.
\label{la3eq47}
\e
We will see later that $A(L_{\bs\phi}^\al)$ is the area of a certain
$J$-holomorphic curve in~$\C^m$.

Now suppose $\phi_1,\ldots,\phi_m\in(0,\pi)$ with
$(m-1)\pi<\phi_1+\cdots+\phi_m<m\pi$. Then $(\pi-\phi_1),\ldots,
(\pi-\phi_m)\in(0,\pi)$ with $0<(\pi-\phi_1)+\cdots+
(\pi-\phi_m)<\pi$, so as above we have an LMCF expander
$L_{\bs\pi-\bs\phi}^\al$ for each $\al>0$, where
$\bs\pi-\bs\phi=(\pi-\phi_1,\ldots,\pi-\phi_m)$. As in \eq{la3eq15},
define
\begin{equation*}
\ti L_{\bs\phi}^\al=\begin{pmatrix} e^{i\phi_1} & 0 & \cdots & 0
\\ 0 & e^{i\phi_2} & \cdots & 0 \\ \vdots & \vdots & \ddots & \vdots \\
0 & 0 & \cdots & e^{i\phi_m} \end{pmatrix}L_{\bs\pi-\bs\phi}^\al.
\end{equation*}
As for $\ti L_{{\bs\phi},-A}$ in Example \ref{la3ex1}, $\ti
L_{\bs\phi}^\al$ is an LMCF expander satisfying $H=\al F^\perp$ and
asymptotic to $\Pi_0\cup\Pi_{\bs\phi}$, with
\begin{equation*}
A(\ti L_{{\bs\phi},-A})=2((m-1)\pi-\phi_1-\cdots-\phi_m)/\al<0.
\end{equation*}
There is a unique grading $\th_{\ti L_{\bs\phi}^\al}$ on $\ti
L_{\bs\phi}^\al$ with $\th_{\ti L_{\bs\phi}^\al}\ra 0$ on the end of
$\ti L_{\bs\phi}^\al$ asymptotic to $\Pi_0$, and $\th_{\ti
L_{\bs\phi}^\al}\ra\phi_1+\cdots+\phi_m-(m-1)\pi$ on the end
asymptotic to~$\Pi_{\bs\phi}$.

Observe that $L_{\bs\phi}^\al$ is exact and graded with
$\th_{L_{\bs\phi}^\al}\ra 0$ on the end of $L_{\bs\phi}^\al$
asymptotic to $\Pi_0$ and $\th_{L_{\bs\phi}^\al}\ra\phi_1+\cdots+
\phi_m-\pi$ on the end asymptotic to $\Pi_{\bs\phi}$, and so
satisfies Corollary \ref{la2cor}(a). Similarly, $\ti
L_{\bs\phi}^\al$ is exact and graded with $\th_{\ti
L_{\bs\phi}^\al}\ra 0$ on the end of $\ti L_{\bs\phi}^\al$
asymptotic to $\Pi_0$ and $\th_{\ti L_{\bs\phi}^\al}\ra
\phi_1+\cdots+\phi_m-(m-1)\pi$ on the end asymptotic
to~$\Pi_{\bs\phi}$, and so satisfies Corollary~\ref{la2cor}(b).
\label{la3ex2}
\end{ex}

\section{Uniqueness theorems}
\label{la4}

The following two theorems are our main results, which imply
Theorems \ref{la1thm1} and \ref{la1thm2}. We will prove them in
\S\ref{la46}--\S\ref{la47}, after some preparation
in~\S\ref{la41}--\S\ref{la45}.

\begin{thm} Let\/ $m\ge 3$ and\/ $\phi_1,\ldots,\phi_m\in(0,\pi)$
with\/ $\phi_1+\cdots+\phi_m=\pi,$ and define special Lagrangian
$m$-planes $\Pi_0,\Pi_{\bs\phi}$ in $\C^m$ by
\e
\Pi_0=\bigl\{(x_1,\ldots,x_m):x_j\in\R\bigr\},\;\>
\Pi_{\bs\phi}=\bigl\{({\rm e}^{i\phi_1}x_1,\ldots, {\rm
e}^{i\phi_m}x_m):x_j\in\R\bigr\}.
\label{la4eq1}
\e
Suppose $L$ is a closed, embedded, exact, Asymptotically Conical
special Lagrangian in $\C^m$ asymptotic with rate $\rho<0$ to
$C=\Pi_0\cup\Pi_{\bs\phi}$. Then $L$ is the `Lawlor neck'
$L_{\bs\phi,A}$ from Example\/ {\rm\ref{la3ex1},} for some
unique~$A>0$.

If instead\/ $L$ has rate $\rho<2$ rather than $\rho<0,$ then $L$ is
a translation $L_{\bs\phi,A}+\bs c$ of\/ $L_{\bs\phi,A},$ for some
unique $A>0$ and\/ $\bs c=(c_1,\ldots,c_m)\in\C^m$.

If instead\/ $\phi_1+\cdots+\phi_m=(m-1)\pi,$ then $L=\ti
L_{\bs\phi,A}$ from Example\/ {\rm\ref{la3ex1}} for some $A<0$ if\/
$\rho<0,$ and\/ $L=\ti L_{\bs\phi,A}+\bs c$ for $\rho<2$. If\/
$\phi_1+\cdots+\phi_m=k\pi$ for $1<k<m-1,$ there exist no closed,
embedded, exact SL $m$-folds in $\C^m$ asymptotic to
$C=\Pi_0\cup\Pi_{\bs\phi}$ with rate $\rho<2$.

If instead\/ $L$ is immersed rather than embedded, the only extra
possibilities are $L=\Pi_0\cup\Pi_{\bs\phi}$ for $\rho<0,$ and\/
$L=\Pi_0\cup\Pi_{\bs\phi}+\bs c$ for $\rho<2$.
\label{la4thm1}
\end{thm}

\begin{thm} Let\/ $m\ge 3$ and\/ $\phi_1,\ldots,\phi_m\in(0,\pi)$
with\/ $0<\phi_1+\cdots+\phi_m<\pi,$ and define Lagrangian
$m$-planes $\Pi_0,\Pi_{\bs\phi}$ in $\C^m$ by \eq{la4eq1}. Suppose
$L$ is a closed, embedded, exact, Asymptotically Conical Lagrangian
MCF expander in $\C^m$ asymptotic with rate $\rho<2$ to
$C=\Pi_0\cup\Pi_{\bs\phi},$ satisfying the expander equation $H=\al
F^\perp$ for $\al>0$. Then $L$ is the LMCF expander
$L_{\bs\phi}^\al$ found by Joyce, Lee and Tsui\/ {\rm\cite[Th.s C \& D]{JLT},} described in Example\/ {\rm\ref{la3ex2}}.

If instead\/ $(m-1)\pi<\phi_1+\cdots+\phi_m<m\pi,$ then $L$ is $\ti
L_{\bs\phi}^\al$ from Example\/ {\rm\ref{la3ex2}}. There exist no
closed, embedded, exact, AC LMCF expanders asymptotic to
$C=\Pi_0\cup\Pi_{\bs\phi}$ when $\pi\le\phi_1+\cdots+\phi_m\le
(m-1)\pi$.

If instead\/ $L$ is immersed rather than embedded, the only extra
possibility is~$L=\Pi_0\cup\Pi_{\bs\phi}$.
\label{la4thm2}
\end{thm}

\subsection{Outline of the method of proof}
\label{la41}

We begin with two elementary propositions on intersection properties
of SL $m$-folds and graded LMCF expanders. Readers are advised to compare
Proposition \ref{la4prop1} with Theorem \ref{la2thm1}(a), and
Proposition \ref{la4prop2} with Theorem~\ref{la2thm1}(b).

\begin{prop} Let\/ $(M,J,g,\Om)$ be a Calabi--Yau $m$-fold, and\/
$L,L'$ be transversely intersecting SL\/ $m$-folds in $M$. Then
$0<\mu_{L,L'}(p)<m$ for all\/ $p\in L\cap L'$.

In particular, there are no $p\in L\cap L'$ with\/ $\mu_{L,L'}(p)=0$
or $m$.
\label{la4prop1}
\end{prop}

\begin{proof} This follows from equation \eq{la2eq7} and $\th_L=\th_{L'}=0$.
\end{proof}

\begin{prop} Let\/ $L,L'$ be transversely intersecting, graded LMCF
expanders in $\C^m,$ both satisfying the expander equation $H=\al
F^\perp$ with the same value of\/ $\al>0$. Then
\begin{itemize}
\setlength{\parsep}{0pt}
\setlength{\itemsep}{0pt}
\item[{\bf(a)}] As in {\rm\S\ref{la34}} $L,L'$ are exact, with potentials $f_L=-2\th_L/\al$ and\/ $f_{L'}=-2\th_{L'}/\al$. Then for any $p\in L\cap L'$ we have
\e
\frac{\al}{2\pi}(f_{L'}(p)-f_L(p))<\mu_{L,L'}(p)<
\frac{\al}{2\pi}(f_{L'}(p)-f_L(p))+m.
\label{la4eq2}
\e
\item[{\bf(b)}] Suppose $J$ is any almost complex structure on $\C^m$ compatible with\/ $\om,$ and\/ $\Si$ is a $J$-holomorphic disc in $\C^m$ with boundary in $L\cup L'$ and corners at\/ $p,q\in L\cap L',$ of the form shown in Figure\/ {\rm\ref{la4fig1}}.
\begin{figure}[htb]
\centerline{$\splinetolerance{.8pt}
\begin{xy}
0;<1mm,0mm>:
,(-20,0);(20,0)**\crv{(0,10)}
?(.95)="a"
?(.85)="b"
?(.75)="c"
?(.65)="d"
?(.55)="e"
?(.45)="f"
?(.35)="g"
?(.25)="h"
?(.15)="i"
?(.05)="j"
?(.5)="y"
,(-20,0);(-30,-6)**\crv{(-30,-5)}
,(20,0);(30,-6)**\crv{(30,-5)}
,(-20,0);(20,0)**\crv{(0,-10)}
?(.95)="k"
?(.85)="l"
?(.75)="m"
?(.65)="n"
?(.55)="o"
?(.45)="p"
?(.35)="q"
?(.25)="r"
?(.15)="s"
?(.05)="t"
?(.5)="z"
,(-20,0);(-30,6)**\crv{(-30,5)}
,(20,0);(30,6)**\crv{(30,5)}
,"a";"k"**@{.}
,"b";"l"**@{.}
,"c";"m"**@{.}
,"d";"n"**@{.}
,"e";"o"**@{.}
,"f";"p"**@{.}
,"g";"q"**@{.}
,"h";"r"**@{.}
,"i";"s"**@{.}
,"j";"t"**@{.}
,"y"*{<}
,"z"*{>}
,(-20,0)*{\bu}
,(-20,-3)*{p}
,(20,0)*{\bu}
,(20,-3)*{q}
,(0,0)*{\Si}
,(-32,4)*{L}
,(-32,-4)*{L'}
,(32,4)*{L}
,(32.5,-4)*{L'}
\end{xy}$}
\caption{Holomorphic disc $\Si$ with boundary in $L\cup L'$}
\label{la4fig1}
\end{figure}
Then $\mu_{L,L'}(q)-\mu_{L,L'}(p)<m$. In particular, we cannot have
$\mu_{L,L'}(p)=0$ and\/ $\mu_{L,L'}(q)=m$.
\end{itemize}

\label{la4prop2}
\end{prop}

\begin{proof} Part (a) follows by substituting $f_L=-2\th_L/\al$ and\/ $f_{L'}=-2\th_{L'}/\al$ into equation \eq{la2eq7}. For (b), by equation \eq{la4eq2} for $p,q$ we have
\begin{align*}
&\mu_{L,L'}(q)-\mu_{L,L'}(p)<\frac{\al}{2\pi}(f_{L'}(q)-f_L(q))+m
-\frac{\al}{2\pi}(f_{L'}(p)-f_L(p))\\
&=m-\frac{\al}{2\pi}\bigl(f_L(q)-f_L(p)+f_{L'}(p)-f_{L'}(q)\bigr)
=m-\frac{\al}{2\pi}\mathop{\rm area}(\Si)<m,
\end{align*}
as $\area(\Si)>0$.
\end{proof}

We now sketch a general method for proving uniqueness of SL $m$-folds or Lagrangian MCF expanders. For SL $m$-folds, our method is based on Thomas and Yau \cite[Th.~4.3]{ThYa}, who prove that (under some extra assumptions we will not give) a Hamiltonian deformation class of compact Lagrangians $L$ in a compact Calabi--Yau $m$-fold $(M,J,g,\Om)$ can contain at most one special Lagrangian.

Let $\scr C$ be some class of SL $m$-folds $L$ in a Calabi--Yau manifold $(M,J,g,\Om)$ that we want to classify (in our case $M=\C^m$ and $\scr C$ is the class of all exact AC SL $m$-folds in $\C^m$ with cone $C=\Pi_0\cup\Pi_{\bs\phi}$, for fixed $\bs\phi$) or some class of LMCF expanders $L$ in $\C^m$ that we want to classify (in our case $\scr C$ is the class of all exact AC LMCF expanders $L$ in $\C^m$ with cone $C=\Pi_0\cup\Pi_{\bs\phi}$).

Suppose we know some class of examples $\scr D\subseteq\scr C$ (in our case $\scr D$ is the Lawlor necks \cite{Lawl} or Joyce--Lee--Tsui LMCF expanders \cite{JLT}) and we want to prove that these are the only examples, that is, $\scr C=\scr D$. Then we follow these steps:
\begin{itemize}
\setlength{\parsep}{0pt}
\setlength{\itemsep}{0pt}
\item[(A)] Choose some class of cohomological or analytic
invariants $I(L)$ of objects $L$ in $\scr C$ which distinguish
objects in $\scr D$, that is, if $L,L'\in\scr D$ with
$I(L)=I(L')$ then $L=L'$. These invariants may be based on the
relative cohomology classes $[\om]\in H^2(M;L,\R)$ or
$[\Im\Om]\in H^m(M;L,\R)$. For noncompact SL $m$-folds, $I(L)$
may describe the leading order asymptotic behaviour of $L$ at
infinity. In our case we take $I(L)=A(L)$ from
Definition~\ref{la2def6}.
\item[(B)] We need to be in a situation in which we can apply
Lagrangian Floer cohomology and Fukaya categories, in particular
Theorem \ref{la2thm1}. Ideally, we would like to be working with
compact, exact, graded Lagrangians in a symplectic Calabi--Yau
Liouville manifold.

If this does not apply already, look for some partial
compactification $\bM$ of $M$ or $\C^m$, such that $\bM$
is a symplectic Calabi--Yau Liouville manifold, and objects $L$
in $\scr C$ extend naturally to compact, exact, graded
Lagrangians $\bar L$ in $\bM$. (In our case $\bM=T^*\cS^m\# T^*\cS^m$ from~\S\ref{la24}.)
\item[(C)] Suppose for a contradiction that $L,L'\in\scr C$
with $I(L)=I(L')$, but $L\ne L'$. We have compactifications
$\bar L,\bar L'$ in $\bM$. By using results on the derived
Fukaya category $D^b\sF(\bM)$, prove that $\bar L,\bar L'$
are isomorphic in~$D^b\sF(\bM)$.
\item[(D)] If $L,L'$ intersect transversely in $M$, then
Proposition \ref{la4prop1} or \ref{la4prop2} applies to $L,L'$.
Note however that $\bar L,\bar L'$ may not intersect
transversely at points in $\bM\sm M$, and even if they do,
the conclusions of Proposition \ref{la4prop1} or \ref{la4prop2}
may not hold there. Let $\bar L''$ be any small Hamiltonian perturbation of $\bar L'$ which is sufficiently $C^1$-close to $\bar L'$, intersects $\bar L$ transversely, and satisfies some extra asymptotic conditions near $\bM\sm M$. We prove that $\bar L$ and $\bar L''$ satisfy the conclusions of Proposition \ref{la4prop1} or \ref{la4prop2} in~$M$.

We also use the extra conditions to control  
$\mu_{\bar L,\bar L''}(p)$ or $f_{\bar L}(p)-f_{\bar L''}(p)$ for $p\in (\bar L\cap\bar L'')\sm M$. This may use the
assumption~$I(L)=I(L')$.

\item[(E)] By (B),(C), we have $\bar L,\bar L''$ in $\bM$
which are transversely intersecting, isomorphic in $D^b\sF(\bM)$, and satisfy the conclusions of Proposition \ref{la4prop1}
or \ref{la4prop2}. Combining Theorem \ref{la2thm1}(a) or (b) and
Proposition \ref{la4prop1} or \ref{la4prop2}, together with a special argument for the case when $p$ or $q$ lies in $(\bar L\cap\bar L'')\sm M$, now gives a contradiction. Hence $I(L)=I(L')$ implies $L=L'$ in~$\scr C$.
\item[(F)] Prove that there do not exist $L\in\scr C$ with
$I(L)\notin I(\scr D)$. (In our case we will show that $L$ in
$\scr C$ implies that $A(L)>0$ in case Corollary
\ref{la2cor}(a), and $A(L)<0$ in case Corollary
\ref{la2cor}(b)). Then (C)--(E) imply that~$\scr C=\scr D$.
\end{itemize}

The rest of this section follows this method to prove Theorems
\ref{la4thm1} and~\ref{la4thm2}.

\subsection{The immersed case, and rates $\rho\in[0,2)$}
\label{la42}

The next two propositions will be used to prove the immersed cases
in Theorems \ref{la4thm1} and \ref{la4thm2}, and to bridge the gap
between rates $\rho<0$ and $\rho<2$ in Theorem~\ref{la4thm1}.

\begin{prop} Suppose $L$ is a closed AC SL\/ $m$-fold or LMCF
expander in $\C^m$ which is asymptotic at rate $\rho<0$ to
$C=\Pi_0\cup\Pi_{\bs\phi},$ and that\/ $L$ is immersed not embedded.
Then $L=\Pi_0\cup\Pi_{\bs\phi}$.
\label{la4prop3}
\end{prop}

\begin{proof} Lotay and Neves \cite[Lem.~2.4]{LoNe} prove the LMCF
expander case of this using Huisken's monotonicity formula.
Essentially the same proof also works in the special Lagrangian
case.
\end{proof}

The proof of the next result is taken from Imagi~\cite[Th.~6]{Imag}:

\begin{prop} Suppose $L$ is a closed AC SL\/ $m$-fold in $\C^m$
which is asymptotic at rate $\rho<2$ to $C=\Pi_0\cup\Pi_{\bs\phi}$.
Then there exists a unique $\bs c=(c_1,\ldots,c_m)$ in $\C^m$ such
that\/ $L'=L-\bs c$ is asymptotic to $C$ at rate\/~$\rho'<0$.
\label{la4prop4}
\end{prop}

\begin{proof} In the notation of Joyce \cite[\S 6]{Joyc6}, the SL cone $C$ has link $\Si=S^{m-1}\amalg S^{m-1}$. The set of `critical rates' is $\cD_\Si=\{0,1,2,\ldots\}\amalg \{2-m,1-m,-m,\ldots\}$.
Joyce \cite[Th.~6.6]{Joyc6} proves that if $\rho,\ti\rho<2$ lie in the same connected component of $\R\sm\cD_\Si$ and $L$ is AC with rate $\rho$, then $L$ is also AC with rate $\ti\rho$. Hence $L$ is AC with rate $\ti\rho$ for every $\ti\rho\in(1,2)$. 

The obstruction to crossing rate $1$ comes from homogeneous harmonic functions of degree 1 on $\Pi_0$ or $\Pi_{\bs\phi}$. Homogeneous harmonic functions on Euclidean spaces are polynomials. All homogeneous harmonic functions of degree 1 on a Lagrangian plane in $\C^m$ are restrictions of moment maps of translations on~$\C^m$.

From the proof of \cite[Th.~6]{Imag}, there is a unique $\bs c\in\C^m$ such that $L'=L-\bs c$ is AC with rate $\rho'\in(0,1)$.
The obstruction to crossing rate 0 comes from a cohomology class $Y(L')\in H^1(\Si)$ in the notation of Joyce\cite[Def.~6.2]{Joyc6}, but $H^1(\Si)=0$ as $\Si=S^{m-1}\amalg S^{m-1}$ for $m>2$, so $Y(L')=0$. Thus $L'=L-\bs c$ is AC with any rate $\rho'\in(2-m,0)$, completing the proof.
\end{proof}

\subsection{Perturbing $\bar L,\bar L'$ to transverse: special Lagrangian case}
\label{la43}

Next, in a rather long proof, we show that we can perturb compactifications $\bar L,\bar L'$ of AC SL $m$-folds $L,L'$ to intersect transversely, whilst preserving the property that the intersections of $L$ and $L'$ away from infinity have indices strictly between 0 and $m$. We do introduce intersection points of index 0 and $m$ at infinity, and we need an additional argument with a holomorphic disc (guaranteed by Theorem \ref{la2thm1}(b)) to get a contradiction. Readers happy to believe Proposition \ref{la4prop5} can skip forward to \S\ref{la44} or~\S\ref{la45}.

The perturbation $\bar L''$ we choose in the proposition is not special or difficult to define --- away from $\iy_0,\iy_{\bs\phi}$, any generic, sufficiently small Hamiltonian perturbation will do. But the proof that it has the properties (a)--(c) we need is subtle, and uses the fact that $L,L'$ are {\it real analytic}, including the description in Theorem \ref{la3thm4} of $L,L'$ near infinity in terms of a real analytic function, and the ``Taylor's Theorem'' for real analytic functions in Theorem~\ref{laAthm1}.

One could also use Morse-theoretic ideas to construct the perturbation $\bar L''$, as in \S\ref{laA2} and the proof of \cite[Th.~4.3]{ThYa}, although there would be some difficult issues to deal with. We chose the proof below as it shows that any small generic perturbation will do, rather than a very carefully chosen perturbation, and it showcases some new methods involving real analytic functions.

\begin{prop} Suppose $L,L'$ are distinct, closed, exact AC SL\/
$m$-folds in $\C^m$ for $m\ge 3$ asymptotic at rate $\rho<0$ to the cone
$C=\Pi_0\cup\Pi_{\bs\phi}$. Then {\rm\S\ref{la24}} defines a
symplectic Calabi--Yau Liouville manifold $(M,\om),\la$ such that\/
$L,L'$ extend naturally to compact, exact, graded Lagrangians $\bar
L=L\cup\{\iy_0,\iy_{\bs\phi}\}$ and\/ $\bar
L'=L'\cup\{\iy_0,\iy_{\bs\phi}\}$ in $M,$ with potentials $f_{\bar
L},f_{\bar L'}$ which by \eq{la2eq15} satisfy
\begin{equation*}    
f_{\bar L}(\iy_0)=f_{\bar L'}(\iy_0)=0,\quad
f_{\bar L}(\iy_{\bs\phi})=A(L),\quad f_{\bar L'}(\iy_{\bs\phi})=A(L').
\end{equation*}

Then we may choose a small Hamiltonian perturbation $\bar L''=L''\cup\{\iy_0,\iy_{\bs\phi}\}$ of\/ $\bar L'$ in\/ $(M,\om),$ with potential\/ $f_{\bar L''},$ satisfying the conditions:
\begin{itemize}
\setlength{\parsep}{0pt}
\setlength{\itemsep}{0pt}
\item[{\bf(a)}] $\bar L$ and\/ $\bar L''$ intersect
transversely in $M$.
\item[{\bf(b)}] Each\/ $p\in \bar L\cap \bar L''\cap\C^m$ has
$0<\mu_{\bar L,\bar L''}(p)<m$.
\item[{\bf(c)}] We also have $\iy_0,\iy_{\bs\phi}\in\bar
L\cap \bar L'',$ and\/ $\mu_{\bar L,\bar L''}(\iy_0),\mu_{\bar L,\bar L''}(\iy_{\bs\phi})\in\{0,m\}$. The potential\/ $f_{\bar L''}$ of\/ $\bar L''$ satisfies
\e
f_{\bar L''}(\iy_0)=0\quad\text{and\/}\quad f_{\bar
L''}(\iy_{\bs\phi})=A(L').
\label{la4eq3}
\e
\end{itemize}
The analogues of\/ {\bf(a)--(c)} also hold when $L=\Pi_0\cup\Pi_{\bs\phi},$ so that\/ $\bar L=\cS_0\cup\cS_\phi,$ with potential\/ $f_{\bar L}=0,$ although in this case\/ $L$ is immersed rather than embedded and\/ $A(L)$ is not defined.
\label{la4prop5}
\end{prop}

\begin{proof} As the proof is rather long, extending to the end of \S\ref{la43} and including Lemmas \ref{la4lem1}--\ref{la4lem5}, we begin with a summary of the key steps. In the first part of the proof we construct a real analytic Lagrangian neighbourhood $\Psi:T^*L'\supset W\ra\C^m$ for $L'$, so that exact Lagrangians $\ti L'$ in $\C^m$ which are $C^1$-close to $L'$ may be written as $\Psi(\Ga_{\d f})$ for $f:L'\ra\R$ smooth. 

We cannot write $L=\Psi(\Ga_{\d f})$ globally, as $L$ may not be $C^1$-close to $L'$. However, $L$ is $C^1$-close to $L'$ locally both near infinity in $\C^m$, and near any $p\in L\cap L'$ with $T_pL=T_pL'$. So we choose $V\subseteq L'$ open and $G:V\ra\R$ real analytic such that $\Psi(\Ga_{\d f})\subseteq L$ is open, and $\Psi(\Ga_{\d f})$ contains the infinite ends of $L$, and all $p\in L\cap L'$ with $T_pL=T_pL'$. Using the real analytic property of $G$, Lemma \ref{la4lem1} proves some estimates which roughly say that $G$ cannot be too close to $\ep r^{2-m}$ to second-order at $q\in V$ with $\ep>0$ small and~$r(q)\gg 0$. 

We then choose a Hamiltonian perturbation $L^\dag$ of $L'$ of the form $L^\dag=\Psi(\Ga_{\ep\d H})$, where $H:L'\ra\R$ is smooth with $H=r^{2-m}+O(r^{1-m})$ near infinity in $L'$ and $\ep>0$ is small. This $L^\dag$ is an Asymptotically Conical Lagrangian in $\C^m$ with cone $C=\Pi_0\cup\Pi_{\bs\phi}$ and rate $\rho<0$, so $\bar L^\dag=L^\dag\cup\{\iy_0,\iy_{\bs\phi}\}$ is a compact, smooth Lagrangian in $(M,\om)$, a Hamiltonian perturbation of~$\bar L'$.

The form $H\approx r^{1-m}$ near infinity was chosen as $\ep\,\d(r^{1-m})$ is an infinitesimal deformation of $L'$ as an AC special Lagrangian, to leading order, and dominates all other such infinitesimal deformations near infinity.

The important property of $L^\dag$, proved in Lemmas \ref{la4lem2}--\ref{la4lem4}, is that it intersects $L$ transversely in finitely many points $p_1,\ldots,p_k$ with $0<\mu_{L,L^\dag}(p_i)<m$. If $L^\dag$ were special Lagrangian then $0<\mu_{L,L^\dag}(p_i)<m$ by Proposition \ref{la4prop1}. Thus, we start with AC SL $m$-folds $L,L'$ which may not be transverse, and deform $L'$ to $L^\dag$ which is Lagrangian but not special, and intersects $L$ transversely, with the same intersection properties as if $L,L^\dag$ were transverse SL $m$-folds. Lemma \ref{la4lem1} is used to show that $L\cap L^\dag$ is bounded in $\C^m$, and so finite.

The proof that $0<\mu_{L,L^\dag}(p_i)<m$ for $p_i\in L\cap L^\dag$ is divided into four cases:
\begin{itemize}
\setlength{\parsep}{0pt}
\setlength{\itemsep}{0pt}
\item[(i)] $r(p_i)\gg 0$ (actually, we show no such $p_i$ exist);
\item[(ii)] $p_i$ close to $p\in L\cap L'$ with $T_pL=T_pL'$;
\item[(iii)] $p_i$ close to $p\in L\cap L'$ with $T_pL\ne T_pL'$; and
\item[(iv)] $p_i$ not close to $L\cap L'$ (actually, we show no such $p_i$ exist).
\end{itemize}
Lemmas \ref{la4lem2}, \ref{la4lem3} and \ref{la4lem4} handle cases (i)--(ii) and (iii) and (iv), respectively. 

Now $L^\dag$ is not quite what we want, for although $\bar L^\dag$ is a Hamiltonian perturbation of $\bar L'$ and $\bar L,\bar L^\dag$ intersect in finitely many points $p_1,\ldots,p_k,\iy_0,\iy_{\bs\phi}$, the intersections of $\bar L,\bar L^\dag$ at $\iy_0,\iy_{\bs\phi}$ may not be transverse. So in Lemma \ref{la4lem5} we show we can make a small Hamiltonian perturbation $\bar L''$ of $\bar L^\dag$  which changes $\bar L^\dag$ only near $\iy_0,\iy_{\bs\phi}$, such that $\bar L\cap\bar L''=\{p_1,\ldots,p_k,\iy_0,\iy_{\bs\phi}\}$, but now $\bar L,\bar L''$ are transverse at $\iy_0,\iy_{\bs\phi}$, with $\mu_{\bar L,\bar L''}(\iy_0),\mu_{\bar L,\bar L''}(\iy_{\bs\phi})\in\{0,m\}$. This concludes our outline of the proof of Proposition \ref{la4prop5}, we now begin the proof itself.
\smallskip

Since $L'$ is an embedded AC Lagrangian in $\C^m$, by the Lagrangian Neighbourhood Theorem we may choose an open tubular neighbourhood $W$ of the zero section $z(L')$ in $T^*L'$, with projection $\pi:W\ra L'$, where $z:L'\ra T^*L'$ is the zero section and $T^*L'$ is regarded as a symplectic manifold with its canonical symplectic form, and a symplectomorphism $\Psi:W\ra\Psi(W)$ to an open neighbourhood $\Psi(W)$ of $L'$ in $\C^m$, such that~$\Psi\ci z=\id_{L'}:L'\ra\Psi(W)\supset L'$. 

We require that $W$ should `grow linearly at infinity', in that
\e
\bigl\{\text{$(p,\al):p\in L'$, $\al\in T_p^*L'$, $\nm{\al}_{g'}<c\, r(p)$}\bigr\}\subseteq W
\label{la4eq4}
\e
for some small $c>0$, where $\nm{\al}_{g'}$ is computed using the natural Riemannian metric $g'=g_{\sst\C^m}\vert_{L'}$ on $L'$, and $r:\C^m\ra[0,\iy)$ is the radius function $r(p)=\md{p}$. Given $W$, we fix $\Psi$ by requiring that
\e
\Psi(W\cap T_p^*L')\subset p+iT_pL'\subset\C^m\quad\text{for all $p\in L'$.}
\label{la4eq5}
\e
This determines the projection $\pi:W\ra L'$. Identifying the tangent spaces of $W\cap T_p^*L'$ and $\C^m$ with $T_p^*L',\C^m$, the derivative $\d\Psi\vert_q:T_p^*L'\ra \C^m\supset p+iT_pL'$ of $\Psi$ at $q\in W\cap T_p^*L'$ is the composition
\begin{equation*}
\smash{\xymatrix@C=40pt{ T_p^*L' \ar[r]^{\d\pi\vert_q^*} & T_q^*\C^m \ar@{=}[r] & (\C^m)_{\sst\R}^* \ar[r]^{\om^{-1}} & \C^m, }}
\end{equation*}
where $(\C^m)_{\sst\R}^*$ is the dual of $\C^m$ as a real vector space. So \eq{la4eq5} determines the derivative of $\Psi\vert_{W\cap T_p^*L'}:W\cap T_p^*L'\ra p+iT_pL'$ for each $p\in L'$, and this, $\Psi(p,0)=p$, and $W\cap T_p^*L'$ connected, imply that \eq{la4eq5} determines $\Psi$ uniquely. Since $L'$ is real analytic, so is $\Psi$, as it is canonically constructed from~$L'$.

As $\Psi$ is a symplectomorphism, $\Psi^{-1}(L)$ is Lagrangian in $W\subset T^*L'$. We choose an open subset $V\subset L'$ and a real analytic function $G:V\ra\R$ satisfying:
\begin{itemize}
\setlength{\parsep}{0pt}
\setlength{\itemsep}{0pt}
\item[(i)] $\Ga_{\d G}=\bigl\{(p,\d G\vert_p):p\in V\bigr\}=\Psi^{-1}(L)\cap \pi^{-1}(V)\subset W\subset T^*L'$.
\item[(ii)] For some $R\gg 0$ we have $L'\sm\,\ov{\!B}_R\subseteq V$, and $G(p)\ra 0$ as $r(p)\ra\iy$ for $p$ in $L'$.
\item[(iii)] If $p\!\in\! L\!\cap\! L'$ with $T_pL\!=\!T_pL'$ then $p\!\in\! V$, and~$G(p)\!=\!\d G(p)\!=\!\nabla\d G(p)\!=\!0$.
\end{itemize}
Here the point is to write $\Psi^{-1}(L)$ as the graph of an exact 1-form $\d G$ over a subset $V$ in $L'$, where $G$ is locally unique up to addition of a constant. Note that if $G$ exists it is real analytic, as $L,L'$ and $\Psi$ are. 

If $G$ is identically zero on any open set in $V$ then $L,L'$ agree in an open set, so $L=L'$ as they are closed, connected, real analytic submanifolds of $\C^m$, contradicting $L,L'$ distinct. So $G\ne 0$ on every connected component of~$V$.

Writing $\Psi^{-1}(L)$ as $\Ga_{\d G}$ may not be possible globally, as $\Psi^{-1}(L)$ may not be transverse to the fibres of $\pi:W\ra L'$. However, it is possible near infinity in $L'$ since $L,L'$ are both asymptotic to $C=\Pi_0\cap\Pi_{\bs\phi}$, so $\Psi^{-1}(L)$ is asymptotic to $z(L')$ near infinity in $L'$. Thus we can choose $V$ to satisfy (ii), and we fix the additive constant by requiring $G(p)\ra 0$ as $r(p)\ra\iy$. Also, perhaps after making $W$ smaller, it is possible to write $\Psi^{-1}(L)$ as $\Ga_{\d G}$ near the closed set of $p\in L\cap L'$ with $T_pL=T_pL'$, as there $\Psi^{-1}(L)$ is transverse to the fibres of $\pi$, so we can choose $V$ to satisfy (iii), and we fix the additive constant by requiring $G(p)=0$ at such~$p$.

Since $L,L'$ are special Lagrangian, a similar argument to the proof of Theorem \ref{la3thm4} shows that $G$ satisfies a nonlinear elliptic p.d.e.\ similar to \eq{la3eq10}, which we may write in tensor notation in the form
\e
g^{\prime ab}\nabla_a\nabla_bG(\bs y)+P\bigl(\bs y,\nabla_a G(\bs y),\nabla_a\nabla_bG(\bs y)\bigr)=0,
\label{la4eq6}
\e
where $\nabla$ is the Levi-Civita connection of the metric $g'=g\vert_{L'}$ on $L'$, and $P$ is a nonlinear function of its arguments such that
\e
P(\bs y,\al,\be)=O\bigl((1+\md{\bs y})^{1-m}\ms{\al}+\ms{\be}\bigr)
\label{la4eq7}
\e
for small $\al,\be$. Here the factor $(1+\md{\bs y})^{1-m}$ is because asymptotically at infinity, \eq{la4eq6} depends only on $\nabla_a\nabla_b G$ and not on $\nabla_aG$, as in~\eq{la3eq3}.

Making $R\gg 0$ larger if necessary, we have a natural decomposition 
\begin{equation*}
\bigl\{p\in L':r(p)>R\bigr\}=L'_0\amalg L'_{\bs\phi},
\end{equation*}
where $L'_0,L'_{\bs\phi}$ are the ends of $L'$ asymptotic to $\Pi_0,\Pi_{\bs\phi}$, so that
\e
L'_0\approx \Pi_0\sm\,\ov{\!B}_R,\qquad L'_{\bs\phi}\approx \Pi_{\bs\phi}\sm\,\ov{\!B}_R,
\label{la4eq8}
\e
for $\,\ov{\!B}_R$ the closed ball of radius $R$ about 0 in $\C^m$. 

Use $(x_1,\ldots,x_m):=(\Re z_1\vert_{L'},\ldots,\Re z_m\vert_{L'})$ as coordinates on $L_0'$, so that $G\vert_{L_0'}=G\vert_{L_0'}(x_1,\ldots,x_m)$, and write $r=(x_1^2+\cdots+x_m^2)^{1/2}$. Now Theorem \ref{la3thm4} writes $L,L'$ near infinity in $\Pi_0$ as $\Ga_{\d f},\Ga_{\d f'}$, for 
\begin{equation*}
f(\bs x)=r^{2-m}\cdot F(x_1/r^2,\ldots,x_m/r^2),\;\>
f'(\bs x)=r^{2-m}\cdot F'(x_1/r^2,\ldots,x_m/r^2),
\end{equation*}
and $F(y_1,\ldots,y_m),F'(y_1,\ldots,y_m)$ are real analytic functions defined near 0 in $\R^m$. Under the identification \eq{la4eq8} we may take $G\vert_{L'_0}(\bs x)$ to be $f(\bs x)-f'(\bs x)$. Thus we may write
\e
G\vert_{L_0'}(x_1,\ldots,x_m)=r^{2-m}\cdot F_0(x_1/r^2,\ldots,x_m/r^2),
\label{la4eq9}
\e
where $F_0:\bigl\{(y_1,\ldots,y_m)\in\R^m:y_1^2+\cdots+y_m^2<R^{-2}\bigr\}\ra\R$ is real analytic. The analogue holds for $G\vert_{L_{\bs\phi}'}$, in terms of $F_{\bs\phi}:\bigl\{\bs y\in\Pi_{\bs\phi}:\md{\bs y}<R^{-1}\bigr\}\ra\R$.

The next lemma, which will be used in the proof of Lemma \ref{la4lem2}, depends only on the form \eq{la4eq9} for $G\vert_{L_0'},G\vert_{L_{\bs\phi}'}$ with $F_0,F_{\bs\phi}$ real analytic. The following explanation may help to understand the proof of Lemma \ref{la4lem1}. Suppose $q(t)\in V$ and $\ep(t)>0$ are smooth functions of $t\in(0,\de)$ and satisfy
\e
\begin{split}
\ep(t)^{-1}(\nabla_aG)(q(t))&=(\nabla_ar^{2-m})(q(t)),\\
\ep(t)^{-1}(\nabla_a\nabla_bG)(q(t))&=(\nabla_a\nabla_br^{2-m})(q(t)),
\end{split}
\label{la4eq10}
\e
for all $t\in(0,\de)$. Then by computing $\frac{\d^2}{\d t^2}\bigl[\ep(t)^{-1}G(q(t))\bigr]$ in two different ways using \eq{la4eq10}, we find that $\frac{\d}{\d t}\ep(t)=0$, so that if $\ep(t)=At^k+O(t^{k+1})$ for $A>0$ then $k=0$, which is how we get the contradiction in the proof below. The proof works by showing that if \eq{la4eq10} only holds approximately, as in \eq{la4eq11}, then $\frac{\d}{\d t}\ep(t)$ is small compared to $\ep(t)$ when $r(q(t))\gg 0$, so $\ep(t)$ cannot reach 0, and solutions $q,\ep$ of \eq{la4eq11} with $r(q)\gg 0$ must satisfy~$\ep\ge\ep_\iy>0$.

\begin{lem} Let\/ $C_1,C_2>0$ be given. Then making $R\gg 0$ larger if necessary, there exists\/ $\ep_\iy>0$ such that there are no $q\in L_0'\cup L_{\bs\phi}'$ and\/ $0\!<\!\ep\!<\!\ep_\iy$ satisfying
\e
\begin{split}
\bnm{\ep^{-1}\nabla_aG(q)-\nabla_ar^{2-m}(q)}_{g'}&<C_1r(q)^{-m}
\quad\text{and\/}\\
\bnm{\ep^{-1}\nabla_a\nabla_bG(q)-\nabla_a\nabla_br^{2-m}(q)}_{g'}&<C_2r(q)^{-1-m}.
\end{split}
\label{la4eq11}
\e

\label{la4lem1}
\end{lem}

\begin{proof} Use $(x_1,\ldots,x_m)=(\Re z_1\vert_{L'},\ldots,\Re z_m\vert_{L'})$ as coordinates on $L_0'$ and set $(y_1,\ldots,y_m)=(x_1/r^2,\ldots,x_m/r^2)$, so that $(x_1,\ldots,x_m)=(y_1/s^2,\ldots,y_m/s^2)$ with $s=(y_1^2+\cdots+y_m^2)^{1/2}=r^{-1}$. Define
\e
\begin{split}
T=\bigl\{&(y_1,\ldots,y_m,\ep)\in\R^{m+1}:0<y_1^2+\cdots+y_m^2<R^{-2},\;\> \ep>0,\\ 
&\nm{\ep^{-1}\nabla_aG(y_1/s^2,\ldots,y_m/s^2)-\nabla_as^{m-2}}_{g'}<C_1s^m,\\
&\nm{\ep^{-1}\nabla_a\nabla_bG(y_1/s^2,\ldots,y_m/s^2)-\nabla_a\nabla_bs^{m-2}}_{g'}<C_2s^{m+1}\bigr\}.
\end{split}
\label{la4eq12}
\e
That is, $T$ is just the set of points $(q,\ep)$ satisfying \eq{la4eq11} for $q\in L_0'$, but written in terms of $(y_1,\ldots,y_m)$ and $s$ rather than $(x_1,\ldots,x_m)$ and $r$. From \eq{la4eq9} and the fact that the conformal inversion $s^4g'$ extends to a real analytic metric over $(y_1,\ldots,y_m)=0$, which follows from the formula for $g'$ on $L_0'$
\begin{equation*}
g'_{ij}(x_1,\ldots,x_m)=\de_{ij}+\sum_{k=1}^m\frac{\pd^2G}{\pd x_i\pd x_k}(x_1,\ldots,x_m)\frac{\pd^2G}{\pd x_j\pd x_k}(x_1,\ldots,x_m)
\end{equation*}
and the expression \eq{la4eq9} with $F_0$ real analytic, we see that $T$ in \eq{la4eq12} is defined by finitely many real analytic inequalities in $(y_1,\ldots,y_m,\ep)$, which in particular are real analytic at~$y_1=\cdots=y_m=0$.

Suppose for a contradiction that $(0,\ldots,0,0)$ lies in the closure $\bar T$ of $T$ in $\R^{m+1}$. Since $T$ is defined by finitely many real analytic inequalities, $\bar T$ has no pathological behaviour at its boundary $\bar T\sm T$ (Example \ref{laAex} below illustrates the kind of pathological behaviour we are worried about), and we can find a real analytic curve $\bigl(y_1(t),\ldots,y_m(t),\ep(t)\bigr):(-\de,\de)\ra\bar T$ for $\de>0$ such that $\bigl(y_1(0),\ldots,\ab y_m(0),\ab \ep(0)\bigr)=(0,\ldots,0)$ and $\bigl(y_1(t),\ldots,y_m(t),\ep(t)\bigr)\in T$ for $t>0$. Write $s(t)=(y_1(t)^2+\cdots+y_m(t)^2)^{1/2}$, so that $s(0)=0$ and $s(t)>0$ for $t>0$, and for $t\in(0,\de)$ write $r(t)=s(t)^{-1}$ and $x_i(t)=y_i(t)/s(t)^2$ for~$i=1,\ldots,m$.

Now both $\ep(t)$ and $s(t)^2$ are real analytic functions $(-\de,\de)\ra[0,\iy)$ with $\ep(0)=s(0)^2=0$ and $\ep(t),s(t)^2>0$ for $t>0$, and also $s(t)^2\ge 0$ for $t<0$. Therefore $\ep(t)=At^k+O(t^{k+1})$, $s(t)^2=Bt^{2l}+O(t^{2l+1})$ for real $A,B>0$ and integers $k,l>0$. By making $\de>0$ smaller and replacing $t$ by $t'=t\cdot(s(t)^2/t^{2l})^{1/2l}$, we can suppose that $s(t)^2=t^{2l}$, so that $s(t)=t^l$ for $t\ge 0$.

From \eq{la4eq9} we see that $G(x_1(t),\ldots,x_m(t))=t^{l(m-2)}F_0\bigl(y_1(t),\ldots,y_m(t)\bigr)$ is a real analytic function of $t\in[0,\de)$, including at $t=0$. Differentiating gives
\e
\begin{split}
\ts\frac{\d}{\d t}&\bigl[\ep(t)^{-1}G(x_1(t),\ldots,x_m(t))\bigr]\\
&=\ts-\ep(t)^{-2}\frac{\d\ep}{\d t}(t)\cdot G(x_1(t),\ldots,x_m(t))
+\frac{\d}{\d t}\bigl(r(t)^{2-m}\bigr)\\
&\ts\qquad+\sum_{i=1}^m\bigl[\ep(t)^{-1}\frac{\pd G}{\pd x_i}(x_1(t),\ldots,x_m(t))-\frac{\pd}{\pd x_i}(r^{2-m})(t)\bigr]\cdot\frac{\d x_i}{\d t}(t)\\
&\ts=-\ep(t)^{-1}\frac{\d\ep}{\d t}(t)\cdot \bigl[\ep(t)^{-1} G(x_1(t),\ldots,x_m(t))\bigr]+l(m-2)t^{l(m-2)-1}\\
&\qquad+O(t^{lm})\cdot O(t^{-l-1})
\end{split}
\label{la4eq13}
\e
for $t>0$, where in the first step we add and subtract $\frac{\d}{\d t}(r(t)^{2-m})$, and in the second we use $\ep(t)^{-1}\frac{\pd G}{\pd x_i}(x_1(t),\ldots,x_m(t))\!-\!\frac{\pd}{\pd x_i}(r^{2-m})(t)\ab=O(s(t)^m)=O(t^{lm})$ by \eq{la4eq12}, and~$\frac{\d x_i}{\d t}(t)=O\bigl(\frac{\d r}{\d t}(t)\bigr)=O(t^{-l-1})$.

Differentiating \eq{la4eq13} again, substituting \eq{la4eq13} for the derivative of the term $\bigl[\ep(t)^{-1} G(x_1(t),\ldots,x_m(t))\bigr]$, using $\ep(t)^{-1}\frac{\d\ep}{\d t}(t)=kt^{-1}+O(1)$, and noting that as the remainder term $O(t^{lm-l-1})$ is smooth we have $\frac{\d}{\d t}O(t^{lm-l-1})=O(t^{lm-l-2})$, we see that for $t>0$ we have
\ea
&\ts\frac{\d^2}{\d t^2}\bigl[\ep(t)^{-1}G(x_1(t),\ldots,x_m(t))\bigr]=
\nonumber\\
&\ts-\!\ep(t)^{-2}\frac{\d^2\ep}{\d t^2}(t)\cdot G(x_1(t),\ldots,x_m(t))
\!+\!2\ep(t)^{-3}\bigl(\frac{\d\ep}{\d t}(t)\bigr)^2\cdot G(x_1(t),\ldots,x_m(t))
\nonumber\\
&+l(m-2)(l(m-2)-k-1)t^{l(m-2)-2}+O(t^{lm-l-2}).
\label{la4eq14}
\ea

Doing the computation a different way yields, for $t>0$
\ea
&\ts\frac{\d^2}{\d t^2}\bigl[\ep(t)^{-1}G(x_1(t),\ldots,x_m(t))\bigr]=-\ep(t)^{-2}\frac{\d^2\ep}{\d t^2}(t)\cdot G(x_1(t),\ldots,x_m(t))\nonumber\\
&\ts-\!2\ep(t)^{-1}\frac{\d\ep}{\d t}(t)\cdot \frac{\d}{\d t}\bigl[\ep(t)^{-1} G(x_1(t),\ldots,x_m(t))\bigr]\!+\!\ep(t)^{-1}\frac{\d^2}{\d t^2}\bigl[G(x_1(t),\ldots,x_m(t))\bigr]
\nonumber\\
&=\ts-\ep(t)^{-2}\frac{\d^2\ep}{\d t^2}(t)\cdot G(x_1(t),\ldots,x_m(t))\nonumber\\
&\ts-2\ep(t)^{-1}\frac{\d\ep}{\d t}(t)\cdot \frac{\d}{\d t}\bigl[\ep(t)^{-1} G(x_1(t),\ldots,x_m(t))\bigr]+\frac{\d^2}{\d t^2}\bigl(r(t)^{2-m}\bigr)\nonumber\\
&\ts+\sum_{i,j=1}^m\bigl[\ep(t)^{-1}\frac{\pd^2 G}{\pd x_i\pd x_j}(x_1(t),\ldots,x_m(t))-\frac{\pd^2}{\pd x_i\pd x_j}(r^{2-m})(t)\bigr]\cdot\frac{\d x_i}{\d t}(t)\frac{\d x_j}{\d t}(t)
\nonumber\\
&\ts+\sum_{i=1}^m\bigl[\ep(t)^{-1}\frac{\pd G}{\pd x_i}(x_1(t),\ldots,x_m(t))-\frac{\pd}{\pd x_i}(r^{2-m})(t)\bigr]\cdot\frac{\d^2 x_i}{\d t^2}(t)
\nonumber\\
&\ts=-\ep(t)^{-2}\frac{\d^2\ep}{\d t^2}(t)\cdot G(x_1(t),\ldots,x_m(t))
\!+\!2\ep(t)^{-3}\bigl(\frac{\d\ep}{\d t}(t)\bigr)^2\cdot G(x_1(t),\ldots,x_m(t))
\nonumber\\
&+l(m-2)(l(m-2)-2k-1)t^{l(m-2)-2}
\nonumber\\
&+O(t^{-1})\cdot O(t^{lm-l-1})+O(t^{l(m+1)})\cdot O(t^{-2l-2})+O(t^{lm})\cdot O(t^{-l-2}),
\label{la4eq15}
\ea
where in the second step we add and subtract $\frac{\d^2}{\d t^2}(r(t)^{2-m})$, and in the third step we substitute in \eq{la4eq13} and use $\ep(t)^{-1}\frac{\d\ep}{\d t}(t)=kt^{-1}+O(1)$, the first error term comes from $\ep(t)^{-1}\frac{\d\ep}{\d t}(t)=O(t^{-1})$ and the $O(t^{lm-l-1})$ in \eq{la4eq13}, the second error term from 
$\ep(t)^{-1}\frac{\pd^2 G}{\pd x_i\pd x_j}(x_1(t),\ab\ldots,\ab x_m(t))-\frac{\pd^2}{\pd x_i\pd x_j}(r^{2-m})(t)=O(s(t)^{m+1})=O(t^{l(m+1)})$ by \eq{la4eq12} and $\frac{\d x_i}{\d t}(t),\ab \frac{\d x_j}{\d t}(t)=O\bigl(\frac{\d r}{\d t}(t)\bigr)=O(t^{-l-1})$, and the third error term from $\ep(t)^{-1}\frac{\pd G}{\pd x_i}(x_1(t),\ab\ldots,\ab x_m(t))-\frac{\pd}{\pd x_i}(r^{2-m})(t)=O(s(t)^m)=O(t^{lm})$ by \eq{la4eq12}, and $\frac{\d^2 x_i}{\d t^2}(t)=O\bigl(\frac{\d^2 r}{\d t^2}(t)\bigr)=O(t^{-l-2})$.

Once we cancel the common terms $-\ep^{-2}\frac{\d^2\ep}{\d t^2}\cdot G+2\ep^{-3}(\frac{\d\ep}{\d t})^2\cdot G$ in both \eq{la4eq14} and \eq{la4eq15}, the terms in $t^{l(m-2)-2}$ are dominant. Comparing coefficients of $t^{l(m-2)-2}$ in \eq{la4eq14}--\eq{la4eq15} gives
\begin{equation*}
l(m-2)(l(m-2)-k-1)=l(m-2)(l(m-2)-2k-1),
\end{equation*}
which forces $k=0$ as $l>0$ and $m>2$, a contradiction.

We have shown that $(0,\ldots,0)\notin\bar T$. Thus $(0,\ldots,0)$ has an open neighbourhood which does not intersect $T$. So making $R\gg 0$ larger, there exists $\ep_\iy>0$ such that if $(y_1,\ldots,y_m,\ep)\in\R^{m+1}$ with $y_1^2+\cdots+y_m^2<R^{-2}$ and $0<\ep<\ep_\iy$ then $(y_1,\ldots,y_m,\ep)\notin T$, which implies that $q=(y_1/s^2,\ldots,y_m/s^2)\in L_0'$ and $\ep$ do not satisfy \eq{la4eq11}. This  proves Lemma \ref{la4lem1} for $q\in L_0'$, and increasing $R$ and decreasing $\ep_\iy$, the same works for~$q\in L_{\bs\phi}'$. 
\end{proof}

We now make a Hamiltonian perturbation $L^\dag$ of $L'$ to achieve transversality away from $\iy_0,\iy_{\bs\phi}$. Choose a smooth function $H:L'\ra\R$ satisfying
\e
\nabla^k\bigl(H-r^{2-m}\bigr)=O(r^{1-m-k})
\label{la4eq16}
\e
as $r\ra\iy$ for all $k=0,1,\ldots,$ and
\e
\Crit(G)\cap\Crit(H)=\es,
\label{la4eq17}
\e
which holds for generic $H$ as $V\sm\Crit(G)$ is open and dense in $V$, and an additional genericity condition we will give shortly. Using \eq{la4eq4} and \eq{la4eq16}, we see that there exists $\ep_W>0$ such that if $0<\ep<\ep_W$ then~$\Ga_{\ep\d H}\subset W$. 

Let $\ep\in(0,\ep_W)$ satisfy a finite number of smallness conditions and genericity conditions we will give later in the proof, and define $L^\dag=\Psi(\Ga_{\ep\d H})$, so that $L^\dag$ is an Asymptotically Conical Lagrangian submanifold of $\C^m$ with cone $C=\Pi_0\cup\Pi_{\bs\phi}$ and rate $\rho\in(2-m,0)$, a Hamiltonian perturbation of $L'$. Therefore $\bar L^\dag=L^\dag\cup\{\iy_0,\iy_{\bs\phi}\}$ is a compact, smooth Lagrangian in $(M,\om)$, a Hamiltonian perturbation of $\bar L'$. The additional genericity condition we require on $H$ is that $L,L^\dag$ should intersect transversely in $\C^m$ for generic $\ep\in(0,\ep_W)$. This holds for $H$ generic and $\ep_W$ sufficiently small.

Now in $\pi^{-1}(V)\subset W\subset T^*L'$ we have $\Psi^{-1}(L)=\Ga_{\d G}$ and $\Psi^{-1}(L^\dag)=\Ga_{\ep\d H}$. Thus we see that
\e
L\cap L^\dag\cap\Psi(\pi^{-1}(V))\cong \bigl\{q\in V:(\ep\d H-\d G)(q)=0\bigr\},
\label{la4eq18}
\e
where $p$ on the left hand side corresponds to $q=\pi\ci\Psi^{-1}(p)$ on the right. From \eq{la4eq9} and \eq{la4eq16} we see that on $L_0'$ we have
\e
\ep\d H-\d G=(\ep-F_0(0,\ldots,0))\d r^{2-m}+O(r^{-m}),
\label{la4eq19}
\e
where $\ep-F_0(0,\ldots,0)\ne 0$ as $\ep$ is generic. Thus $\ep\d H-\d G$ has no zeroes near infinity in $L_0'$, as $m\ge 3$, and similarly for $L_{\bs\phi}'$. So \eq{la4eq18} implies $L\cap L^\dag$ is bounded in~$\C^m$. 

Since $L,L^\dag$ intersect transversely in $\C^m$ for generic $\ep$ by choice of $H$, and $L\cap L^\dag$ is bounded, the intersection is finite, say $L\cap L^\dag=\{p_1,\ldots,p_k\}$ for $p_1,\ldots,p_k\in\C^m$. Note however that $\bar L,\bar L^\dag$ do {\it not\/} intersect transversely at $\iy_0,\iy_{\bs\phi}$ in $M$. We will deal with this in Lemma \ref{la4lem5} by perturbing $\bar L^\dag$ to~$\bar L''$.

Next we prove three lemmas about $\mu_{L,L^\dag}(p_i)$ for $p_i$ in different regions of~$\C^m$.

\begin{lem} Making $R\gg 0$ larger and\/ $V\subset L'$ smaller if necessary, if\/ $p\in L\cap L^\dag\cap \Psi(\pi^{-1}(V))$ and\/ $\ep<\ep_\iy$ for some $\ep_\iy>0$ then\/~$0<\mu_{L,L^\dag}(p)<m$.
\label{la4lem2}
\end{lem}

\begin{proof} Suppose $p\in L\cap L^\dag\cap\Psi(\pi^{-1}(V))$. Then under \eq{la4eq18}, $p$ corresponds to $q=\pi\ci\Psi^{-1}(p)$ with
\e
\ep\nabla_aH(q)=\nabla_aG(q).
\label{la4eq20}
\e
Define a symmetric 2-tensor $M_{ab}$ at $q\in L'$ by
\e
M_{ab}=\ep\nabla_a\nabla_bH(q)-\nabla_a\nabla_bG(q).
\label{la4eq21}
\e
That is, $M_{ab}$ is the Hessian of $\ep H-G$ at $q$. Choosing an orthonormal basis for $T_qL'$, we can regard $M_{ab}$ as a symmetric $m\t m$ matrix. Write $\la_1,\ldots,\la_m$ for the eigenvalues of $M_{ab}$. Then the $\la_i$ are all nonzero, as $L,L^\dag$ intersect transversely at $p$. In Definition \ref{la2def4}, computation shows that $\th_L(p)=0$, $\th_{L^\dag}(p)=\sum_{i=1}^m\tan^{-1}(\ha\la_i)$, and $\phi_i=\tan^{-1}(\ha\la_i)$ if $\la_i>0$ and $\phi_i=\pi+\tan^{-1}(\ha\la_i)$ if $\la_i<0$ for $i=1,\ldots,m$, so that $\mu_{L,L^\dag}(p)=0,\ldots,m$ is the number of negative $\la_i$. We will show $M_{ab}$ has at least one positive and one negative eigenvalue. Thus there cannot be either 0 or $m$ positive eigenvalues, forcing $0<\mu_{L,L^\dag}(p)<m$, as we want to prove. 

Divide into the three cases:
\begin{itemize}
\setlength{\parsep}{0pt}
\setlength{\itemsep}{0pt}
\item[(A)] $r(q)>R$ and $\nm{\nabla_a\nabla_bG(q)}_{g'}\le 2\ep\nm{\nabla_a\nabla_bH(q)}_{g'}$; 
\item[(B)] $r(q)>R$ and $\nm{\nabla_a\nabla_bG(q)}_{g'}>2\ep\nm{\nabla_a\nabla_bH(q)}_{g'}$; or
\item[(C)] $r(q)\le R$.
\end{itemize}
First consider cases (A), (B), so $r(q)>R$. Equation \eq{la4eq16} with $k=1,2$ gives
\e
\begin{split}
\bnm{\nabla_aH(q)-\nabla_ar^{2-m}(q)}_{g'}&<C_1r(q)^{-m},\\
\bnm{\nabla_a\nabla_bH(q)-\nabla_a\nabla_br^{2-m}(q)}_{g'}&<C_2r(q)^{-1-m},
\end{split}
\label{la4eq22}
\e
for some $C_1,C_2>0$ and all $q\in L'$ with $r(q)>R$. Applying Lemma \ref{la4lem1} with constants $C_1$ and $(2m+1)C_2$ gives a constant~$\ep_\iy>0$. 

Suppose the $\ep$ used to define $L^\dag$ is chosen to satisfy $\ep<\ep_\iy$. Then \eq{la4eq20} and \eq{la4eq22} show the first equation of \eq{la4eq11} holds, and therefore the second equation with constant $(2m+1)C_2$ does not, by Lemma \ref{la4lem1}. Hence
\e
\bnm{\ep^{-1}\nabla_a\nabla_bG(q)-\nabla_a\nabla_br^{2-m}(q)}_{g'}\ge (2m+1)C_2r(q)^{-1-m}.
\label{la4eq23}
\e

Now
\ea
&(\la_1^2+\cdots+\la_m^2)^{1/2}=\bnm{M_{ab}}_{g'}=
\bnm{\ep\nabla_a\nabla_bH(q)-\nabla_a\nabla_bG(q)}_{g'}
\nonumber\\
&\ge\ep\bnm{\ep^{-1}\nabla_a\nabla_bG(q)\!-\!\nabla_a\nabla_br^{2-m}(q)}_{g'}\!-\!\ep\bnm{\nabla_a\nabla_bH(q)\!-\!\nabla_a\nabla_br^{2-m}(q)}_{g'}
\nonumber\\
&>(2m+1)C_2\ep\, r(q)^{-1-m}-C_2\ep\, r(q)^{-1-m}=2mC_2\ep\, r(q)^{-1-m},
\label{la4eq24}
\ea
by \eq{la4eq21}--\eq{la4eq23}. Also
\e
\begin{split}
&\bmd{\la_1+\cdots+\la_m}=\bmd{g^{\prime ab}M_{ab}}\\
&\le \bmd{g^{\prime ab}\nabla_a\nabla_bG(q)}\!+\!\ep\bmd{g^{\prime ab}\nabla_a\nabla_br^{2-m}}\!+\!\ep\bmd{g^{\prime ab}\bigl(\nabla_a\nabla_bH\!-\!\nabla_a\nabla_br^{2-m})}\\
&\le C_3 \bigl(r(q)^{1-m}\bnm{\nabla_aG(q)}^2_{g'}+\bnm{\nabla_a\nabla_bG(q)}^2_{g'}\bigr)\\
&\qquad +C_4\ep r(q)^{-2m}+\sqrt{m}\,C_2\ep r(q)^{-1-m}
\end{split}
\label{la4eq25}
\e
for some $C_3,C_4>0$, where $C_3(\cdots)$ estimates $\md{g^{\prime ab}\nabla_a\nabla_bG(q)}$ using \eq{la4eq6}--\eq{la4eq7}. Also $C_4\ep r(q)^{-2m}$ estimates $\ep\md{g^{\prime ab}\nabla_a\nabla_br^{2-m}}$, using the facts that $\De r^{2-m}=0$ on Euclidean space $\R^m$, so that $g_0^{ab}\nabla_a\nabla_br^{2-m}=0$ with $g_0$ the Euclidean metric, and $g^{\prime ab}-g_0^{ab}=O(r^{-m})$. And $\sqrt{m}\, C_2\ep r(q)^{-1-m}$ estimates $\ep\md{g^{\prime ab}\bigl(\nabla_a\nabla_bH\!-\!\nabla_a\nabla_br^{2-m})}$, using \eq{la4eq22} and~$\nm{g^{\prime ab}}_{g'}=\sqrt{m}$.

Now \eq{la4eq20}--\eq{la4eq22} imply that $\nm{\nabla_aG(q)}=O(\ep r(q)^{1-m})$. In case (A), \eq{la4eq22} and $\nm{\nabla_a\nabla_bG(q)}_{g'}\le 2\ep\nm{\nabla_a\nabla_bH(q)}_{g'}$ give $\nm{\nabla_a\nabla_bG(q)}=O(\ep r(q)^{-m})$. So the $C_3,C_4$ terms in \eq{la4eq25} are small compared to the last term for $r(q)$ large, and making $R\gg 0$ bigger we can suppose that $\md{\la_1+\cdots+\la_m}\le 2\sqrt{m}\,C_2\ep\, r(q)^{-1-m}$. Combining this with \eq{la4eq24} implies that
\e
\bmd{\la_1+\cdots+\la_m}< \ts\frac{1}{\sqrt{m}}(\la_1^2+\cdots+\la_m^2)^{1/2}.
\label{la4eq26}
\e
This implies the $\la_i$ cannot be all negative or all positive, so $0<\mu_{L,L^\dag}(p)<m$, proving the lemma in case~(A).

For (B), $\nm{\nabla_a\nabla_bG(q)}_{g'}>2\ep\nm{\nabla_a\nabla_bH(q)}_{g'}$ implies that $\nm{\nabla_a\nabla_bG(q)}_{g'}\ab\le\ab 2\nm{M_{ab}}_{g'}=2(\la_1^2+\cdots+\la_m^2)^{1/2}$, and
$\nm{\nabla_a\nabla_bG(q)}_{g'}=O(r(q)^{-m})$ by \eq{la4eq9}, so
$C_3\nm{\nabla_a\nabla_bG(q)}^2_{g'}$ in \eq{la4eq25} is small compared to $(\la_1^2+\cdots+\la_m^2)^{1/2}$ for large $r(q)$. Again, making $R\gg 0$ bigger we deduce \eq{la4eq26}, and so $0<\mu_{L,L^\dag}(p)<m$, and the lemma holds in case~(B).

For (C), recall that $V\subset L'$ was defined in (i)--(iii) above to contain $\bigl\{q\in L':r(q)>R\bigr\}$ for $R\gg 0$ and an open neighbourhood of each $x\in L\cap L'$ with $T_xL=T_xL'$, and then $G(x)=\nabla_a G(x)=\nabla_a\nabla_bG(x)=0$. Case (C) concerns the open neighbourhoods of $x\in L\cap L'$ with $T_xL=T_xL'$. Let $x$ be such a point. Choose real analytic coordinates $(w_1,\ldots,w_m)$ on a small open neighbourhood $V_x$ of $x$ in $V$, such that $x=(0,\ldots,0)$, and regard $G\vert_{V_x}$ as a real analytic function $G(w_1,\ldots,w_m)$. Then $G(0,\ldots,0)=\frac{\pd G}{\pd w_i}(0,\ldots,0)=\frac{\pd^2G}{\pd w_i\pd w_j}(0,\ldots,0)=0$ for all $i,j$, so making $V_x$ smaller, for some $D_1,D_2>0$ we have
\e
\bnm{\nabla_aG(\bs w)}_{g'}\le D_1u^2,\quad
\bnm{\nabla_a\nabla_bG(\bs w)}_{g'}\le D_2u,
\label{la4eq27}
\e
for all $\bs w=(w_1,\ldots,w_m)\in V_x$, where $u:=(w_1^2+\cdots+w_m^2)^{1/2}$. 

Applying Theorem \ref{laAthm1} to the real analytic function $\frac{\pd G}{\pd w_i}(\bs w)$ with $k=0$ for $i=1,\ldots,m$, and making $V_x$ smaller, we deduce that there exists $D_3>0$ such that for all $\bs w\in V_x$ we have 
\e
\bnm{\nabla_aG(\bs w)}_{g'}\le D_3u\bnm{\nabla_a\nabla_bG(\bs w)}_{g'}.
\label{la4eq28}
\e
Combining \eq{la4eq6} and \eq{la4eq7}, we see that making $V_x$ smaller, there exists $D_4>0$ such that for all $\bs w\in V_x$ we have 
\e
\bmd{g^{\prime ab}\nabla_a\nabla_bG(\bs w)}\le D_4\bigl(\bnm{\nabla_aG(\bs w)}_{g'}^2+\bnm{\nabla_a\nabla_bG(\bs w)}_{g'}^2\bigr).
\label{la4eq29}
\e

As $\nabla_a G(x)=0$, equation \eq{la4eq17} gives $\nabla_a H(x)\ne 0$, so making $V_x$ smaller, there exists $D_5>0$ such that for all $\bs w\in V_x$ we have 
\e
\bnm{\nabla_a\nabla_bH(\bs w)}_{g'}\le D_5\bnm{\nabla_aH(\bs w)}_{g'}.
\label{la4eq30}
\e

Now let $p\in L\cap L^\dag\cap\Psi(\pi^{-1}(V))$, $q=\pi\ci\Psi^{-1}(p)$, $M_{ab}$ and $\la_1\le\cdots\le\la_m$ be as in the first part of the proof, and suppose $q\in V_x$, so that $q=(w_1,\ldots,w_m)=\bs w$ in coordinates. Then
\e
\begin{split}
&\bmd{\la_1+\cdots+\la_m}=\bmd{g^{\prime ab}M_{ab}}\le
\bmd{g^{\prime ab}\nabla_a\nabla_bG(\bs w)}+
\bmd{\ep g^{\prime ab}\nabla_a\nabla_bH(\bs w)}\\
&\le D_4\bigl(\bnm{\nabla_aG(\bs w)}_{g'}^2+\bnm{\nabla_a\nabla_bG(\bs w)}_{g'}^2\bigr)+\sqrt{m}\,\bnm{\ep\nabla_a\nabla_bH(\bs w)}_{g'}\\
&\le D_4(D_1u^2\cdot D_3u+D_2u)\bnm{\nabla_a\nabla_bG(\bs w)}_{g'}+\sqrt{m}\,D_5\bnm{\ep\nabla_aH(\bs w)}_{g'}\\
&=D_4(D_1D_3u^3+D_2u)\bnm{\nabla_a\nabla_bG(\bs w)}_{g'}+\sqrt{m}\,D_5\bnm{\nabla_aG(\bs w)}_{g'}\\
&\le D_4(D_1D_3u^3+D_2u+\sqrt{m}\,D_5D_3u)\bnm{\nabla_a\nabla_bG(\bs w)}_{g'},
\end{split}
\label{la4eq31}
\e
using equations \eq{la4eq20} and \eq{la4eq27}--\eq{la4eq30} and $\nm{g^{\prime ab}}_{g'}=\sqrt{m}$. Also
\ea
&(\la_1^2+\cdots+\la_m^2)^{1/2}=\bnm{M_{ab}}_{g'}\ge \bnm{\nabla_a\nabla_bG(\bs w)}_{g'}-\bnm{\ep\nabla_a\nabla_bH(\bs w)}_{g'}
\nonumber\\
&\ge \bnm{\nabla_a\nabla_bG(\bs w)}_{g'}\!-\!D_5\bnm{\ep\nabla_aH(\bs w)}_{g'}\!=\!\bnm{\nabla_a\nabla_bG(\bs w)}_{g'}\!-\!D_5\bnm{\nabla_aG(\bs w)}_{g'}
\nonumber\\
&\ge (1-D_5D_3u)\bnm{\nabla_a\nabla_bG(\bs w)}_{g'},
\label{la4eq32}
\ea
using \eq{la4eq20}, \eq{la4eq28} and \eq{la4eq30}. Making $V_x$ smaller, so that $u=(w_1^2+\cdots+w_m^2)^{1/2}$ is small on $V_x$, we can suppose that $D_5D_3u\le\ha$ and $D_4(D_1D_3u^3+D_2u+\sqrt{m}\,D_5D_3u)<\frac{1}{2\sqrt{m}}$ on $V_x$. Then \eq{la4eq31}--\eq{la4eq32} imply \eq{la4eq26}, so as in cases (A) and (B) we have~$0<\mu_{L,L^\dag}(p)<m$. 

This shows that each $x\in L\cap L'$ with $T_xL=T_xL'$ has an open neighbourhood $V_x$ in $V$ such that $0<\mu_{L,L^\dag}(p)<m$ for all $p\in L\cap L^\dag\cap\Psi(\pi^{-1}(V))$ with $q=\pi\ci\Psi^{-1}(p)\in V_x$. Making $V$ smaller, we can suppose $V$ is the union of a family of such open neighbourhoods $V_x$, together with the region $\bigl\{q\in L':r(q)>R\bigr\}$ already dealt with in cases (A), (B). This completes case (C), and the proof of Lemma~\ref{la4lem2}.
\end{proof}

\begin{lem} Suppose $x\in L\cap L'$ with\/ $T_xL\ne T_xL'$. Then there exist\/ $\ep_x>0$ and an open neighbourhood\/ $U_x$ of\/ $x$ in $\C^m$ such that if\/ $0<\ep<\ep_x$ and\/ $p\in L\cap L^\dag\cap U_x$ then\/~$0<\mu_{L,L^\dag}(p)<m$.
\label{la4lem3}
\end{lem}

\begin{proof} Since $L,L'$ are both special Lagrangian in $\C^m$, after applying an $\SU(m)$ coordinate change to $\C^m$ we may suppose that
\e
\begin{split}
T_xL&=\bigl\{(x_1,\ldots,x_m):x_1,\ldots,x_m\in\R^m\bigr\},\\
T_xL'&=\bigl\{(e^{i\psi_1}x_1,\ldots,e^{i\psi_m}x_m):x_1,\ldots,x_m\in\R^m\bigr\},
\end{split}
\label{la4eq33}
\e
for unique $\psi_1,\ldots,\psi_m\in[0,\pi)$ with $\psi_1\le \psi_2\le\cdots\le\psi_m$. Then $\psi_1=\cdots=\psi_k=0$ and $\psi_{k+1},\ldots,\psi_m\in(0,\pi)$ for some unique $k$, where $0\le k<m$ as $T_xL\ne T_xL'$. Since $T_xL'$ is special Lagrangian we have $\psi_1+\cdots+\psi_m=a\pi$ for $a\in\Z$. As $\psi_1=\cdots=\psi_k=0$ and $\psi_{k+1},\ldots,\psi_m\in(0,\pi)$, we see that~$0<a<m-k$.

If $U_x$ is a sufficiently small open neighbourhood of $x$ in $\C^m$, then $T_pL$ is close to $T_xL$ as Lagrangian planes in $\C^m$ for any $p\in L\cap U_x$. If also $\ep$ is sufficiently small, then $L^\dag=\Psi(\Ga_{\ep\d H})$ is $C^1$-close to $L'$, so $T_pL^\dag$ is close to $T_xL'$ for any $p\in L^\dag\cap U_x$. Thus, if $U_x$ and $\ep_x>0$ are sufficiently small, then for any $0<\ep<\ep_x$ and $p\in L\cap L^\dag\cap U_x$, the pair $(T_pL,T_pL^\dag)$ is close to $(T_xL,T_xL')$ as pairs of Lagrangian planes in $\C^m$. Therefore we may choose $U_x,\ep_x$ such that after an $\SU(m)$ coordinate change we may write
\begin{align*}
T_pL&=\bigl\{(x_1,\ldots,x_m):x_1,\ldots,x_m\in\R^m\bigr\},\\
T_pL^\dag&=\bigl\{(e^{i\ti\psi_1}x_1,\ldots,e^{i\ti\psi_m}x_m):x_1,\ldots,x_m\in\R^m\bigr\},
\end{align*}
where $\md{\ti\psi_j-\psi_j}<\min\bigl(\frac{\pi}{2m},\psi_{k+1},\ldots,\psi_m,\pi-\psi_{k+1},\ldots,\pi-\psi_m\bigr)$ for $j=1,\ldots,m$, and the phase function $\th_{L^\dag}$ of $L^\dag$ satisfies~$\md{\th_{L^\dag}(p)}<\frac{\pi}{2}$.

Then we have $\ti\psi_j\in(-\frac{\pi}{2m},\frac{\pi}{2m})$ for $j=1,\ldots,k$, and $\ti\psi_j\in(0,\pi)$ for $j=k+1,\ldots,m$. Also $\ti\psi_j\ne 0$, as $L,L^\dag$ intersect transversely. Write $b$ for the number of $j=1,\ldots,k$ with $\ti\psi_j\in(-\frac{\pi}{2m},0)$, where $0\le b\le k$. Then in the definition of $\mu_{L,L^\dag}(p)$ in Definition \ref{la2def4}, the angles between $T_pL$ and $T_pL^\dag$ are $\ti\psi_j$ for $m-b$ values of $j=1,\ldots,m$, and $\ti\psi_j+\pi$ for the remaining $b$ values of $j=1,\ldots,m$. Hence \eq{la2eq6} gives 
\e
\begin{split}
\mu_{L,L^\dag}(p)&=\ts\frac{1}{\pi}(\ti\psi_1+\cdots+\ti\psi_m+b\pi+\th_L(p)-\th_{L^\dag}(p))\\
&=a+b+\ts\frac{1}{\pi}\bigl[(\ti\psi_1-\psi_1)+\cdots+(\ti\psi_m-\psi_m)-\th_{L^\dag}(p)\bigr],
\end{split}
\label{la4eq34}
\e
using $\th_L(p)=0$ as $L$ is special Lagrangian and $\psi_1+\cdots+\psi_m=a\pi$. But the inequalities $\md{\ti\psi_j-\psi_j}<\frac{\pi}{2m}$ and $\md{\th_{L^\dag}(p)}<\frac{\pi}{2}$ imply that the term $\frac{1}{\pi}[\cdots]$ in the second line of \eq{la4eq34} is of size smaller than 1. Since the remaining terms in the equation are integers, we must have $\frac{1}{\pi}[\cdots]=0$. Hence $\mu_{L,L^\dag}(p)=a+b$. But $0<a<m-k$ and $0\le b\le k$ from above, so $0<\mu_{L,L^\dag}(p)<m$ if $0<\ep<\ep_x$ and $p\in L\cap L^\dag\cap U_x$, as we have to prove.
\end{proof}

\begin{lem} Suppose $x\in\C^m\sm(L\cap L')$. Then there exist\/ $\ep_x>0$ and an open neighbourhood\/ $U_x$ of\/ $x\in\C^m$ such that if\/ $0<\ep<\ep_x$ then\/~$L\cap L^\dag\cap U_x=\es$.
\label{la4lem4}
\end{lem}

\begin{proof} Choose an open neighbourhood $U_x$ of $x$ in $\C^m\sm L'$ whose closure $\kern .15em\ov{\kern -.15em U}_x$ in $\C^m$ is compact with $\kern .15em\ov{\kern -.15em U}_x\subset\C^m\sm L'$. Then $T=\bigl\{\ep\in [0,\ep_W]:\Psi(\Ga_{\ep\d H})\cap \kern .15em\ov{\kern -.15em U}_x=\es\bigr\}$ is an open subset of $[0,\ep_W]$, since $\kern .15em\ov{\kern -.15em U}_x$ is compact and $L^\dag=\Psi(\Ga_{\ep\d H})$ is closed in $\C^m$ and depends continuously on $\ep$, and $0\in T$ as $L'\cap \kern .15em\ov{\kern -.15em U}_x=\es$ with $L'=\Psi(\Ga_{0\d H})$. Hence there exists $\ep_x\in(0,\ep_W]$ with $[0,\ep_x)\subseteq T$. These $U_x,\ep_x$ satisfy the conclusions of the lemma.
\end{proof}

Let $R,V,\ep_\iy$ be as in Lemma \ref{la4lem2}. Write $\,\ov{\!B}_{2R}$ for the closed ball of radius $2R$ about 0 in $\C^m$. Then $\,\ov{\!B}_{2R}\sm\Psi(\pi^{-1}(V))$ is closed and bounded in $\C^m$, and so compact. If $x\in\,\ov{\!B}_{2R}\sm\Psi(\pi^{-1}(V))$ then either (a) $x\in L\cap L'$ with $T_xL\ne T_xL'$ or (b) $x\in\C^m\sm (L\cap L')$, as if $x\in L\cap L'$ with $T_xL=T_xL'$ then $x\in V\subset\Psi(\pi^{-1}(V))$. Lemmas \ref{la4lem3} and \ref{la4lem4} in cases (a),(b) respectively give open neighbourhoods $U_x$ of $x$ in $\C^m$ and constants $\ep_x>0$, such that if $p\in L\cap L^\dag\cap U_x$ then $0<\mu_{L,L^\dag}(p)<m$. By compactness we can choose a finite subset $\{x_1,\ldots,x_N\}$ with $\,\ov{\!B}_{2R}\sm\Psi(\pi^{-1}(V))\subset U_{x_1}\cup\cdots\cup U_{x_N}$. Define $\ep'=\min(\ep_W,\ep_\iy,\ep_{x_1},\ldots,\ep_{x_N})$.

Suppose now that $0<\ep<\ep'$, with $\ep$ generic. As above $L,L^\dag$ intersect transversely in finitely many points $p_1,\ldots,p_k$. Let $i=1,\ldots,k$. If $p_i\in\Psi(\pi^{-1}(V))$ then Lemma \ref{la4lem2} proves that $0<\mu_{L,L^\dag}(p_i)<m$. If $p_i\notin\Psi(\pi^{-1}(V))$ then $p_i\in\,\ov{\!B}_{2R}\sm\Psi(\pi^{-1}(V))$, since $L^\dag\sm\,\ov{\!B}_{2R}\subset\Psi(\pi^{-1}(V))$, so $p_i\in U_{x_j}$ for some $j=1,\ldots,N$. If $x_j\in L\cap L'$ with $T_{x_j}L\ne T_{x_j}L'$ then Lemma \ref{la4lem3} shows that $0<\mu_{L,L^\dag}(p_i)<m$. If $x_j\in\C^m\sm (L\cap L')$ then Lemma \ref{la4lem4} gives a contradiction. Hence $0<\mu_{L,L^\dag}(p_i)<m$ for all $i=1,\ldots,k$, as in Proposition~\ref{la4prop5}(b). 

So far, we have defined a small Hamiltonian perturbation $\bar L^\dag$ of $\bar L'$ in $(M,\om)$, such that $\bar L\cap\bar L^\dag=\{p_1,\ldots,p_k,\iy_0,\iy_{\bs\phi}\}$, where the intersections at $p_1,\ldots,p_k$ are transverse, and have $0<\mu_{\bar L,\bar L^\dag}(p_i)<m$ for $i=1,\ldots,k$. However, the intersections at $\iy_0,\iy_{\bs\phi}$ are not transverse, although they are isolated, so $\bar L''=\bar L^\dag$ does not satisfy Proposition \ref{la4prop5}(a). We deal with this in the next lemma, by perturbing $\bar L^\dag$ near $\iy_0,\iy_{\bs\phi}$ to make it intersect $\bar L$ transversely.

\begin{lem} We can choose a compact Lagrangian $\bar L''$ in $(M,\om)$ which is a small Hamiltonian perturbation of\/ $\bar L^\dag,$  with\/ $\bar L''=\bar L^\dag$ except near $\iy_0,\iy_{\bs\phi},$ such that\/ $\bar L,\bar L''$ intersect transversely in $p_1,\ldots,p_k,\iy_0,\iy_{\bs\phi},$ with\/ $0<\mu_{\bar L,\bar L''}(p_i)<m$ for $i=1,\ldots,k$ and\/ $\mu_{\bar L,\bar L''}(\iy_0),\mu_{\bar L,\bar L''}(\iy_{\bs\phi})\in\{0,m\}$.
\label{la4lem5}
\end{lem}

\begin{proof} As above $L,L^\dag$ differ near $\iy_0$ by the graph of the 1-form $\d\bigl(\ep H-G\bigr)$. The symplectic manifold $(M,\om)$ and Lagrangian spheres $\cS_0=\Pi_0\cup\{\iy_0\}$, $\cS_{\bs\phi}=\Pi_{\bs\phi}\cup\{\iy_{\bs\phi}\}$ were defined in Example \ref{la2ex4} in \S\ref{la24}. Writing $r:\Pi_0\ra[0,\iy)$ for the radius function on $\Pi_0$, and $\ti r:\cS_0\sm\{0\}\ra[0,\iy)$ for the distance function from $\iy_0$ in $\cS_0$, on $\Pi_0\sm\{0\}$ we have $\ti r=F(r)=1/\log(1+r^2)$ as in Example \ref{la2ex4}, so that $r=(e^{1/\ti r}-1)^{1/2}$. Hence, by \eq{la4eq19}, we see that $\bar L,\bar L^\dag$ differ near $\iy_0$ by the graph of an exact 1-form $\d K_0$ on $\bar L$, where
\e
\begin{aligned} &K_0=(\ep-F_0(0,\ldots,0))e^{(2-m)/2\ti r}+O(e^{(1-m)/2\ti r}),\\ 
&\nabla^k(K_0-(\ep-F_0(0,\ldots,0))e^{(2-m)/2\ti r})=O(\ti r^{-2k}e^{(1-m)/2\ti r}),\end{aligned}
\label{la4eq35}
\e
with $\ep-F_0(0,\ldots,0)\ne 0$ as $\ep$ is generic. The analogue applies near~$\iy_{\bs\phi}$.

Note that $K_0$ has an isolated local minimum (if $\ep>F_0(0,\ldots,0)$) or local maximum (if $\ep<F_0(0,\ldots,0)$) at $\iy_0$, although this minimum or maximum is not Morse. Choose a smooth function $\ti K_0$ defined near $\iy_0$ on $\bar L$, such that 
\e
\ti K_0=\begin{cases} (\ep-F_0(0,\ldots,0))\ti r_0^{-2}e^{(2-m)/2\ti r_0} \ti r^2, & \ti r\le \ti r_0, \\
\text{has no critical points,} & \ti r_0\le\ti r\le 2\ti r_0, \\
K_0, & \ti r\ge 2\ti r_0,
\end{cases}
\label{la4eq36}
\e
for $\ti r_0>0$ very small, with $\ti K_0-K_0$ small in $C^1$. To do this, write $E$ for the $O(e^{(1-m)/2\ti r})$ error term on the first line of \eq{la4eq35}, and set
\begin{equation*}
\ti K_0=(\ep-F_0(0,\ldots,0))h(\ti r)+\eta(\ti r/\ti r_0)\cdot E\quad\text{in $\ti r_0\le\ti r\le 2\ti r_0$,} 
\end{equation*}
for $h:[\ti r_0,2\ti r_0]\ra\R$ smooth interpolating between $h(\ti r)=\ti r_0^{-2}e^{(2-m)/2\ti r_0}\ti r^2$ near $\ti r=\ti r_0$ and $h(\ti r)=e^{(2-m)/2\ti r}$ near $\ti r=2\ti r_0$, with $\frac{\d h}{\d\ti r}(\ti r)\ne 0$, and $\eta:[1,2]\ra[0,1]$ smooth with $\eta=0$ near 1 and $\eta=1$ near 2. Then in $\ti r_0\le\ti r\le 2\ti r_0$ we have
\begin{equation*}
\d\ti K_0=\ts(\ep-F_0(0,\ldots,0))\frac{\d h}{\d\ti r}(\ti r)\,\d\ti r+\eta(\ti r/\ti r_0)\,\d E+E\,\ti r_0^{-1}\frac{\d\eta}{\d t}(\ti r/\ti r_0)\,\d\ti r.
\end{equation*}
The first term is roughly of size $(\ep-F_0(0,\ldots,0))\ti r_0^{-2}e^{(2-m)/2\ti r_0}$, and the second and third terms are $O(\ti r_0^{-2}e^{(1-m)/2\ti r_0})$ by \eq{la4eq35}. So for small $\ti r_0$ the first term dominates, and $\ti K_0$ has no critical points in~$\ti r_0\le\ti r\le 2\ti r_0$. Choose functions $K_{\bs\phi},\ti K_{\bs\phi}$ on $\bar L$ near $\iy_{\bs\phi}$ in the same way. 

Now define $\bar L''$ to be equal to $\bar L^\dag$ away from $\iy_0,\iy_{\bs\phi}$, to be the graph of $\d\ti K_0$ over $\bar L$ near $\iy_0$, and to be the graph of $\d\ti K_{\bs\phi}$ over $\bar L$ near $\iy_{\bs\phi}$. Then $\bar L,\bar L''$ intersect transversely at $\iy_0$, and the computation after \eq{la4eq21} shows that $\mu_{\bar L,\bar L''}(\iy_0)=0$ if $\ep>F_0(0,\ldots,0)$, when $\ti K_0$ has a Morse local minimum, and $\mu_{\bar L,\bar L''}(\iy_0)=m$ if $\ep<F_0(0,\ldots,0)$, when $\ti K_0$ has a Morse local maximum, and similarly for $\iy_{\bs\phi}$. By the second line of \eq{la4eq36}, the perturbation of $\bar L^\dag$ to $\bar L''$ introduces no extra intersection points with $\bar L$. Thus $\bar L\cap\bar L''=\{p_1,\ldots,p_k,\iy_0,\iy_{\bs\phi}\}$. Also $\bar L,\bar L''$ intersect transversely at $p_1,\ldots,p_k$ with $\mu_{\bar L,\bar L''}(p_i)=\mu_{\bar L,\bar L^\dag}(p_i)$, since $\bar L''=\bar L^\dag$ near $p_i$, so $0<\mu_{\bar L,\bar L''}(p_i)<m$.
\end{proof}

Lemma \ref{la4lem5} proves Proposition \ref{la4prop5}(a),(b) and the first part of (c). For the second part of (c), since $\ti K_0(\iy_0)=K_0(\iy_0)=0$ and $\ti K_{\bs\phi}(\iy_{\bs\phi})=K_{\bs\phi}(\iy_{\bs\phi})=0$, the potential $f_{\bar L''}$ of $\bar L''$ satisfies $f_{\bar L''}(\iy_0)=f_{\bar L^\dag}(\iy_0)=0$, $f_{\bar L''}(\iy_{\bs\phi})=f_{\bar L^\dag}(\iy_{\bs\phi})=A(L')$, so that \eq{la4eq3} holds, and $f_{\bar L''}=f_{\bar L^\dag}$ away from $\iy_0,\iy_{\bs\phi}$ where $\bar L'',\bar L^\dag$ coincide. For the last part of the proposition, with $L=\Pi_0\cup\Pi_\phi$, note that we can suppose $0\notin L^\dag$, as $L^\dag$ is a generic Hamiltonian perturbation of $L'$ away from infinity in $\C^m$. So $L$ having a self-intersection point at 0 makes no difference to the intersections of $\bar L,\bar L^\dag$ or $\bar L,\bar L''$. This finally completes the proof of Proposition~\ref{la4prop5}.
\end{proof}

\subsection{Perturbing $\bar L,\bar L'$ to transverse: LMCF expander case}
\label{la44}

Here is an analogue of Proposition \ref{la4prop5} for Lagrangian MCF expanders, showing that we can perturb $\bar L,\bar L'$ to intersect transversely, and still satisfy the analogue \eq{la4eq37} of Proposition \ref{la4prop2}(a). Note that Proposition \ref{la4prop2}(b) follows directly from Proposition \ref{la4prop2}(a), so it also holds for~$\bar L,\bar L''$ below.

\begin{prop} Suppose $L,L'$ are distinct, closed, graded, AC
Lagrangian MCF expanders in $\C^m$ asymptotic at rate $\rho<0$ to
the cone $C=\Pi_0\cup\Pi_{\bs\phi},$ and both satisfying the
expander equation $H=\al F^\perp$ with the same value of\/ $\al>0$.
Then {\rm\S\ref{la24}} defines $(M,\om),\la$ such that\/ $L,L'$
extend naturally to compact, exact, graded Lagrangians $\bar
L=L\cup\{\iy_0,\iy_{\bs\phi}\}$ and\/ $\bar
L'=L'\cup\{\iy_0,\iy_{\bs\phi}\}$ in $M,$ with potentials $f_{\bar
L}=-2\th_{\bar L}/\al,$ $f_{\bar L'}=-2\th_{\bar L'}/\al$ which by \eq{la2eq15} satisfy
\begin{equation*}
f_{\bar L}(\iy_0)=f_{\bar L'}(\iy_0)=0,\quad
f_{\bar L}(\iy_{\bs\phi})=A(L),\quad f_{\bar L'}(\iy_{\bs\phi})=A(L').
\end{equation*}

Then we may choose a small Hamiltonian perturbation $\bar L''$ of\/ $\bar L'$ in\/ $(M,\om),$ with potential\/ $f_{\bar L''},$ satisfying the conditions:
\begin{itemize}
\setlength{\parsep}{0pt}
\setlength{\itemsep}{0pt}
\item[{\bf(a)}] $\bar L$ and\/ $\bar L''$ intersect
transversely in $M$.
\item[{\bf(b)}] For each\/ $p\in \bar L\cap \bar L''\cap\C^m,$
as in equation \eq{la4eq2} we have
\e
\frac{\al}{2\pi}(f_{\bar L''}(p)-f_{\bar L}(p))<\mu_{\bar L,\bar L''}(p)<
\frac{\al}{2\pi}(f_{\bar L''}(p)-f_{\bar L}(p))+m.
\label{la4eq37}
\e
\item[{\bf(c)}] We may also have one or both of\/ $\iy_0,\iy_{\bs\phi}$ in $\bar L\cap \bar L''$. We do not require \eq{la4eq37} to hold at\/ $\iy_0,\iy_{\bs\phi},$ but instead we have
\e
\begin{aligned}
\mu_{\bar L,\bar L''}(\iy_0)&\in\{0,m\}, & f_{\bar L''}(\iy_0)&=0
&&\text{if\/ $\iy_0\in \bar L\cap \bar L''$,}\\
\mu_{\bar L,\bar L''}(\iy_{\bs\phi})&\in\{0,m\}, & f_{\bar
L''}(\iy_{\bs\phi})&=A(L')
&&\text{if\/ $\iy_{\bs\phi}\in \bar L\cap \bar L''$.}
\end{aligned}
\label{la4eq38}
\e
\end{itemize}
The analogues of\/ {\bf(a)--(c)} also hold when $L=\Pi_0\cup\Pi_{\bs\phi},$ so that\/ $\bar L=\cS_0\cup\cS_\phi,$ with potential\/ $f_{\bar L}\vert_{\cS_0}=0$ and\/ $f_{\bar L}\vert_{\cS_{\bs\phi}}=A(L'),$ although in this case\/ $L$ is immersed rather than embedded and\/ $A(L)$ is not defined.
\label{la4prop6}
\end{prop}

\begin{proof} The proof follows that of Proposition \ref{la4prop5} in \S\ref{la43} fairly closely, so we will just explain the modifications to the proof in \S\ref{la43}.

First we follow the proof of Proposition \ref{la4prop5} up to Lemma \ref{la4lem1}, defining $z(L')\subset W\subset T^*L$, $\Psi:W\ra\C^m$, $V\subset L'$, $R>0$, $L_0',L_{\bs\phi}'\subset V$, and $G:V\ra\R$ with $\Psi(\Ga_{\d G})\subset L$. There are two differences with \S\ref{la43}: firstly, rather than \eq{la4eq6}--\eq{la4eq7}, $G$ satisfies 
\e
\begin{split}
g^{\prime ab}\nabla_a\nabla_bG(\bs y)
&+\al\bigl(g^{\prime ab}\nabla_a (\ha r^2)\nabla_bG-2G\bigr)\\
&+P\bigl(\bs y,\nabla_a G(\bs y),\nabla_a\nabla_bG(y)\bigr)=0,
\end{split}
\label{la4eq39}
\e
as in \eq{la3eq19}--\eq{la3eq20}, where $P$ is a nonlinear function of its arguments such that
\e
P(\bs y,\al,\be)=O\bigl((1+\md{\bs y})^{-m}e^{-\al \ms{\bs y}/2}\ms{\al}+\ms{\be}\bigr).
\label{la4eq40}
\e
Secondly, rather than using Theorem \ref{la3thm4} to write $G\vert_{L_0'}$ as in \eq{la4eq9}, and similarly for $G\vert_{L_{\bs\phi}'}$, we use Theorem \ref{la3thm6} to write
\e
G\vert_{L_0'}(x_1,\ldots,x_m)=r^{-1-m}e^{-\al r^2/2}F_0\bigl((x_1/r,\ldots,x_m/r),r^{-2}\bigr),
\label{la4eq41}
\e
where $F_0:\cS^{m-1}\t[0,R^{-2})\ra\R$ is smooth on $\cS^{m-1}\t[0,R^{-2})$ and real analytic on $\cS^{m-1}\t(0,R^{-2})$, and defining $F_0^\iy(\bs y)=F_0\bigl(\bs y,0)$, then $F_0^\iy:\cS^{m-1}\ra\R$ is also real analytic.

The analogue of Lemma \ref{la4lem1} is:

\begin{lem} Let\/ $C_1,C_2>0$ be given. Then making $R\gg 0$ larger if necessary, there exists\/ $\ep_\iy>0$ such that there are no $q\in L_0'\cup L_{\bs\phi}'$ and\/ $0\!<\!\ep\!<\!\ep_\iy$ satisfying
\begin{align*}
\bnm{\ep^{-1}\nabla_aG(q)-\nabla_a(r^{-1-m}e^{-\al r^2/2})(q)}_{g'}&<C_1(r^{-1-m}e^{-\al r^2/2})(q)\quad\text{and\/}\\
\bnm{\ep^{-1}\nabla_a\nabla_bG(q)\!-\!\nabla_a\nabla_b(r^{-1-m}e^{-\al r^2/2})(q)}_{g'}&<C_2(r^{-m}e^{-\al r^2/2})(q).
\end{align*}

\label{la4lem6}
\end{lem}

\begin{proof} The proof is similar to that of Lemma \ref{la4lem1}, so we just explain the differences. The analogue of equation \eq{la4eq12} is
\begin{align*}
T=\bigl\{&((y_1,\ldots,y_m),t,\ep)\in\cS^{m-1}\t(0,R^{-2})\t(0,\iy):\\
&\bnm{\ep^{-1}\nabla_aG(q)-\nabla_a(r^{-1-m}e^{-\al r^2/2})(q)}_{g'}<C_1(r^{-1-m}e^{-\al r^2/2})(q),\\
&\bnm{\ep^{-1}\nabla_a\nabla_bG(q)\!-\!\nabla_a\nabla_b(r^{-1-m}e^{-\al r^2/2})(q)}_{g'}<C_2(r^{-m}e^{-\al r^2/2})(q),\\
&q=(t^{-1/2}y_1,\ldots,t^{-1/2}y_m)\bigr\}.
\end{align*}
Rewriting this using \eq{la4eq41} yields
\ea
T&=\bigl\{(\bs y,t,\ep)\in\cS^{m-1}\t(0,R^{-2})\t(0,\iy):
\label{la4eq42}\\
&\bnm{\ep^{-1}\nabla_aF_0(\bs y,t)+\bigl(\ep^{-1}F_0(\bs y,t)-1\bigr)\bigl(\ts\frac{m+1}{2t}+\frac{\al}{2t^2}\bigr)\nabla_at}{}^2_{g'}<C_1^2,
\nonumber\\
&\bnm{\ep^{-1}\nabla_a\nabla_bF_0(\bs y,t)+
\ep^{-1}\bigl(\ts\frac{m+1}{2t}+\frac{\al}{2t^2}\bigr)\bigl(\nabla_aF_0(\bs y,t)\nabla_bt+\nabla_bF_0(\bs y,t)\nabla_at\bigr)
\nonumber\\
&+\!\bigl(\ep^{-1}F_0(\bs y,t)\!-\!1\bigr)\bigl[
\bigl(\ts\frac{m+1}{2t}\!+\!\frac{\al}{2t^2}\bigr)\nabla_a\nabla_bt
\!-\!(\frac{m+1}{2t^2}\!+\!\frac{\al}{t^3})\nabla_at\nabla_bt\bigr]}{}^2_{g'}
\!<\!C_2^2t^{-1}\bigr\}.
\nonumber
\ea
Consider $T$ as a subset of $\cS^{m-1}\t[0,R^{-2})\t[0,\iy)$. If $F_0:\cS^{m-1}\t[0,R^{-2})\ra\R$ above were real analytic, then $T$ in \eq{la4eq42} would be defined by real analytic inequalities in $\cS^{m-1}\t[0,R^{-2})\t[0,\iy)$. In fact we only know that $F_0$ is smooth, and real analytic on $\cS^{m-1}\t(0,R^{-2})$ and $\cS^{m-1}\t\{0\}$. But since  $F_0$ is determined by real analytic functions $\bar h^\iy,\bar h^{\prime\iy}:\cS^{m-1}\ra\R$ as in Theorem \ref{la3thm6}, $T$ cannot have pathological behaviour at its $t=0$ boundary.

Therefore if $(\bs y,0,0)$ lies in the closure $\bar T$ of $T$ in $\cS^{m-1}\t[0,R^{-2})\t[0,\iy)$, then we can find a real analytic curve $\bigl(\bs y(s),t(s),\ep(s)\bigr):[0,\de)\ra\bar T$ with $\bigl(\bs y(0),t(0),\ep(0)\bigr)=(\bs y,0,0)$ and $\bigl(\bs y(s),t(s),\ep(s)\bigr)\in T$ for $s\in(0,\de)$. Then by studying the leading order behaviour of the function $s\mapsto\ep(s)^{-1}F_0\bigl(\bs y(s),t(s)\bigr)$ at $s=0$, as in \eq{la4eq13}--\eq{la4eq15} we can derive a contradiction using the inequalities in \eq{la4eq42}. Hence $(\bs y,0,0)\notin\bar T$ for any $\bs y\in\cS^{m-1}$, so by compactness of $\cS^{m-1}$, making $R\gg 0$ larger, there exists $\ep_\iy>0$ with $T\cap \bigl(\cS^{m-1}\t(0,R^{-2})\t(0,\ep_\iy)\bigr)=\es$. The lemma follows.
\end{proof}

Next, in an analogue of \eq{la4eq16}, choose $H:L'\ra\R$ to be a generic smooth function satisfying
\e
\nabla^k\bigl(H-r^{-1-m}e^{-\al r^2/2}\bigr)=O(r^{k-3-m}e^{-\al r^2/2})
\label{la4eq43}
\e
as $r\ra\iy$ for all $k=0,1,\ldots.$ By genericity \eq{la4eq17} holds. As in \S\ref{la43}, there exists $\ep_W>0$ such that if $0<\ep<\ep_W$ then $\Ga_{\ep\d H}\subset W$. Let $\ep\in(0,\ep_W)$ be generic and satisfy a finite number of smallness conditions chosen later, and define $L^\dag=\Psi(\Ga_{\ep\d H})$, so that $L^\dag$ is an AC Lagrangian submanifold of $\C^m$ with cone $C=\Pi_0\cup\Pi_{\bs\phi}$ and any rate $\rho<2$. Therefore $\bar L^\dag=L^\dag\cup\{\iy_0,\iy_{\bs\phi}\}$ is a compact, smooth Lagrangian in $(M,\om)$, a Hamiltonian perturbation of~$\bar L'$. 

Equation \eq{la4eq18} still holds. An analogue of \eq{la4eq19} is
\ea
\begin{split}
&\ts\frac{\pd}{\pd r}\bigl(\ep\d H-\d G\bigr)(x_1,\ldots,x_m)\\
&=-\al r^{-m}e^{-\al r^2/2}\bigl(\ep-F_0^\iy(x_1/r,\ldots,x_m/r)\bigr)\d r
+O(r^{-m-2}e^{-\al r^2/2}),
\end{split}
\label{la4eq44}\\
\begin{split}&\d_{\cS^{m-1}}\bigl(\ep\d H-\d G\bigr)(x_1,\ldots,x_m)\\
&=-r^{-1-m}e^{-\al r^2/2}\d_{\cS^{m-1}}F_0^\iy(x_1/r,\ldots,x_m/r)
+O(r^{-m-3}e^{-\al r^2/2}),
\end{split}
\label{la4eq45}
\ea
where $\d_{\cS^{m-1}}$ in \eq{la4eq45} means the derivative in the $\cS^{m-1}$ directions in $\R^m\sm\,\ov{\!B}_R\cong (R,\iy)\t\cS^{m-1}$. Since $\ep$ is assumed generic, there are no points $\bs y\in\cS^{m-1}$ such that $\ep-F_0^\iy(\bs y)=0$ and $\d_{\cS^{m-1}}F_0^\iy(\bs y)=0$. Therefore the leading terms in \eq{la4eq44}--\eq{la4eq45} dominate the error terms $O(\cdots)$ for $r\gg 0$, and $\ep\d H-\d G$ has no zeroes near infinity in $L_0'$, and similarly for $L_{\bs\phi}'$. So \eq{la4eq18} implies that $L\cap L^\dag$ is bounded in $\C^m$. The proof in \S\ref{la43} now shows that $L,L^\dag$ intersect transversely in finitely many points $p_1,\ldots,p_k$ in~$\C^m$.

The analogues of Lemmas \ref{la4lem2}--\ref{la4lem4} are:

\begin{lem} Making $R\gg 0$ larger and\/ $V\subset L'$ smaller if necessary, if\/ $p\in L\cap L^\dag\cap \Psi(\pi^{-1}(V))$ and\/ $\ep<\ep_\iy$ for some $\ep_\iy>0$ then
\e
\frac{\al}{2\pi}(f_{L^\dag}(p)-f_L(p))<\mu_{L,L^\dag}(p)<
\frac{\al}{2\pi}(f_{L^\dag}(p)-f_L(p))+m.
\label{la4eq46}
\e

\label{la4lem7}
\end{lem}

\begin{proof} We first follow a very similar proof to Lemma \ref{la4lem2}, using \eq{la4eq43}, Lemma \ref{la4lem6}, and \eq{la4eq39}--\eq{la4eq40} in place of \eq{la4eq16}, Lemma \ref{la4lem1}, and \eq{la4eq6}--\eq{la4eq7}, to show that
making $R$ larger and $V$ smaller, for some $\ep_\iy>0$ we have $0<\mu_{L,L^\dag}(p)<m$ for all $p$ as in the lemma, provided $\ep<\ep_\iy$. 

Then we note that making $R$ larger and $V,\ep_\iy$ smaller, we can suppose $\md{f_{L^\dag}-f_L}$ is arbitrarily small on $V$. This is because $V\subset L'$ is defined to be the union of $\bigl\{p\in L':r(p)>R\bigr\}$ and an open neighbourhood $V_x$ of each $x\in L\cap L'$ with $T_xL=T_xL'$. As $f_{L^\dag},f_L =O(r^{-1-m}e^{-\al r^2/2})$, by increasing $R$ we can make $\md{f_{L^\dag}-f_L}$ arbitrarily small on $\bigl\{p\in L':r(p)>R\bigr\}$. If $x\in L\cap L'$ with $T_xL=T_xL'$ then $f_L(x)=-2\th_L(x)/\al=-2\th_{L'}(x)/\al=f_{L'}(x)$. Making the open neighbourhood $V_x$ smaller, we can suppose $\md{f_{L'}-f_L}$ is small on $V_x$, and making $\ep_\iy$ smaller, we can suppose $\md{f_{L^\dag}-f_{L'}}$ is small on $V_x$. Thus, making $R$ larger and $V,\ep_\iy$ smaller, we can suppose that $\bmd{\frac{\al}{2\pi}(f_{L^\dag}(p)-f_L(p))}<1$ for all $p$ as in the lemma. Then \eq{la4eq46} follows from this and~$0<\mu_{L,L^\dag}(p)<m$.
\end{proof}

\begin{lem} Suppose $x\in L\cap L'$ with\/ $T_xL\ne T_xL'$. Then there exist\/ $\ep_x>0$ and an open neighbourhood\/ $U_x$ of\/ $x\in\C^m$ such that if\/ $0<\ep<\ep_x$ and\/ $p\in L\cap L^\dag\cap U_x$ then\/ \eq{la4eq46} holds.
\label{la4lem8}
\end{lem}

\begin{proof} After applying a $\U(m)$ coordinate change to $\C^m$ we may take $T_xL,T_xL'$ to be given by \eq{la4eq33} for unique $0\!=\!\psi_1\!=\!\cdots\!=\!\psi_k\!<\!\psi_{k+1}\!\le\!\cdots\!\le\!\psi_m\!<\!\pi$, where $0\le k<m$ as $T_xL\ne T_xL'$. In Lemma \ref{la4lem3} we had $\psi_1+\cdots+\psi_m=a\pi$ for $a\in\Z$ with $0<a<m-k$ as $T_xL,T_xL'$ were special Lagrangian. We modify this to
\begin{equation*}
a\pi=\psi_1+\cdots+\psi_m+\th_L(x)-\th_{L'}(x)=
\psi_1+\cdots+\psi_m+\ts\frac{\al}{2}(f_{L'}(p)-f_L(p))
\end{equation*}
for $a\in\Z$, since $f_L=-2\th_L/\al$, $f_{L'}=-2\th_{L'}/\al$ as $L,L'$ are LMCF expanders. As $0<\psi_1+\cdots+\psi_m<(m-k)\pi$, this implies that
\e
\ts\frac{\al}{2\pi}(f_{L'}(x)-f_L(x))<a<m-k+\ts\frac{\al}{2\pi}(f_{L'}(x)-f_L(x)).
\label{la4eq47}
\e

The proof of Lemma \ref{la4lem3} now shows that we can choose open $x\in U_x\subset\C^m$ and $\ep_x>0$ such that if $0<\ep<\ep_x$ and $p\in L\cap L^\dag\cap U_x$ then $\mu_{L,L^\dag}(p)=a+b$ for some $b$ with $0\le b\le k$. Adding \eq{la4eq47} and $0\le b\le k$ gives
\e
\ts\frac{\al}{2\pi}(f_{L'}(x)-f_L(x))<\mu_{L,L^\dag}(p)<m+\ts\frac{\al}{2\pi}(f_{L'}(x)-f_L(x)).
\label{la4eq48}
\e
Making $U_x,\ep_x$ smaller we can suppose that $\md{f_{L'}(p)-f_{L'}(x)}$, $\md{f_L(p)-f_L(x)}$, $\md{f_{L^\dag}(p)-f_{L'}(p)}$ are arbitrarily small for all $p\in U_x$. Thus we can make the left and right hand sides of \eq{la4eq46} and \eq{la4eq48} arbitrarily close, so \eq{la4eq48} implies \eq{la4eq46} for small enough $U_x,\ep_x$, proving the lemma.
\end{proof}

\begin{lem} Suppose $x\in\C^m\sm(L\cap L')$. Then there exist\/ $\ep_x>0$ and an open neighbourhood\/ $U_x$ of\/ $x\in\C^m$ such that if\/ $0<\ep<\ep_x$ then\/~$L\cap L^\dag\cap U_x=\es$.
\label{la4lem9}
\end{lem}

\begin{proof} The proof of Lemma \ref{la4lem4} works without change.
\end{proof}

The method of \S\ref{la43}, but using Lemmas \ref{la4lem7}--\ref{la4lem9}, now shows that if $\ep>0$ is generic and sufficiently small then \eq{la4eq46} holds for all $p\in L\cap L^\dag=\{p_1,\ldots,p_k\}$.

So far, we have constructed a small Hamiltonian perturbation $\bar L^\dag$ of $\bar L'$ in $(M,\om)$, such that $\bar L\cap\bar L^\dag=\{\iy_0,\iy_{\bs\phi},p_1,\ldots,p_k\}$, where the intersections at $p_1,\ldots,p_k$ are transverse, and satisfy Proposition \ref{la4prop6}(b). However, the intersections at $\iy_0,\iy_{\bs\phi}$ are not transverse, although they are isolated, so $\bar L''=\bar L^\dag$ does not satisfy Proposition~\ref{la4prop6}(a).

As in \S\ref{la43}, we will deal with this by constructing a small Hamiltonian perturbation $\bar L''$ of $\bar L^\dag$ with $\bar L''=\bar L^\dag$ except near $\iy_0,\iy_{\bs\phi}$, such that $\bar L,\bar L''$ intersect transversely. However, the process for doing this is more complex than in \S\ref{la43}, because the asymptotic model \eq{la4eq44}--\eq{la4eq45} for $\ep\d H-\d G$ near infinity is more complex than the asymptotic model \eq{la4eq19} in \S\ref{la43}. The reason for the complexity is discussed in Remark~\ref{la3rem}.

The symplectic manifold $(M,\om)$ and Lagrangian spheres $\cS_0=\Pi_0\cup\{\iy_0\}$, $\cS_{\bs\phi}=\Pi_{\bs\phi}\cup\{\iy_{\bs\phi}\}$ were defined in Example \ref{la2ex4} in \S\ref{la24}. In the coordinates $(\ti x_1,\ldots,\ti x_m)$ on $\cS_0$ near $\iy_0$, writing $\ti r=(\ti x_1^2+\cdots+\ti x_m^2)^{1/2}=1/\log(1+r^2)$, so that $r=(e^{1/\ti r}-1)^{1/2}$, by \eq{la4eq44}--\eq{la4eq45}, we see that $\bar L,\bar L^\dag$ differ near $\iy_0$ by the graph of an exact 1-form $\d K_0$ on $\bar L$, where in an analogue of \eq{la4eq35} we have
\e\label{la4eq49}\begin{split}
K_0(\ti x_1,\ldots,\ti x_m)&=e^{-(m+1)/2\ti r-\al(e^{1/\ti r}-1)/2}\bigl[\ep-F_0^\iy\bigl(\ts\frac{\ti x_1}{\ti r},\ldots,\frac{\ti x_m}{\ti r}\bigr)\bigr] \\
&\qquad\quad +\text{higher order terms}.
\end{split}\e
The `higher order terms' and their derivatives are small enough that they will not affect the argument that follows. As above there are no $\bs y\in\cS^{m-1}$ with $\ep-F_0^\iy(\bs y)=0$ and $\d_{\cS^{m-1}}\bigl(\ep-F_0^\iy(\bs y)\bigr)=0$. In particular, this means that the maximum and minimum values of $\ep-F_0^\iy$ are nonzero.

Divide into cases
\begin{itemize}
\setlength{\parsep}{0pt}
\setlength{\itemsep}{0pt}
\item[(i)] $\ep>F_0^\iy(\bs y)$ for all $y\in\cS^{m-1}$;
\item[(ii)] $\ep<F_0^\iy(\bs y)$ for all $y\in\cS^{m-1}$; and
\item[(iii)] $\ep-F_0^\iy(\bs y)$ is positive, and negative, and zero, somewhere on $\cS^{m-1}$.
\end{itemize}
In case (i), as in \S\ref{la43} we can choose a smooth perturbation $\ti K_0$ of $K_0$, such that $\ti K_0=K_0$ except near $\iy_0$, and $\ti K_0$ is Morse with one local minimum at $\iy_0$ with $\ti K_0(\iy_0)=0$ and no other critical points near $\iy_0$, and $\ti K_0-K_0$ is $C^1$-small. Case (ii) is the same as (i), except that $\ti K_0$ has a local maximum at~$\iy_0$.

In case (iii), observe that on a small closed ball $\,\ov{\!B}_\de(\iy_0)$ about $\iy_0$ in $\cS_0$, $K_0\vert_{\,\ov{\!B}_\de(\iy_0)}$ attains its global maximum and minimum on the boundary $\pd \,\ov{\!B}_\de(\iy_0)$, and no critical points of $K_0$ lie on $\pd \,\ov{\!B}_\de(\iy_0)$. Thus by Theorem \ref{laAthm2} with $f=K_0$ and $M=\,\ov{\!B}_\de(\iy_0)$,
we may choose a perturbation $\widehat K_0$ of $K_0$, such that $\widehat K_0-K_0$ is supported on the open ball $B_\de(\iy_0)$, and $\widehat K_0$ is Morse on $B_\de(\iy_0)$ with no critical points of index 0 or $m$. By perturbing $\widehat K_0$ slightly we can also suppose that $\iy_0$ is not a critical point of~$\widehat K_0$.

Now we have a minor problem: we would like to define $\bar L''$ near $\iy_0$ to be $\Phi(\Ga_{\smash{\d\widehat K_0}})$, for $\Phi$ a Lagrangian Neighbourhood of $\bar L$ in $(M,\om)$, but this is only well defined if $\widehat K_0-K_0$ is sufficiently small in $C^1$, so that $\Ga_{\smash{\d\widehat K_0}}$ lies in the domain of $\Phi$, and $\bar L''$ stays close to $\bar L'$. However, the proof of Theorem \ref{laAthm2} gives no control over the $C^1$-norm of~$\widehat K_0-K_0$.

To define an improved version $\ti K_0$ of $\widehat K_0$ such that $\ti K_0-K_0$ is $C^1$-small, choose small $\ep\in(0,\de)$ and a modification $\ti K_0$ of $K_0$ with the following properties:
\begin{itemize}
\item[(A)] $\ti K_0(\bs{\ti x})=\frac{f(\ep)}{f(\de)}\widehat K_0(\frac{\de}{\ep}\bs{\ti x})$ on $B_\ep(\infty_0)$,
where $f(\ti r)=e^{-(m+1)/2\ti r-\al(e^{1/\ti r}-1)/2}$;
\item[(B)] $\ti K_0=\ti f(\ti r)\bigr(\ep-F_0^\infty(\bs{\ti x}/\ti r)\bigr)+\text{(error terms)}$ on $B_\de(\infty_0)\sm B_\ep(\infty_0)$, for some strictly-increasing function $\ti f(\ti r)$ interpolating between $f(\ti r)$ near $\ti r=\de$ and $\ti f(\ti r)=\frac{f(\ep)}{f(\de)}f(\frac{\de}{\ep}\ti r)$ near $\ti r=\ep$, and the `error terms' are $C^1$-small and cause no critical points of $\ti K_0$ in this region for small $\de,\ep$; and 
\item[(C)] $\ti K_0=K_0$ outside $B_\de(\infty_0)$.
\end{itemize}

This may be done as for \eq{la4eq36}, using \eq{la4eq49}. In (A) we note that $\widehat K_0(\frac{\ti r}{\ep})$ has $C^1$-norm of order $\ep^{-1}$ as $\ep\ra 0$, but $f(\ep)$ decays exponentially as $\ep\ra 0$, so that $\nm{(\ti K_0-K_0)\vert_{B_\ep(\infty_0)}}_{C^1}\ra 0$ as $\ep\ra 0$ with $\de>0$ fixed. In (B) we note that $\ti f(\ti r)$ is constructed from the exponentially decaying functions $f(\ti r)$ and $\ti f(\ti r)=\frac{f(\ep)}{f(\de)}f(\frac{\de}{\ep}\ti r)$. These imply $\ti K_0-K_0$ is $C^1$-small for small $\ep$. From (A) and $\widehat K_0$ Morse on $B_\de(\iy_0)$ with no critical points of index $0,m$ or at $\iy_0$, we see that $\ti K_0$ is too, and (B),(C) imply that $\ti K_0$ has no critical points near $\iy_0$ outside $B_\ep(\iy_0)$. This completes case~(iii).

We have written $\bar L^\dag$ near $\infty_0$ as the graph of an exact $1$-form $\d K_0$, and in each case (i)--(iii) we have constructed a Morse perturbation $\ti K_0$ of $K_0$ near $\infty_0$ with $\ti K_0-K_0$ $C^1$-small. Similarly, write $\bar L^\dag$ near $\iy_\phi$ as the graph of an exact 1-form $\d K_{\bs\phi}$, and choose a Morse perturbation $\ti K_{\bs\phi}$ of $K_\phi$ in the same way. Now as in Lemma \ref{la4lem5} define $\bar L''$ to be equal to $\bar L^\dag$ away from $\iy_0,\iy_{\bs\phi}$, to be the graph of $\d\ti K_0$ over $\bar L$ near $\iy_0$, and to be the graph of $\d\ti K_{\bs\phi}$ over $\bar L$ near~$\iy_{\bs\phi}$. 

We claim that $\bar L''$ satisfies Proposition \ref{la4prop5}(a)--(c). For $p\in\bar L\cap\bar L''$ not close to $\iy_0$ or $\iy_{\bs\phi}$, we have $\bar L''=\bar L^\dag$ near $p$, so (a),(b) hold at $p$ from above. Any $p\in\bar L\cap\bar L''$ close to $\iy_0$ lies in $\Crit(\ti K_0)$. Since $\ti K_0$ is Morse the intersection is transverse at $p$, as in (a). In case (i) the only $p$ is $p=\iy_0$, with $\mu_{\bar L,\bar L''}(\iy_0)=0$ as $\iy_0$ is a local minimum. As $\ti K_0(\iy_0)=K_0(\iy_0)=0$, the potential $f_{\bar L''}$ of $\bar L''$ satisfies $f_{\bar L''}(\iy_0)=f_{\bar L^\dag}(\iy_0)=0$, so the first line of \eq{la4eq38} holds. Case (ii) is the same except that $\mu_{\bar L,\bar L''}(\iy_0)=m$ as $\iy_0$ is a local maximum. 

In case (iii), any $p\in\bar L\cap\bar L''$ close to $\iy_0$ has $p\ne\iy_0$, so that $p\in\bar L\cap\bar L''\cap\C^m$, and $0<\mu_{\bar L,\bar L''}(p)<m$ since $\ti K_0$ is Morse near $\iy_0$ with no local maxima or minima. As in the proof of Lemma \ref{la4lem8}, we can suppose $\md{f_{\bar L''}-f_{\bar L}}$ is arbitrarily small near $\iy_0$, so in particular $\bmd{\frac{\al}{2\pi}(f_{\bar L''}(p)-f_{\bar L}(p))}<1$. This and $0<\mu_{\bar L,\bar L''}(p)<m$ imply that \eq{la4eq37} holds. Thus, any $p\in\bar L\cap\bar L''$ close to $\iy_0$ satisfies Proposition \ref{la4prop6}(b) or (c). The proof for $p\in\bar L\cap\bar L''$ close to $\iy_{\bs\phi}$ is the same, except that we use $\ti K_{\bs\phi},K_{\bs\phi}$ rather than $\ti K_0,K_0$, and in cases (i),(ii) since $\ti K_0(\iy_{\bs\phi})=K_0(\iy_{\bs\phi})=0$, the potential $f_{\bar L''}$ of $\bar L''$ satisfies $f_{\bar L''}(\iy_{\bs\phi})=f_{\bar L^\dag}(\iy_{\bs\phi})=A(L')$, so that the second line of \eq{la4eq38} holds. The last part is proved as in \S\ref{la43}. This completes the proof of Proposition~\ref{la4prop6}.
\end{proof}

\subsection{Restrictions on $A(L),$ and on cones
$C=\Pi_0\cup\Pi_{\bs\phi}$}
\label{la45}

Applying Corollary \ref{la2cor} with $L$ special Lagrangian, so that
$\th_L=0,$ constrains the cones $C=\Pi_0\cup\Pi_{\bs\phi}$ possible
for AC SL $m$-folds. Note that cases (a),(b) are exactly the cones
realized by the Lawlor necks $L_{\bs\phi,A},\ti L_{\bs\phi,A}$ in
Example~\ref{la3ex1}.

\begin{cor} Let\/ $m\ge 3,$ and\/ $L$ be an exact AC SL\/ $m$-fold
in $\C^m$ asymptotic at rate $\rho<0$ to the cone
$C=\Pi_0\cup\Pi_{\bs\phi}$ defined in \eq{la4eq1} using
$\phi_1,\ldots,\phi_m\in(0,\pi),$ with compactification $\bar L$ in
$(M,\om)$ as in {\rm\S\ref{la24}}. Then $H^*(L;\Z_2)\cong
H^*(\cS^{m-1}\t\R;\Z_2),$ so $L$ is connected, and either:
\begin{itemize}
\setlength{\parsep}{0pt}
\setlength{\itemsep}{0pt}
\item[{\bf(a)}] $\phi_1+\cdots+\phi_m=\pi,$ and\/ $L$ satisfies
Corollary\/ {\rm\ref{la2cor}(a),} and\/ $\bar L$ satisfies
Theorem\/ {\rm\ref{la2thm3}(a);} or
\item[{\bf(b)}] $\phi_1+\cdots+\phi_m=(m-1)\pi,$ and\/ $L$
satisfies Corollary\/ {\rm\ref{la2cor}(b),} and\/ $\bar L$
satisfies Theorem\/ {\rm\ref{la2thm3}(b)}.
\end{itemize}
In particular, there exist no exact AC SL\/ $m$-folds asymptotic at
rate $\rho<0$ to $C=\Pi_0\cup\Pi_{\bs\phi}$ when
$\phi_1+\cdots+\phi_m=k\pi$ for $k=2,\ldots,m-2$.
\label{la4cor}
\end{cor}

Next we determine the sign of the invariant $A(L)$ from Definition
\ref{la2def6} for AC SL $m$-folds. Again, possibilities (a),(b)
agree with the signs of $A(L_{\bs\phi,A}),A(\ti L_{\bs\phi,A})$ for
the Lawlor necks $L_{\bs\phi,A},\ti L_{\bs\phi,A}$ in
Example~\ref{la3ex1}.

\begin{prop} Suppose $L'$ is a closed, exact AC SL\/ $m$-fold in
$\C^m$ for $m\ge 3$ asymptotic at rate $\rho<0$ to
$C=\Pi_0\cup\Pi_{\bs\phi}$. Then
\begin{itemize}
\setlength{\parsep}{0pt}
\setlength{\itemsep}{0pt}
\item[{\bf(a)}] If\/ $\phi_1+\cdots+\phi_m=\pi$ then $A(L')>0,$
and
\item[{\bf(b)}] If\/ $\phi_1+\cdots+\phi_m=(m-1)\pi$ then\/
$A(L')<0$.
\end{itemize}
\label{la4prop7}
\end{prop}

\begin{proof} Let $(M,\om),\la$ be as in \S\ref{la24}, so that
$\bar L'=L'\cup\{\iy_0,\iy_{\bs\phi}\}$ is a compact, exact, graded
Lagrangian in $M$. Apply Proposition \ref{la4prop5} to $L'$ with
$L=\Pi_0\cup\Pi_{\bs\phi}$ to get a Hamiltonian perturbation $\bar
L''=L''\cup\{\iy_0,\iy_{\bs\phi}\}$ of $\bar L'$ in $M$, which
intersects $\bar L=\cS_0\cup\cS_{\bs\phi}$ transversely, with
$0<\mu_{\bar L,\bar L''}(p)<m$ for $p\in \bar L\cap \bar
L''\cap\C^m$, $f_{\bar L''}(\iy_0)=0$ and $f_{\bar
L''}(\iy_{\bs\phi})=A(L')$. Note that $\bar L'$ and $\bar L''$ are
isomorphic in $D^b\sF(M)$, as they are Hamiltonian isotopic.

In case (a), $L'$ satisfies Corollary \ref{la2cor}(a), and $\bar L'$
satisfies Theorem \ref{la2thm3}(a) with $n=0$. So \eq{la2eq25} and
$\bar L'\cong\bar L''$ give a distinguished triangle in $D^b\sF(M)$
\e
\xymatrix@C=31pt{ \cS_{\bs\phi}[-1] \ar[r] & \bar L'' \ar[r] & \cS_0
\ar[r] & \cS_{\bs\phi}. }
\label{la4eq50}
\e
Since $\bar L'',\cS_0,\cS_{\bs\phi}$ are pairwise transverse, we may
apply Theorem \ref{la2thm2} with $\cS_{\bs\phi}[-1],\bar L'',\cS_0$
in place of $L,L',L''$. This shows there exists a holomorphic
triangle $\Si$ as in Figure \ref{la4fig2}, with corners
$p\in\cS_{\bs\phi}[-1]\cap\bar L''$ with
$\mu_{\cS_{\bs\phi}[-1],\bar L''}(p)=0$, $q\in \bar L''\cap \cS_0$
with $\mu_{\bar L'',\cS_0}(q)=0$, and $r\in\cS_0\cap
\cS_{\bs\phi}[-1]$ with $\mu_{\cS_0,\cS_{\bs\phi}[-1]}(r)=1$.
\begin{figure}[htb]
\centerline{$\splinetolerance{.8pt}
\begin{xy}
0;<1mm,0mm>:
,(-20,0);(20,0)**\crv{(0,10)}
?(.95)="a"
?(.85)="b"
?(.75)="c"
?(.65)="d"
?(.55)="e"
?(.45)="f"
?(.35)="g"
?(.25)="h"
?(.15)="i"
?(.05)="j"
?(.5)="y"
,(-20,0);(-30,-6)**\crv{(-30,-5)}
,(20,0);(30,-6)**\crv{(30,-5)}
,(20,0);(0,-12)**\crv{(10,-4)}
?(.1)="k"
?(.3)="l"
?(.5)="m"
?(.7)="n"
?(.9)="o"
?(.6)="x"
,(-20,0);(0,-12)**\crv{(-10,-4)}
?(.9)="p"
?(.7)="q"
?(.5)="r"
?(.3)="s"
?(.1)="t"
?(.6)="z"
,(0,-12);(4,-17)**\crv{(2.5,-14)}
,(0,-12);(-4,-17)**\crv{(-2.5,-14)}
,(20,0);(30,3)**\crv{(25,2)}
,(-20,0);(-30,3)**\crv{(-25,2)}
,"x";(6.5,-8.6)**\crv{(6.5,-8.6)}
,"x";(4.5,-7)**\crv{}
,"z";(-10,-3.5)**\crv{(-10,-3.5)}
,"z";(-11,-5.5)**\crv{}
,"a";"k"**@{.}
,"b";"l"**@{.}
,"c";"m"**@{.}
,"d";"n"**@{.}
,"e";"o"**@{.}
,"f";"p"**@{.}
,"g";"q"**@{.}
,"h";"r"**@{.}
,"i";"s"**@{.}
,"j";"t"**@{.}
,"y"*{<}
,(0,-12)*{\bu}
,(-20,0)*{\bu}
,(-20,-3)*{q}
,(20,0)*{\bu}
,(20,-3)*{r}
,(3,-12.3)*{p}
,(0,-3)*{\Si}
,(-32,3.5)*{\bar L''}
,(-32.5,-5.5)*{\cS_0}
,(37.4,2.5)*{\cS_{\bs\phi}[-1]}
,(33,-5.5)*{\cS_0}
,(7,-16)*{\bar L''}
,(-10.2,-16)*{\cS_{\bs\phi}[-1]}
\end{xy}$}
\caption{Holomorphic triangle $\Si$ with boundary in
$\cS_{\bs\phi}[-1]\cup\bar L''\cup\cS_0$}
\label{la4fig2}
\end{figure}

Note that $\cS_0$ has phase $\th_{\cS_0}=0$, and $\cS_{\bs\phi}$
phase $\th_{\cS_{\bs\phi}}=\phi_1+\cdots+\phi_m=\pi$, so that
$\cS_{\bs\phi}[-1]$ has phase $\th_{\cS_{\bs\phi}[-1]}=0$ by
\eq{la2eq18}. Thus $\cS_0\cup\cS_{\bs\phi}[-1]$ has the correct
phase 0 to be the compactification $\bar L$ of the special
Lagrangian cone $L=\Pi_0\cup\Pi_{\bs\phi}$. Therefore Proposition
\ref{la4prop5}(b) gives $0<\mu_{\cS_0,\bar L''}(s)<m$ for all $s\in
\cS_0\cap\bar L''\cap\C^m$, and $0<\mu_{\cS_{\bs\phi}[-1],\bar
L''}(s)<m$ for all $s\in \cS_{\bs\phi}[-1]\cap\bar L''\cap\C^m$.
Since $\mu_{\cS_{\bs\phi}[-1],\bar L''}(p)=0$, $\mu_{\cS_0,\bar
L''}(q)=m$, this implies that $p,q\notin\C^m$, so the only
possibilities are $p=\iy_{\bs\phi}$ and $q=\iy_0$. Also
$\cS_0\cap\cS_{\bs\phi}=\{0\}$, so~$r=0$.

The method of \eq{la2eq3} now gives
\ea
&\mathop{\rm area}(\Si)=
\int_{\!\!\xymatrix@C=26pt{\mathop{\bu}\limits_0 \ar[r]_{\text{in
$\cS_0$}} & \mathop{\bu}\limits_{\iy_0} }}
\!\!\!\!\!\!\!\!\!\!\!\!\!\!\!\!\!\!\!\!\!\!\! \d
f_{\cS_0}\,\,\,\,\,\,\,\,\,
+\int_{\!\!\xymatrix@C=21pt{\mathop{\bu}\limits_{\iy_0}
\ar[r]_{\text{in $\bar L''$}} & \mathop{\bu}\limits_{\iy_{\bs\phi}}
}} \!\!\!\!\!\!\!\!\!\!\!\!\!\!\!\!\!\!\!\!\!\!\!\!\!\! \d f_{\bar
L''}\,\,\,\,\,\,\,\,\,\,
+\int_{\!\!\xymatrix@C=35pt{\mathop{\bu}\limits_{\iy_{\bs\phi}}
\ar[r]_{\text{in $\cS_{\bs\phi}[-1]$}} & \mathop{\bu}\limits_0 }}
\!\!\!\!\!\!\!\!\!\!\!\!\!\!\!\!\!\!\!\!\!\!\!\!\!\!\!\!\!\!\! \d
f_{\cS_{\bs\phi}[-1]}
\nonumber\\
&\;=f_{\cS_0}(\iy_0)\!-\!f_{\cS_0}(0)\!+\!
f_{\bar L''}(\iy_{\bs\phi})\!-\!f_{\bar L''}(\iy_0)\!+\!
f_{\cS_{\bs\phi}[-1]}(0)\!-\!f_{\cS_{\bs\phi}[-1]}(\iy_{\bs\phi})
\nonumber\\
&\;=0-0+A(L')-0+0-0=A(L'),
\label{la4eq51}
\ea
using $f_{\cS_0}=f_{\cS_{\bs\phi}[-1]}=0$ and Proposition
\ref{la4prop5}(c). As $\mathop{\rm area}(\Si)>0$ we have $A(L')>0$,
proving part (a).

For part (b), $L'$ satisfies Corollary \ref{la2cor}(a), and $\bar
L'$ Theorem \ref{la2thm3}(a) with $n=0$, so as for \eq{la4eq50},
equation \eq{la2eq26} gives a distinguished triangle in $D^b\sF(M)$
\begin{equation*}
\xymatrix@C=31pt{  \cS_0 \ar[r] & \bar L'' \ar[r] &
\cS_{\bs\phi}[1-m] \ar[r] & \cS_{\bs\phi}. }
\end{equation*}
The argument above with $\cS_{\bs\phi}[1-m],\cS_0$ in place of
$\cS_0,\cS_{\bs\phi}[-1]$ yields $\mathop{\rm area}(\Si)=-A(L')$, so
$A(L')<0$. This completes the proof.
\end{proof}

Here is an analogue of Corollary \ref{la4cor} and Proposition
\ref{la4prop7} for AC LMCF expanders, determining both the possible
cones $C=\Pi_0\cup\Pi_{\bs\phi}$, and the sign of $A(L)$. The
possibilities in (a),(b) agree exactly with those for the
Joyce--Lee--Tsui expanders $L_{\bs\phi}^\al,\ti L_{\bs\phi}^\al$ in
Example~\ref{la3ex2}.

\begin{prop} Suppose $L'$ is a closed, exact AC Lagrangian MCF
expander in $\C^m$ for $m\ge 3$ with rate $\rho<2$ and cone
$C=\Pi_0\cup\Pi_{\bs\phi}$ defined using
$\phi_1,\ldots,\phi_m\in(0,\pi),$ and satisfying the expander
equation $H=\al F^\perp$ for $\al>0$. Then either:
\begin{itemize}
\setlength{\parsep}{0pt}
\setlength{\itemsep}{0pt}
\item[{\bf(a)}] There is a unique grading $\th_{L'}$ on $L'$ such
that\/ $\th_{L'}\ra 0$ as $r\ra\iy$ in the end of\/ $L'$
asymptotic to $\Pi_0,$ and\/ $\th_{L'}\ra
\phi_1\!+\!\cdots\!+\!\phi_m\!-\!\pi$ as $r\ra\iy$ in the end
asymptotic to $\Pi_{\bs\phi}$. As for \eq{la3eq47}, Definition\/
{\rm\ref{la2def6}} and\/ \eq{la3eq16} give
\e
A(L')=2\bigl(\pi-\phi_1-\cdots-\phi_m\bigr)/\al.
\label{la4eq52}
\e
In this case $A(L')>0,$ so $0<\phi_1+\cdots+\phi_m<\pi$. Or:
\item[{\bf(b)}] There is a unique grading $\th_{L'}$ on $L'$ such
that\/ $\th_{L'}\ra 0$ as $r\ra\iy$ in the end of\/ $L'$
asymptotic to $\Pi_0,$ and\/ $\th_{L'}\ra
\phi_1+\cdots+\phi_m-(m-1)\pi$ as $r\ra\iy$ in the end
asymptotic to $\Pi_{\bs\phi},$ so
\e
A(L')=2\bigl((m-1)\pi-\phi_1-\cdots-\phi_m\bigr)/\al.
\label{la4eq53}
\e
In this case $A(L')<0,$ so $(m-1)\pi<\phi_1+\cdots+\phi_m<m\pi$.
\end{itemize}
Thus, there exist no closed, exact, AC LMCF expanders asymptotic to
$C=\Pi_0\cup\Pi_{\bs\phi}$ when $\pi\le\phi_1+\cdots+\phi_m\le
(m-1)\pi$.
\label{la4prop8}
\end{prop}

\begin{proof} The first parts of (a),(b) follow from Corollary \ref{la2cor} (where we use Theorem \ref{la3thm6}(b) to improve the asymptotic rate $\rho<2$ to rate $\rho<0$ so that Corollary \ref{la2cor} applies), and \eq{la4eq52}--\eq{la4eq53} follow from these,
Definition \ref{la2def6}, and \eq{la3eq16}. It remains only to prove
that $A(L')>0$ in case (a), and $A(L')<0$ in case~(b).

For part (a), follow the first half of the proof of Proposition
\ref{la4prop7}, but using Proposition \ref{la4prop6} rather than
Proposition \ref{la4prop5}. This yields a holomorphic triangle as in
Figure \ref{la4fig2}, with corners $p\in\cS_{\bs\phi}[-1]\cap\bar
L''$ with $\mu_{\cS_{\bs\phi}[-1],\bar L''}(p)=0$, $q\in \bar
L''\cap \cS_0$ with $\mu_{\bar L'',\cS_0}(q)=0$, and $r\in\cS_0\cap
\cS_{\bs\phi}[-1]$ with $\mu_{\cS_0,\cS_{\bs\phi}[-1]}(r)=1$. As
$\cS_0\cap\cS_{\bs\phi}=\{0\}$ we have $r=0$. 

Since $f_{\cS_0}=0$ and $f_{\cS_{\bs\phi}[-1]}=A(L')$, the analogue of \eq{la4eq51} gives
\e
\begin{split}
\mathop{\rm area}(\Si)&=f_{\cS_0}(q)\!-\!f_{\cS_0}(0)\!+\!
f_{\bar L''}(p)\!-\!f_{\bar L''}(q)\!+\!
f_{\cS_{\bs\phi}[-1]}(0)\!-\!f_{\cS_{\bs\phi}[-1]}(p)\\
&=A(L') - (f_{\cS_{\bs\phi}[-1]}(p)-f_{\bar L''}(p))-(f_{\bar L''}(q)-f_{\cS_0}(q))>0.
\end{split}
\label{la4eq54}
\e
We claim also that
\e
f_{\cS_{\bs\phi}[-1]}(p)-f_{\bar L''}(p)\ge 0\quad\text{and}\quad
f_{\bar L''}(q)-f_{\cS_0}(q)\ge 0.
\label{la4eq55}
\e
The first equation of \eq{la4eq55} follows from $f_{\cS_{\bs\phi}[-1]}=A(L')$ and the second line of \eq{la4eq38} if $p=\iy_{\bs\phi}$, and from $\mu_{\cS_{\bs\phi}[-1],\bar L''}(p)=0$ and the first inequality of \eq{la4eq37} if $p\ne\iy_{\bs\phi}$. Similarly, the second equation of \eq{la4eq55} follows from $f_{\cS_0}=0$ and the first line of \eq{la4eq38} if $q=\iy_0$, and from $\mu_{\cS_0,\bar L''}(q)=m$ and the second inequality of \eq{la4eq37} if $q\ne\iy_0$. Adding up \eq{la4eq54}--\eq{la4eq55} shows that $A(L')>0$, proving (a). Part (b) is similar, with $\cS_{\bs\phi}[1-m],\cS_0$ in place of $\cS_0,\cS_{\bs\phi}[-1]$ in the above, as in the proof of Proposition~\ref{la4prop7}(b).
\end{proof}

\subsection{Proof of Theorem \ref{la4thm1}}
\label{la46}

First suppose $m\ge 3$ and $\phi_1,\ldots,\phi_m\in(0,\pi)$ with
$\phi_1+\cdots+\phi_m=\pi$, and $L$ is a closed, embedded, exact, AC
SL $m$-fold $\C^m$ asymptotic with rate $\rho<0$ to
$C=\Pi_0\cup\Pi_{\bs\phi}$. Then Proposition \ref{la4prop7}(a) shows
that $A(L)>0$. Thus, by Example \ref{la3ex1} there is a unique
`Lawlor neck' $L_{\bs\phi,A}$ with cone $C=\Pi_0\cup\Pi_{\bs\phi}$
and $A(L_{\bs\phi,A})=A:=A(L)$. Suppose for a contradiction
that~$L\ne L_{\bs\phi,A}$.

Apply Proposition \ref{la4prop5} to $L$ and $L'=L_{\bs\phi,A}$. This
defines a Hamiltonian perturbation $\bar
L''=L''\cup\{\iy_0,\iy_{\bs\phi}\}$ of $\bar L_{\bs\phi,A}$ in $M$
which intersects $\bar L=L\cup\{\iy_0,\iy_{\bs\phi}\}$ transversely.
As both $L,L_{\bs\phi,A}$ satisfy Corollary \ref{la2cor}(a), their
compactifications $\bar L,\bar L_{\bs\phi,A}$ satisfy Theorem
\ref{la2thm3}(a) with $n=0$, and this implies that $\bar L\cong \bar
L_{\bs\phi,A}$ in $D^b\sF(M)$. Hence $\bar L\cong \bar L''$ in
$D^b\sF(M)$ as $\bar L''$ and $\bar L_{\bs\phi,A}$ are Hamiltonian
isotopic. Therefore Theorem \ref{la2thm1}(b) applies to $\bar L,\bar
L''$, and shows there exists a $J$-holomorphic disc $\Si$ in $M$ of
the form shown in Figure \ref{la4fig3}, with corners $p,q\in\bar
L\cap\bar L''$ with $\mu_{\bar L,\bar L''}(p)=0$ and $\mu_{\bar
L,\bar L''}(q)=m$.
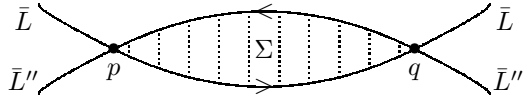
\begin{figure}[htb]
\centerline{$\splinetolerance{.8pt}
\begin{xy}
0;<1mm,0mm>:
,(-20,0);(20,0)**\crv{(0,10)}
?(.95)="a"
?(.85)="b"
?(.75)="c"
?(.65)="d"
?(.55)="e"
?(.45)="f"
?(.35)="g"
?(.25)="h"
?(.15)="i"
?(.05)="j"
?(.5)="y"
,(-20,0);(-30,-6)**\crv{(-30,-5)}
,(20,0);(30,-6)**\crv{(30,-5)}
,(-20,0);(20,0)**\crv{(0,-10)}
?(.95)="k"
?(.85)="l"
?(.75)="m"
?(.65)="n"
?(.55)="o"
?(.45)="p"
?(.35)="q"
?(.25)="r"
?(.15)="s"
?(.05)="t"
?(.5)="z"
,(-20,0);(-30,6)**\crv{(-30,5)}
,(20,0);(30,6)**\crv{(30,5)}
,"a";"k"**@{.}
,"b";"l"**@{.}
,"c";"m"**@{.}
,"d";"n"**@{.}
,"e";"o"**@{.}
,"f";"p"**@{.}
,"g";"q"**@{.}
,"h";"r"**@{.}
,"i";"s"**@{.}
,"j";"t"**@{.}
,"y"*{<}
,"z"*{>}
,(-20,0)*{\bu}
,(-20,-3)*{p}
,(20,0)*{\bu}
,(20,-3)*{q}
,(0,0)*{\Si}
,(-32,4)*{\bar L}
,(-32,-4)*{\bar L''}
,(32,4)*{\bar L}
,(32.5,-4)*{\bar L''}
\end{xy}$}
\caption{Holomorphic disc $\Si$ with boundary in $\bar L\cup\bar L''$}
\label{la4fig3}
\end{figure}

Proposition \ref{la4prop5}(b) shows that $0<\mu_{\bar L,\bar
L''}(r)<m$ for any $r\in\bar L\cap\bar L''\cap\C^m$, so $p,q\in
(\bar L\cap\bar L'')\sm\C^m=\{\iy_0,\iy_{\bs\phi}\}$. By Proposition
\ref{la4prop5}(c) we have
\e
\begin{split}
f_{\bar L''}(\iy_0)&=f_{\bar L_{\bs\phi,A}}(\iy_0)=0=f_{\bar L}(\iy_0),\\
f_{\bar L''}(\iy_{\bs\phi})&=f_{\bar L_{\bs\phi,A}}(\iy_{\bs\phi})=
A(L_{\bs\phi,A})=A=A(L)=f_{\bar L}(\iy_{\bs\phi}).
\end{split}
\label{la4eq56}
\e
So as $p,q\in\{\iy_0,\iy_{\bs\phi}\}$ we see that $f_{\bar L''}(p)=
f_{\bar L}(p)$ and $f_{\bar L''}(q)= f_{\bar L}(q)$. Thus
\eq{la2eq3} gives $\mathop{\rm area}(\Si)=f_{\bar L}(q)-f_{\bar
L}(p)+f_{\bar L''}(p)-f_{\bar L''}(q)=0,$ a contradiction. Hence
$L=L_{\bs\phi,A}$ for $A=A(L)$, proving the first part of
Theorem~\ref{la4thm1}.

For the second part, suppose $L$ has rate $\rho<2$ rather than
$\rho<0$. Then Proposition \ref{la4prop4} gives a unique $\bs
c\in\C^m$ with $L'=L-\bs c$ asymptotic to $C$ at rate $\rho'<0$. The
first part now shows that $L'=L_{\bs\phi,A}$, so
$L=L_{\bs\phi,A}+\bs c$, as we want.

The third part, with $\phi_1+\cdots+\phi_m=(m-1)\pi$, is proved as
for the first and second parts, but using part (b) of Proposition
\ref{la4prop7}, Corollary \ref{la2cor} and Theorem \ref{la2thm3}
rather than part (a), and $\ti L_{\bs\phi,A}$ for $A<0$ from Example
\ref{la3ex1} rather than $L_{\bs\phi,A}$ for $A>0$. The fourth part,
that there exist no exact AC SL $m$-folds $L$ with cone
$\Pi_0\cup\Pi_{\bs\phi}$ if $\phi_1+\cdots+\phi_m=k\pi$ for
$1<k<m-1$, is the last part of Corollary \ref{la4cor}. The final
part, concerning immersed AC SL $m$-folds, follows from
Proposition~\ref{la4prop3}.

\begin{rem} The proof above involves showing that for each exact AC
SL $m$-fold $L$ with cone $C=\Pi_0\cup\Pi_{\bs\phi}$, there is a
unique $L_{\bs\phi,A}$ with $A(L)=A(L_{\bs\phi,A})$. As
$A(L)=A(L_{\bs\phi,A})$ we have $f_{\bar L}(\iy_0)=f_{\bar
L_{\bs\phi,A}}(\iy_0)$ and $f_{\bar L}(\iy_{\bs\phi})=f_{\bar
L_{\bs\phi,A}}(\iy_{\bs\phi})$. These are used to compute the
(positive) area of the curve $\Si$ given by Theorem \ref{la2thm1}(b)
if it has corners at $\iy_0$ or $\iy_{\bs\phi}$, and prove a
contradiction if~$L\ne L_{\bs\phi,A}$.

We now sketch an alternative method. Roughly speaking, the invariant
$A(L)\in\R$ of AC Lagrangians $L$ asymptotic to
$\Pi_0\cup\Pi_{\bs\phi}$ from \S\ref{la23} depends on a compactified
version of the relative cohomology class $[\om]$ in
$H^2(\C^m,L;\R)$. Following Joyce \cite[Def.~7.2]{Joyc2}, one can
define an invariant $Z(L)\in\R$ of connected AC special Lagrangians
$L$ asymptotic to $\Pi_0\cup\Pi_{\bs\phi}$, depending on the
relative cohomology class $[\Im\Om]$ in $H^m(\C^m,L;\R)$. This
$Z(L)$ controls the asymptotic behaviour of $L$: writing the end of
$L$ asymptotic to $\Pi_0$ as a graph $\Ga_{\d f}$ as in \S\ref{la32}
for $f:\R^m\ra\R$ defined near infinity in $\R^m$,
then~$f=Z(L)r^{2-m}+O(r^{1-m})$.

Rather than comparing $L$ to $L_{\bs\phi,A}$ with
$A(L)=A(L_{\bs\phi,A})$, we compare $L$ to $L_{\bs\phi,A}$ with
$Z(L)=Z(L_{\bs\phi,A})$. In the analogue of Proposition
\ref{la4prop5} with $L'=L_{\bs\phi,A}$, we do not prescribe $f_{\bar
L''}(\iy_0),f_{\bar L''}(\iy_{\bs\phi})$, but instead use
$Z(L)=Z(L')$ to show we can choose $\bar L''$ with $\mu_{\bar L,\bar
L''}(\iy_0),\mu_{\bar L,\bar L''}(\iy_{\bs\phi})\ne 0,m$, roughly
because $f_{\bar L'}-f_{\bar L}$ cannot have a maximum or minimum at
$\iy_0,\iy_{\bs\phi}$. Then we prove a contradiction if $L\ne
L_{\bs\phi,A}$ using Theorem \ref{la2thm1}(a) rather than
Theorem~\ref{la2thm1}(b).

We did not use this method for two reasons. Firstly, $Z(L)$ does not
make sense for LMCF expanders, so this approach would not work for
Theorem \ref{la4thm2}. Secondly, we would need an analogue of
Proposition \ref{la4prop7}, that $Z(L)>0$ if
$\phi_1+\cdots+\phi_m=\pi$ and $Z(L)<0$ if $\phi_1+\cdots+
\phi_m=(m-1)\pi$. The authors know of no direct proof, though one
can deduce it from Proposition \ref{la4prop7} by showing that $A(L)$
and $Z(L)$ have the same sign, since $A(L)Z(L)$ is proportional
to~$\Vert F^\perp\Vert_{L^2}^2$.

\label{la4rem}
\end{rem}

\subsection{Proof of Theorem \ref{la4thm2}}
\label{la47}

First suppose $m\ge 3$ and $\phi_1,\ldots,\phi_m\in(0,\pi)$ with
$0<\phi_1+\cdots+\phi_m<\pi$, and $L$ is a closed, embedded, exact,
Asymptotically Conical Lagrangian MCF expander $\C^m$ asymptotic
with rate $\rho<2$ to $C=\Pi_0\cup\Pi_{\bs\phi}$, satisfying $H=\al
F^\perp$ for $\al>0$. Then Example \ref{la3ex2} defines a
Joyce--Lee--Tsui expander $L_{\bs\phi}^\al$. Suppose for a
contradiction that~$L\ne L_{\bs\phi}^\al$.

Theorem \ref{la3thm5} implies that $L$ is asymptotic to $C$ with any
rate $\rho'<2$, so we can take $\rho'<0$, and thus the
compactification $\bar L$ in $(M,\om)$ exists as in \S\ref{la24}.
Apply Proposition \ref{la4prop6} to $L$ and $L'=L_{\bs\phi}^\al$.
This defines a Hamiltonian perturbation $\bar
L''=L''\cup\{\iy_0,\iy_{\bs\phi}\}$ of $\bar L_{\bs\phi}^\al$ in $M$
which intersects $\bar L=L\cup\{\iy_0,\iy_{\bs\phi}\}$ transversely.
As both $L,L_{\bs\phi}^\al$ satisfy Proposition \ref{la4prop8}(a)
and hence Corollary \ref{la2cor}(a), their compactifications $\bar
L,\bar L_{\bs\phi}^\al$ satisfy Theorem \ref{la2thm3}(a) with $n=0$,
and thus $\bar L\cong \bar L_{\bs\phi,A}$ in $D^b\sF(M)$. Hence
$\bar L\cong \bar L''$ in $D^b\sF(M)$ as $\bar L''$ and $\bar
L_{\bs\phi}^\al$ are Hamiltonian isotopic. Therefore Theorem
\ref{la2thm1}(b) applies to $\bar L,\bar L''$, and shows there
exists a $J$-holomorphic disc $\Si$ in $M$ of the form shown in
Figure \ref{la4fig3}, with corners $p,q\in\bar L\cap\bar L''$ with
$\mu_{\bar L,\bar L''}(p)=0$ and~$\mu_{\bar L,\bar L''}(q)=m$.

We claim that $f_{\bar L}(p)-f_{\bar L''}(p)\ge 0$ and $f_{\bar
L}(q)-f_{\bar L''}(q)\le 0$. To prove these, note that if $p$ or $q$
lie in $\{\iy_0,\iy_{\bs\phi}\}$ then as in \eq{la4eq56} we have
$f_{\bar L}(p)=f_{\bar L''}(p)$ or $f_{\bar L}(q)=f_{\bar L''}(q)$,
since as for \eq{la4eq56}, using Proposition \ref{la4prop6}(c) we
have
\begin{align*}
f_{\bar L''}(\iy_0)&=f_{\bar L_{\bs\phi,A}}(\iy_0)=0=f_{\bar L}(\iy_0),\\
f_{\bar L''}(\iy_{\bs\phi})&=f_{\bar L_{\bs\phi}^\al}(\iy_{\bs\phi})=
A(L_{\bs\phi,A})=A(L)=f_{\bar L}(\iy_{\bs\phi}).
\end{align*}
If $p\in\bar L\cap\bar L''\cap\C^m$ then the first inequality of \eq{la4eq37} at $p$ gives $f_{\bar L}(p)-f_{\bar L''}(p)>0$, as $\mu_{\bar L,\bar L''}(p)=0$. If $q\in\bar L\cap\bar L''\cap\C^m$ then the second inequality of \eq{la4eq37} at $q$ gives $f_{\bar L}(q)-f_{\bar L''}(q)<0$, as $\mu_{\bar L,\bar L''}(q)=m$, proving the claim. But this contradicts
\begin{equation*}
0< \mathop{\rm area}(\Si)=f_{\bar L}(q)-f_{\bar L}(p)+f_{\bar
L''}(p)-f_{\bar L''}(q),
\end{equation*}
by \eq{la2eq3}. Hence $L=L_{\bs\phi}^\al$, proving the first part of
Theorem~\ref{la4thm2}.

The second part, with $(m-1)\pi<\phi_1+\cdots+\phi_m<m\pi$, is
proved as for the first part, but using part (b) of Proposition
\ref{la4prop8}, Corollary \ref{la2cor} and Theorem \ref{la2thm3}
rather than part (a), and $\ti L_{\bs\phi}^\al$ from Example
\ref{la3ex2} rather than $L_{\bs\phi}^\al$. The third part, that
there exist no exact AC LMCF expanders $L$ with cone
$\Pi_0\cup\Pi_{\bs\phi}$ if $\pi\le\phi_1+\cdots+\phi_m\le
(m-1)\pi$, is the last part of Proposition \ref{la4prop8}. The final
part, concerning immersed AC LMCF expanders, follows from
Proposition~\ref{la4prop3}.

\appendix

\section{Additional material}
\label{laA}

Finally we prove two new results that were used in \S\ref{la43}--\S\ref{la44} above, on real analytic functions and Morse theory.

\subsection{A real analytic ``Taylor's Theorem'' type result}
\label{laA1}

In Theorem \ref{laAthm1} we prove a result similar to Taylor's theorem for real analytic functions, which was used in the proofs of Propositions \ref{la4prop5} and \ref{la4prop6}. So far as the authors know it is new, and it may be of independent interest. 

The usual versions of Taylor's Theorem are of the form
\begin{equation*}
f(x_1,\ldots,x_m)=\sum_{j=0}^k\sum_{i_1,\ldots,i_j=1}^m \frac{1}{j!}\cdot x_{i_1}\cdots x_{i_j}\frac{\pd^jf}{\pd x_{i_1}\cdots\pd x_{i_j}}(0)+(\text{error term}),
\end{equation*}
where $0\in\R^m$ is thought of as fixed and $(x_1,\ldots,x_m)$ as varying, and the `error term' may be estimated in a variety of ways. In our theorem we exchange the r\^oles of fixed point $0$ and varying point $(x_1,\ldots,x_m)$, and we estimate the error term in terms of a multiple of the next term in the Taylor series. Example \ref{laAex} below shows the real analytic condition on $f$ in Theorem \ref{laAthm1} is essential.

\begin{thm} Let\/ $U$ be an open neighbourhood of\/ $0$ in $\R^m,$ $f:U\ra\R$ a real analytic function, and\/ $k=0,1,2,\ldots.$ Then there exist an open neighbourhood\/ $V$ of\/ $0$ in $U$ and a constant\/ $C\ge 0$ such that for all\/ $(x_1,\ldots,x_m)\in V$ we have
\ea
&\biggl\vert f(0)-\sum_{j=0}^k\sum_{i_1,\ldots,i_j=1\!\!\!\!\!}^m \frac{(-1)^j}{j!}\cdot x_{i_1}\cdots x_{i_j}\frac{\pd^jf}{\pd x_{i_1}\cdots\pd x_{i_j}}(x_1,\ldots,x_m)\biggr\vert 
\label{laAeq1}\\
&\le C (x_1^2\!+\!\cdots\!+\!x_m^2)^{(k+1)/2}
\biggl(\,\,\sum_{i_1,\ldots,i_{k+1}=1}^m
\frac{\pd^{k+1}f}{\pd x_{i_1}\cdots\pd x_{i_{k+1}}}(x_1,\ldots,x_m)^2\biggr)^{1/2}\!\!.
\nonumber
\ea
The same statement holds if\/ $f(x_1,\ldots,x_m)=(x_1^2+\cdots+x_m^2)^{1/2}g(x_1,\ldots,x_m)$ for $g:U\ra\R$ real analytic, taking $(x_1,\ldots,x_m)\ne 0$ in\/~\eq{laAeq1}.
\label{laAthm1}
\end{thm}

\begin{proof} We first prove the case $k=0$ of the theorem, supposing $f(0)=0$. We must show that we can find an open $0\in V\subseteq U$ and $C\ge 0$ such that
\e
\bmd{f(\bs x)}\le C(x_1^2\!+\!\cdots\!+\!x_m^2)^{1/2}\bigl(\ts\sum_{i=1}^m\frac{\pd f}{\pd x_i}(\bs x)^2\bigr)^{1/2} 
\label{laAeq2}
\e
for all $\bs x=(x_1,\ldots,x_m)$ in $V$.

By replacing $U$ by an open neighbourhood $U'$ of 0 in $U$, we can suppose that $f$ has no nonzero critical values, that is, $\d f(x_1,\ldots,x_m)=0$ implies $f(x_1,\ldots,x_m)=0$. Now Hironaka's Theorem holds in the real analytic category, \cite[\S 7]{Hiro}. Thus, we may choose an {\it embedded resolution of singularities\/ $\pi:\ti U\ra U$ of\/}~$f:U\ra\R$.

This means that $\ti U$ is a real analytic manifold, $\pi:\ti U\ra U$ is a proper, surjective, real analytic map, $D_0,\ldots,D_n$ are closed real analytic hypersurfaces in $\ti U$ meeting with normal crossings with $D_0$ the proper transform of $f^{-1}(0)\subset U$ and $\pi^{-1}(\Crit(f))=D_1\cup\cdots\cup D_n$, and there are integers $d_0=1$ and $d_1,\ldots,d_n\ge 2$ such that for any $p\in\ti U$ there exists an open neighbourhood $\ti V_p$ of $p$ in $\ti U$ with
\e
f\ci\pi\vert_{\ti V_p}=F\cdot\ts\prod_{i\in I}y_i^{d_i},
\label{laAeq3}
\e
where $I=\{i=0,\ldots,m:p\in D_i\}$, and $y_i:\ti V_p\ra\R$ is real analytic and vanishes with multiplicity 1 on $D_i\cap \ti V_p$ for $i\in I$, and $F:\ti V_p\ra\R\sm\{0\}$ is real analytic.

Choose the resolution $\pi:\ti U\ra U$ to factor through the real blow-up $\hat\pi:\hat U\ra U$ of $U$ at 0. Then there are unique integers $e_0,\ldots,e_n\ge 0$ such that $\pi^*(I_0)=\prod_{i=0}^nI_{D_i}^{e_i}$, where $I_0$ is the ideal of smooth functions vanishing at 0 in $U$, and $I_{D_i}$ the ideal of smooth functions vanishing along $D_i$ in $\ti U$, and $e_i>0$ if and only if $D_i\subseteq\pi^{-1}(\{0\})$. By treating $(x_1^2+\cdots+x_m^2)\ci\pi$ as for $f\ci\pi$ above and taking square roots, we find that for any $p\in\ti U$ there exists an open neighbourhood $\ti V_p$ of $p$ in $\ti U$ with
\e
(x_1^2+\cdots+x_m^2)^{1/2}\ci\pi\vert_{\ti V_p}=G\cdot\ts\prod_{i\in I}\md{y_i}^{e_i},
\label{laAeq4}
\e
where $I,y_i$ are as above and $G:\ti V_p\ra(0,\iy)$ is real analytic. As $f(0)=0$, making $U$ smaller if necessary we see that $\md{f(x_1,\ldots,x_m)}\le E(x_1^2+\cdots+x_m^2)^{1/2}$ for some $E>0$ and all $(x_1,\ldots,x_m)\in U$. So comparing \eq{laAeq3} and \eq{laAeq4} we see that $0\le e_i\le d_i$ for~$i=0,\ldots,n$.

Suppose $p\in\pi^{-1}(0)$, and let $\ti V_p,I,y_i,F,G$ be as above. Then $p\in D_k$ for at least one $k=1,\ldots,n$ with $e_k>0$, so $k\in I$. Making $\ti V_p$ smaller if necessary, extend $(y_i:i\in I)$ to a real analytic coordinate system $(y_i:i\in I$, $z_1,\ldots,z_l)$ on $\ti V_p$, where $l=m-\md{I}$. Then we have
\begin{equation*}
\pi^*(\d x_1^2+\cdots+\d x_m^2)=\sum_{i,j\in I}A_{ij}\d y_i\d y_j+\sum_{i\in I,\;1\le a\le l}B_{ia}\d y_i\d z_a
+\sum_{1\le a,b\le l}C_{ab}\d z_a\d z_b,
\end{equation*}
for real analytic $A_{ij},B_{ia},C_{ab}:\ti V_p\ra\R$. For $c=1,\ldots,m$, write
\begin{equation*}
x_c\ci\pi=H_c\cdot\ts\prod_{i\in I}y_i^{n_i^c}\quad\text{on $\ti V_p$,}
\end{equation*}
where $H_c:\ti V_p\ra\R$ is real analytic and not divisible by any $y_i$, and $n_i^c\ge 0$.  Then calculation shows that
\ea
A_{ij}&=\sum_{c=1}^m\prod_{i'\in I}y_{i'}^{2n_{i'}^c}\cdot\ts\bigl(\frac{\pd}{\pd y_i}H_c+n_i^c y_i^{-1}H_c\bigr)\bigl(\frac{\pd}{\pd y_j}H_c+n_j^c y_j^{-1}H_c\bigr),
\label{laAeq5}\\
B_{ia}&=\sum_{c=1}^m\prod_{i'\in I}y_{i'}^{2n_{i'}^c}\cdot\ts\bigl(\frac{\pd}{\pd y_i}H_c+n_i^c y_i^{-1}H_c\bigr)\frac{\pd}{\pd z_a}H_c,
\label{laAeq6}\\
C_{ab}&=\sum_{c=1}^m\prod_{i'\in I}y_{i'}^{2n_{i'}^c}\cdot\ts\frac{\pd}{\pd z_a}H_c\,\frac{\pd}{\pd z_b}H_c.
\label{laAeq7}
\ea
Since $x_c^2\le x_1^2+\cdots+x_m^2$, equation \eq{laAeq4} implies that $n_i^c\ge e_i$ for all $c,i$. Therefore making $\ti V_p$ smaller if necessary, we see from \eq{laAeq5}--\eq{laAeq7} that there exists $D>0$ such that for all $\al_i:i\in I$, $\be_1,\ldots,\be_l$ in $\R$ and $q\in \ti V_p$ we have
\e
\begin{split}
\ts\sum_{i,j\in I}A_{ij}(q)\al_i\al_j+\sum_{i\in I,1\le a\le l}B_{ia}(q)\al_i\be_a
+\sum_{1\le a,b\le l}C_{ab}(q)\be_a\be_b&\\
\le D\bigl(\ts\sum_{i\in I}y_i^{\max(2e_i-2,0)}\cdot\prod_{i\ne j\in I}y_j^{2e_j}\cdot\al_i^2+\sum_{a=1}^l\prod_{j\in I}y_j^{2e_j}\be_a^2\bigr)&.
\end{split}
\label{laAeq8}
\e

Since $F\ne 0$ on $\ti V_p$ and $y_k(p)=0$, making $\ti V_p$ smaller, we may suppose that 
\e
\md{y_k}\cdot\bmd{\ts\frac{\pd F}{\pd y_k}(q)}\le \ha d_k \bmd{F(q)}
\qquad\text{for all $q\in \ti V_p$.}
\label{laAeq9}
\e
As $G>0$ on $\ti V_p$, making $\ti V_p$ smaller, we can choose $C_p>0$ such that
\e
\ts\frac{2}{d_k}D^{1/2}\le C_pG(q)\qquad\text{for all $q\in \ti V_p$.}
\label{laAeq10}
\e
Let $q\in\ti V_p$ with $\pi(q)=\bs x=(x_1,\ldots,x_m)$ in $U$. Then we have
\ea
&\ts\bmd{f(\bs x)}=\md{f\ci\pi(q)}=\md{F(q)}\prod_{i\in I}\md{y_i}^{d_i}\le \frac{2}{d_k}\bigl(d_k\md{F(q)}-\md{y_k}\md{\frac{\pd F}{\pd y_k}(q)}\bigr)\prod_{i\in I}\md{y_i}^{d_i}
\nonumber\\
&\le\ts \frac{2}{d_k}\md{y_k}\cdot\bmd{\bigl(d_kF(q)+y_k\frac{\pd F}{\pd y_k}(q)\bigr)y_k^{d_k-1}\prod_{i\in I\sm\{k\}}y_i^{d_i}} 
=\frac{2}{d_k}\md{y_k}\bmd{\ts\frac{\pd}{\pd y_k}(f\ci\pi)(q)}
\nonumber\\
&\le\ts\frac{2}{d_k}D^{1/2}\prod_{i\in I}\md{y_i}^{e_i}\cdot D^{-1/2}\bigl[\ts\sum_{i\in I}y_i^{-\max(2e_i-2,0)}\prod_{i\ne j\in I}y_j^{-2e_j}
\nonumber\\
&\qquad\ts\cdot\ms{\ts\frac{\pd}{\pd y_i}(f\ci\pi)(q)}+\sum_{a=1}^l\prod_{j\in I}y_j^{-2e_j}\ms{\ts\frac{\pd}{\pd z_a}(f\ci\pi)(q)}\bigr]^{1/2} 
\nonumber\\
&\le C_pG(q)\ts\prod_{i\in I}\md{y_i}^{e_i}\cdot\bigl[\pi^*(\d x_1^2+\cdots+\d x_m^2)^{-1}\cdot\bigl(\d (f\ci\pi)(q)\ot \d (f\ci \pi)(q)\bigr)\bigr]^{1/2}
\nonumber\\
&=C_p(x_1^2+\cdots+x_m^2)^{1/2}\bigl[\ts\sum_{i=1}^m\frac{\pd f}{\pd x_i}(\bs x)^2\bigr]^{1/2},
\label{laAeq11}
\ea
using $\pi(q)=\bs x$ in the first step, \eq{laAeq3} in the second and fifth, \eq{laAeq9} in the third, $e_k>0$ in the sixth, \eq{laAeq10} and the square root of the analogue of \eq{laAeq8} for the inverse metrics $\sum_i\frac{\pd^2}{\pd x_i^2},\sum_i\bigl(\frac{\pd}{\pd y_i})^2+\sum_a\bigl(\frac{\pd}{\pd z_a}\bigr)^2$ rather than $\sum_i\d x_i^2,\sum_i\d y_i^2+\sum_a\d z_a^2$ in the seventh, and \eq{laAeq4} in the eighth.

We have shown that each $p\in\pi^{-1}(0)$ has an open neighbourhood $\ti V_p$ in $\ti U$ and a constant $C_p>0$ such that \eq{laAeq11} holds for all $\bs x\in\pi(\ti V_p)$. Since $\pi$ is proper, $\pi^{-1}(0)$ is compact, so we can choose $p_1,\ldots,p_N\in\pi^{-1}(0)$ and $\ti V_{p_1},\ldots,\ti V_{p_N}\subset\ti U$, $C_{p_1},\ldots,C_{p_N}>0$ as above such that $\pi^{-1}(0)\subset\ti V_{p_1}\cup\cdots\cup\ti V_{p_N}$. Define $C=\max(C_{p_1},\ldots,C_{p_N})$. Since $\pi$ is proper and $\ti V_{p_1}\cup\cdots\cup\ti V_{p_N}$ is an open neighbourhood of $\pi^{-1}(0)$ in $\ti U$, there exists an open $0\in V\subseteq U$ with $\pi^{-1}(V)\subseteq\ti V_{p_1}\cup\cdots\cup\ti V_{p_N}$. Let $(x_1,\ldots,x_m)\in V$. Then $(x_1,\ldots,x_m)=\pi(q)$ for $q\in\ti V_{p_i}$, for some $i=1,\ldots,N$, as $\pi$ is surjective, so \eq{laAeq11} gives
\begin{align*}
\bmd{f(\bs x)}&\le C_{p_i}(x_1^2+\cdots+x_m^2)^{1/2}\bigl(\ts\sum_{i=1}^m\frac{\pd f}{\pd x_i}(\bs x)^2\bigr)^{1/2}\\
&\le C(x_1^2+\cdots+x_m^2)^{1/2}\bigl(\ts\sum_{i=1}^m\frac{\pd f}{\pd x_i}(\bs x)^2\bigr)^{1/2}.
\end{align*}
This proves equation \eq{laAeq2}, and the case $k=0$, $f(0)=0$ of the theorem.

For the general case, given $f:U\ra\R$, $k=0,1,\ldots$ and $\de>0$, apply the previous case to the function $\hat f:U\ra\R$ given by
\begin{equation*}
\hat f(x_1,\ldots,x_m)=f(0)-\sum_{j=0}^k\sum_{i_1,\ldots,i_j=1\!\!\!\!\!}^m \frac{(-1)^j}{j!}\cdot x_{i_1}\cdots x_{i_j}\frac{\pd^jf}{\pd x_{i_1}\cdots\pd x_{i_j}}(x_1,\ldots,x_m).
\end{equation*}
Then $\hat f(0)=0$, and for $a=1,\ldots,m$ we have
\begin{equation*}
\frac{\pd\hat f}{\pd x_a}(x_1,\ldots,x_m)=\sum_{i_1,\ldots,i_k=1}^m
\frac{(-1)^k}{k!}\cdot x_{i_1}\cdots x_{i_k}\frac{\pd^{k+1}f}{\pd x_{i_1}\cdots\pd x_{i_k}\pd x_a}(x_1,\ldots,x_m),
\end{equation*}
so that
\begin{align*}
&(\ts\sum_{i=1}^m\frac{\pd\hat f}{\pd x_i}(x_1,\ldots,x_m)^2\bigr)^{1/2}\\
&\le(x_1^2+\cdots+x_m^2)^{k/2}
\biggl(\,\,\sum_{i_1,\ldots,i_{k+1}=1}^m
\frac{\pd^{k+1}f}{\pd x_{i_1}\cdots\pd x_{i_{k+1}}}(x_1,\ldots,x_m)^2\biggr)^{1/2}.
\end{align*}
Hence equation \eq{laAeq2} for $\hat f$ implies \eq{laAeq1} for $f$.

Finally, for the last part with $f(\bs x)=(x_1^2+\cdots+x_m^2)^{1/2}g(\bs x)$ for $g:U\ra\R$ real analytic, note that $g$ has an embedded resolution of singularities $\pi:\ti U\ra U$ such that \eq{laAeq3} holds with $g$ in place of $f$. Multiplying this by \eq{laAeq4} gives
\e
f\ci\pi\vert_{\ti V_p}=F\cdot G\cdot\ts\prod_{i\in I}y_i^{d_i}\md{y_i}^{e_i}.
\label{laAeq12}
\e
We can now follow the proof above for $f$ using \eq{laAeq12} as a substitute for \eq{laAeq3}, with $FG$ in place of $F$ and $d_i+e_i$ in place of $d_i$. Note that $fg$ is not smooth at $\bs x=0$, just as $\md{y_i}^{e_i}$ in \eq{laAeq12} is not smooth where $y_i=0$ and $e_i>0$, so we must restrict to $\bs x\ne 0$ in \eq{laAeq1}. This completes the proof.
\end{proof}

\begin{ex} Define $f:\R\ra\R$ by $f(x)=\int_0^xe^{-1/y}\sin^2(\pi/y)\d y$, so that $\frac{\d f}{\d x}=e^{-1/x}\sin^2(\pi/x)$ for $x\ne 0$ and $\frac{\d f}{\d x}(0)=0$. Then $f$ is smooth, but not real analytic near $x=0$ in $\R$. For $n=1,2,\ldots$ we see that $f(1/n)>0$ but $\frac{\d f}{\d x}(1/n)=0$. Thus equations \eq{laAeq1} for $k=0$ and \eq{laAeq2} do not hold at $(x_1,\ldots,x_m)=x_1=1/n$, so there exists no open neighbourhood $V$ of 0 in $\R$ such that \eq{laAeq1}--\eq{laAeq2} hold, as any such $V$ would contain $1/n$ for~$n\gg 0$. 
\label{laAex}
\end{ex}

\subsection{A result from Morse theory}
\label{laA2}

We now use Morse theory to show that under some assumptions we can perturb a smooth function $f:N\ra\R$ on an open set $M^\ci\subset N$ to a new function $\hat f:N\ra\R$ which is Morse on $M^\ci$ with no local maxima or minima. This was used in the proof of Proposition \ref{la4prop6} to perturb two Lagrangians to intersect transversely with controlled Maslov indices at intersections. The proof is similar to that of a result of Milnor \cite[Th~8.1]{Miln}, who shows one can cancel critical points for Morse functions $f:M\ra[0,1]$ with~$\pd M=f^{-1}(\{0,1\})$.

\begin{thm} Let\/ $N$ be a manifold of dimension $m\ge 3,$ and\/ $M$ a compact, connected manifold with boundary embedded in $N$ with\/ $\dim M=\dim N=m$. We shall denote by $M^\ci$ the interior of\/ $M,$ and by $\pd M$ the boundary of\/ $M$. Suppose $f:N\ra\R$ is smooth, such that\/ $f$ has no critical point in $\pd M,$ and\/ $f\vert_{M}$ attains its maximum and minimum only on $\pd M$.

Then there exists a smooth function $\hat f:N\ra\R,$ such that\/ $\hat f\vert_{N\sm M^\ci}=f\vert_{N\sm M^\ci},$ and\/ $\hat f\vert_{M^\ci}$ is Morse, with no critical points of index\/ $0$ or\/~$m$.
\label{laAthm2}
\end{thm}

\begin{proof} First define $\check M\subset M^\ci$ to be the subset of $x\in M^\ci$ a distance at least $\ep$ from $\pd M$, for some smooth metric on $M$ and small $\ep>0$. Then we can suppose $\check M$ is diffeomorphic to $M$, and $M\sm\check M\cong \pd M\t[0,\ep)$, and $f$ has no critical points in $M\sm\check M^\ci$, and $f\vert_{\smash{\check M}}$ attains its maximum and minimum only on~$\pd\check M$.

Next, choose smooth $\check f:N\ra\R$ such that $\check f-f:N\ra\R$ is a  $C^1$-small function on $N$ supported on $M^\ci\subset N$, and is generic under this condition. As $\check f$ is generic on $M^\ci\supset\check M$, we can suppose that $\check f$ is Morse on $M^\ci$ and $\check f\vert_{\pd\check M}$ is Morse on $\pd\check M$. Since $\check f-f$ is $C^1$-small, we can suppose that $\check f$ has no critical points in $M\sm\check M^\ci$ and $\check f\vert_{\smash{\check M}}$ attains its maximum and minimum only on $\pd\check M$, since $f$ satisfies this and it is an open condition in~$C^1$.

Let $p\in\pd\check M$ be a critical point of the Morse function $\check f\vert_{\smash{\pd\check M}}$. Since $\check f$ has no critical points on $\pd\check M$, if $w\in T_pN$ is a vector pointing inward from $\pd\check M$, then $w\cdot\d f\vert_p\ne 0$. Thus we may divide the critical points of $\check f\vert_{\smash{\pd\check M}}$ into two types. We say that $p$ is of {\it Neumann-type} if $w\cdot\d f\vert_p>0$, and of {\it Dirichlet-type} if we have $w\cdot\d f\vert_p<0$. This is independent of the choice of~$w$.

Now if $\check f\vert_{\pd\check M}$ had Neumann-type critical points of index 1, that would cause problems in the next part of the argument. So we perturb $\check f$ to eliminate such points, as follows. If $p\in\pd\check M$ is an index 1 Neumann-type critical point of $\check f\vert_{\smash{\pd\check M}}$ then we may choose local coordinates $(x_1,\ldots,x_m)$ on $N$ near $p$ such that $p=(0,\ldots,0)$, and $\check M$ is locally given by $x_1\ge 0$, so that $\pd\check M$ is locally $x_1=0$ and $(x_2,\ldots,x_m)$ are local coordinates on $\pd\check M$, and near 0 in $\R^m$ we have
\begin{equation*}
\check f(x_1,\ldots,x_m)=f(p)+x_1-x_2^2+x_3^2+\cdots+x_m^2.
\end{equation*}

We make a $C^0$-small perturbation of $\check f$ near $p$, so that near 0 it becomes
\begin{equation*}
\check f(x_1,\ldots,x_m)=f(p)+x_1\frac{x_1^2+\cdots+x_m^2-\ep^2}{x_1^2+\cdots+x_m^2+\ep^2}-x_2^2+x_3^2+\cdots+x_m^2,
\end{equation*}
for small $\ep>0$. After the perturbation, $\check f\vert_{\pd\check M}$ now has a {\it Dirichlet-type\/} critical point of index 1 at $p$, and two new Morse critical points, one of index 1 in $\check M^\ci$ and one of index 2 in $M\sm\check M$, both at distance $O(\ep)$ from $p$. As $m\ge 3$, this does not introduce any new critical points of index $0,m$. Since $p$ is not a maximum or minimum of $f\vert_{\check M}$, by making $\ep$ small we can still suppose $\check f\vert_{\smash{\check M}}$ attains its maximum and minimum only on $\pd\check M$. In this way we eliminate all index 1 Neumann-type critical point of~$\check f\vert_{\smash{\pd\check M}}$. 

We now use a version of Morse theory for compact manifolds with boundary due to Laudenbach \cite{Laud}, for the Morse function $\check f\vert_{\check M}:\check M\ra\R$. A good reference for Morse theory on manifolds without boundary is Schwarz \cite{Schw}. For each $i=0,1,\ldots,m,$ let $C_i$ be the vector space generated (over $\Z_2$ in our application) by the critical points of index $i$ for $\check f:\check M^\ci\ra\R$, and the {\it Neumann-type} critical points of index $i$ for $\check f\vert_{\smash{\pd\check M}}$ only. 

As in Laudenbach \cite[\S 2.1]{Laud}, we may choose a gradient-like vector field $v$ on $\check M$ satisfying certain properties, and define a boundary operator $\pd:C_i\ra C_{i-1}$ by counting the $v$-trajectories from generators $p$ of $C_i$ to generators $q$ of $C_{i-1}$. One can prove that $\pd^2=0$, so that $(C_*,\pd)$ is a complex, whose homology is isomorphic to $H_*(\check M;\Z_2)$. We have $H_0(\check M;\Z_2)\cong \Z_2$ as $\check M$ is connected, and the map $C_0\ra H_0(C_*,\pd)\cong H_0(\check M;\Z_2)\cong \Z_2$ is~$\sum_pc_p\,p\mapsto\sum_pc_p$.

Let $\check f\vert_{\pd\check M}$ attain its global minimum at $p\in\pd\check M$. Then $p$ is an index 0 critical point of $\check f\vert_{\pd\check M}$, and it is of Neumann-type since $\check f(p)$ is the minimum value of $\check f$ on $\check M$, so that $w\cdot\d f\vert_p\ge 0$ for $w\in T_pN$ inward-pointing. Thus $p$ is a generator of $C_0$. Suppose $\check f$ has a critical point $q$ of index $0$ in $\check M^\ci$, so that $q$ is also a generator of $C_0$. Since $H_0(C_*,\pd)\cong H_0(\check M;\Z_2)\cong \Z_2$, we see that $0\ne [p]=[q]\in H_0(C_*,\pd)$, so $p-q=\pd \al$ for some $\al\in C_1$. Hence there exists an index 1 critical point $r$ and a gradient trajectory $\ga$ from $r$ to~$q$.

As we have ensured there are no index 1 Neumann-type critical points in $\pd\check M$, $r$ must lie in $\check M^\ci$. Therefore by a theorem of Milnor \cite[Th.~5.4]{Miln}, we can modify $\check f$ in $\check M^\ci$ to cancel the critical points $q$ and $r$. By repeating this we can cancel all the critical points of index $0$ in $\check M^\ci$. To cancel all the critical points of index $m$, we apply the same procedure to $-\check f$. Thus, we can construct a smooth $\hat f:N\ra\R$ such that $\hat f-f$ is supported in $M^\ci$, and $\hat f\vert_{M^\ci}$ is Morse with no critical points of index 0 or $m$, as we have to prove.
\end{proof}

\medskip

\noindent Address for Yohsuke Imagi:

\noindent Department of Mathematics, Kyoto University, Kyoto
606-8502, Japan.

\noindent E-mail: {\tt imagi@math.kyoto-u.ac.jp}.
\medskip

\noindent Address for Dominic Joyce:

\noindent The Mathematical Institute, Radcliffe Observatory Quarter,
Woodstock Road, Oxford, OX2 6GG, U.K.

\noindent E-mail: {\tt joyce@maths.ox.ac.uk.}
\medskip

\noindent Address for Joana Oliveira dos Santos:

\noindent Department of Mechanical Engineering and Mathematical Sciences,
Faculty of Technology, Design and Environment,
Oxford Brookes University,
Oxford, 

\noindent OX33 1HX, U.K.

\noindent E-mail: {\tt joliveira-dos-santos-amorim@brookes.ac.uk.}

\end{document}